\documentclass[11pt]{article}

\usepackage[utf8]{inputenc}
\usepackage[T1]{fontenc}
\usepackage[english]{babel}
\usepackage{lmodern}

\usepackage{amsmath, amssymb, amsthm, mathtools}
\usepackage{bm}          %
\usepackage{mathrsfs}    %
\usepackage{bbm}         %
\usepackage{esint}       %
\usepackage{physics}

\usepackage{microtype}   %
\usepackage{enumitem}    %
\usepackage{xcolor}      %
\usepackage{graphicx}    %
\usepackage{tikz, tikz-cd} %
\usepackage{subcaption}

\usetikzlibrary{calc}

\usepackage{hyperref}
\hypersetup{
    colorlinks=true,
    linkcolor=blue,
    citecolor=blue,
    urlcolor=blue
}

\theoremstyle{plain}
\newtheorem{theorem}{Theorem}[section]
\newtheorem{lemma}[theorem]{Lemma}
\newtheorem{proposition}[theorem]{Proposition}
\newtheorem{corollary}[theorem]{Corollary}

\theoremstyle{definition}
\newtheorem{definition}[theorem]{Definition}

\theoremstyle{remark}
\newtheorem{remark}[theorem]{Remark}

\usepackage{cleveref}    %

\crefname{theorem}{Theorem}{Theorems}
\crefname{lemma}{Lemma}{Lemmas}
\crefname{proposition}{Proposition}{Propositions}
\crefname{corollary}{Corollary}{Corollaries}
\crefname{definition}{Definition}{Definitions}
\crefname{example}{Example}{Examples}
\crefname{remark}{Remark}{Remarks}

\newcommand{\R}{\mathbb{R}}
\newcommand{\N}{\mathbb{N}}

\newcommand{\1}{\mathbbm{1}}
\newcommand{\dist}[2]{\mathrm{dist}(#1, #2)}
\newcommand{\hausdBoundary}{\mathcal{H}^{n-1}(\partial\Omega)}
\newcommand{\hausdBoundarySet}[1]{\mathcal{H}^{n-1}(#1)}
\newcommand{\dOmega}{\mathrm{d}_\Omega}
\renewcommand{\d}[2][\Omega]{\mathrm{d}_{#1}(#2)} 

\numberwithin{equation}{section}

\usepackage{geometry}
\usepackage[backend=biber,
            style=numeric,   %
            citestyle=numeric,
            sorting=nyt,        %
            giveninits=true,    %
            maxbibnames=10,     %
            url=false,          %
            doi=true,
            isbn=false,
            eprint=true]{biblatex}

\title{Weighted heat traces of the Dirichlet Laplacian on Lipschitz domains}
\author{
Lucas Kersten\thanks{Department of Mathematical Sciences, University of Gothenburg and Chalmers University of Technology, Chalmers Tvärgata 3, 412 58 Gothenburg, Sweden.%
\newline Email: \texttt{lucaske@chalmers.se}.}
}
\date{}

\begin{document}

\maketitle
\vspace{-1em}

\begin{abstract}
We prove an asymptotic expansion of a weighted heat trace of the Dirichlet Laplacian on a bounded Lipschitz domain. The weight is given by a power of the distance to the boundary times a bounded function that is continuous near the boundary. Depending on the exponent of the weight, we obtain one-term or two-term asymptotics. We present two applications, namely a version of one-term asymptotics of weighted Riesz means for arbitrary orders, as well as convergence results for a family of measures that describes localization of Laplace eigenfunctions.

\end{abstract}

\tableofcontents

\section{Introduction}
\label{sec:introduction}

The goal of this paper is to analyze the boundary behavior of the Dirichlet heat kernel on Lipschitz domains. To do so, we find the short-time asymptotics of a weighted integral over the diagonal part of the heat kernel. We refer to this quantity as a weighted heat trace. The weight is chosen such that, depending on a parameter, the influence of the boundary increases or decreases.

Let $n\in\N$ and $\Omega \subset \R^n$ be an open and bounded set with Lipschitz boundary, that is, the boundary $\partial\Omega$ can be locally represented by the graph of a Lipschitz function. The Laplacian $-\Delta$, defined by the usual quadratic form with form domain $H_0^1(\Omega)$, is then called the Dirichlet Laplacian; see \cite[Section 3.1]{MR4496335} for a detailed introduction. It is well-known that its spectrum consists purely of positive eigenvalues $\{\lambda_k\}_{k\in\N}$ such that
\begin{align*}
    0 < \lambda_1 \leq \lambda_2 \leq \lambda_3 \leq \dots \nearrow \infty,
\end{align*}
and that the eigenvalues depend on the domain $\Omega$.

For $f\in L^2(\Omega)$, the solution $u(x;t)$ of the initial value problem
\begin{align}
\begin{cases}
    (\partial_t - \Delta) u(x;t) = 0, & \forall (x,t)\in \Omega\times(0,\infty) \\
    u(x;t) = 0, & \forall (x,t)\in \partial\Omega\times(0,\infty) \\
    u(x;0) = f(x), & \forall x \in \Omega
\end{cases}
\label{eq:def_HeatEq}
\end{align}
can be formally written as $u(x;t) = e^{t\Delta}f(x)$. This can be represented as an integral operator 
\begin{align*}
    u(x;t) = e^{t\Delta}f(x) = \int_\Omega p_\Omega(x,y;t)f(y) \dd{y},
\end{align*}
with its integral kernel $p_\Omega(x,y;t)$. We call $p_\Omega(x,y;t)$ the Dirichlet heat kernel on $\Omega$. By an expansion into a real orthonormal eigenbasis $\{u_k\}_{k\in\N}$, where each basis function solves $-\Delta u_k = \lambda_k u_k$, we find
\begin{align}
    p_\Omega(x,y;t) = \sum_{k=1}^\infty e^{-t\lambda_k} u_k(x) u_k(y). \label{eq:heatKernelRepInBasis}
\end{align}
Therefore, $p_\Omega$ encodes information about the spectrum $\{\lambda_k\}_{k\in\N}$ and the eigenfunctions $\{u_k\}_{k\in\N}$ simultaneously. Furthermore, the trace of the operator $e^{t\Delta}$ is given by
\begin{align}
    \mathrm{Tr}\Big(e^{t\Delta}\Big) = \int_\Omega p_\Omega(x,x;t)\dd{x} = \sum_{k=1}^\infty e^{-\lambda_k t}. \label{eq:heatTrace}
\end{align}
This trace can be recovered as the Laplace transform of the eigenvalue counting function or its Riesz means. A brief overview of these will be given later in this introduction. For a more thorough treatment, see for example \cite[Chapter 9.1]{MR4655924}.

Let
\begin{align*}
    \d{x} \coloneq \inf_{y\in\Omega^c} \abs{x-y},
\end{align*}
which is the distance of the point $x\in\Omega$ to the boundary of $\Omega$. It is well known that, for bounded $\Omega$ with smooth boundary, $p_\Omega(x,x;t) \asymp_{\Omega,t} \d{x}^2$ as $x\to \partial\Omega$ for $t>0$. Such a quadratic asymptotic is not necessarily true anymore for domains with merely Lipschitz boundary where the kernel might decay to $0$ at a slower rate as $x\to \partial\Omega$.
To analyze this, among other things, we are interested in the existence of an asymptotic expansion of the weighted heat trace
\begin{align}
    \int_\Omega f(x) p_{\Omega}(x,x;t)\dd{x} \label{eq:weightedHeatTrace}
\end{align}
as $t\to 0$ for suitable, possibly singular weight functions $f$. The similarity of \eqref{eq:weightedHeatTrace} to the integral expression in \eqref{eq:heatTrace} motivates the name weighted heat trace. In the following, we will choose $f(x) = g(x)\cdot \d{x}^\alpha$ for specific regular classes of functions $g$ and as wide a range of values of $\alpha$ as possible. We denote
\begin{align*}
    Z_{g,\alpha}(t) \coloneq (4\pi t)^\frac{n}{2}\int_\Omega g(x) \d{x}^\alpha p_\Omega(x,x;t)\dd{x}
\end{align*}
for the weighted heat trace. If $g=1$, we set $Z_{\alpha} = Z_{1,\alpha}$. Note that $Z_{g,\alpha}(t)$ might be infinite. In particular, for any sufficiently regular $\Omega$, this is the case for all $t > 0$ when $g = 1$ and $\alpha \leq -3$. However, for general Lipschitz sets, our method requires us to further restrict the range of possible values for $\alpha$ in order to ensure suitable control of the heat kernel near the boundary. The restriction depends on the value of the weak Hardy constant, that is, the infimum over all $c>0$ for which there exists a $\Lambda \geq 0$ satisfying
\begin{align*}
    \int_\Omega \frac{f(x)^2}{\d{x}^2}\dd{x} \leq c^2\int_\Omega \abs{\nabla f(x)}^2\dd{x} + \Lambda \int_\Omega f(x)^2\dd{x} \textrm{ for all } f\in C_c^\infty(\Omega).
\end{align*}
For domains with Lipschitz boundary, the weak Hardy constant is positive and finite, see Remark~\ref{rem:weakHardyConstForLipschitz}. A more detailed discussion is deferred to Definition \ref{def:weakHardyInequality} and the remarks thereafter.

Let $c_w$ denote the weak Hardy constant of $\Omega$, and define $\alpha^\ast(c_w) \coloneq \max\left(-2\big(1+\frac{1}{c_w}\big), -3\right)$. This value gives a lower bound on the possible exponents that we can treat by our estimates. The following is the first main result of the paper.

\begin{theorem}
\label{thm:mainThmWithG}
    Let $n\in \N$, $n\geq 2$, and $\Omega\subset\R^n$ be an open and bounded set with Lipschitz boundary. Suppose that $g\in L^\infty(\overline{\Omega})$ is continuous in a neighborhood of $\partial\Omega$. Let $c_w$ be the weak Hardy constant of $\Omega$ and $ \alpha > \alpha^\ast(c_w)$. Then $Z_{g,\alpha}(t)$ exists for all $t>0$ and satisfies the following asymptotics, as $t\to 0$,
        \begin{align*}
             Z_{g,\alpha}(t) = \begin{cases}
                \begin{aligned}[c]
                    \int_{\Omega} g(x)\d{x}^{\alpha}\dd{x}- t^{\frac{1+\alpha}{2}} \frac{1}{2}\Gamma\bigg(\frac{1+\alpha}{2}\bigg) \int_{\partial\Omega}g(x)\dd{\mathcal{H}^{n-1}(x)} + o\Big(t^\frac{1+\alpha}{2}\Big),
                \end{aligned}
                    &\qif \alpha > -1, \\[1.25em]
                \begin{aligned}
                     -\log(t)\cdot   \frac{1}{2} \int_{\partial\Omega} g(x)\dd{\mathcal{H}^{n-1}(x)}  + o\Big(\abs{\log t}\Big) ,
                \end{aligned}
                       &\qif \alpha = -1, \\[1.25em]
                \begin{aligned}[c]
                            -t^{\frac{1+\alpha}{2}} \frac{1}{2}\Gamma\bigg(\frac{1+\alpha}{2}\bigg) \int_{\partial\Omega} g(x) \dd{\mathcal{H}^{n-1}(x)} + o\Big(t^{\frac{1+\alpha}{2}}\Big)  ,
                \end{aligned}
                        &\qif \alpha < -1.
            \end{cases}
        \end{align*}
\end{theorem}

Some remarks for this result are in order.
\begin{remark}
    For any integer $1\leq k\leq n$, we denote by $\mathcal{H}^{k}$ the $k$-dimensional Hausdorff measure, using the normalization in \cite[Definition 2.46]{MR1857292}. This normalization has the advantage that the $n$-dimensional and the $(n-1)$-dimensional Hausdorff measures coincide with the Lebesgue measure and the usual surface measure, respectively.
\end{remark}
\begin{remark}
    Choosing $g(x) = 1$ for $x\in\overline{\Omega}$, the boundary integral in the preceding theorem simplifies to
    \begin{align*}
        \int_{\partial\Omega} g(x)\dd{\mathcal{H}^{n-1}(x)} = \hausdBoundary,
    \end{align*}
    which means that the subleading term for $\alpha > -1$ and the leading term for $\alpha \leq -1$ of this expansion are proportional to the perimeter of $\Omega$.
\end{remark}

\begin{remark}
    If $\alpha = 0$ and $g = 1$, hence $f=1$, we have the simplification
    \begin{align*}
        \int_\Omega g(x) \d{x}^\alpha \dd{x} = \abs{\Omega},
    \end{align*}
    which is the usual volume term that arises in Weyl's law and asymptotic expansions of Riesz means. Thus, we obtain
    \begin{align*}
        Z_{1,0}(t) = \abs{\Omega} -\frac{\sqrt{\pi t}}{2} \hausdBoundary + o\Big(t^\frac12\Big), \qas t\to 0,
    \end{align*}
    which recovers the result of Brossard and Carmona \cite{MR834484} for domains with $C^1$ boundary and its extension by Brown \cite{MR1134755} to domains with Lipschitz boundary.
\end{remark}

\begin{remark}
    For bounded $\Omega$ with smooth boundary, our result reproves a particular case of the results of van den Berg, Gilkey, Kirsten and Seeley in \cite{MR2550208}. Namely, the authors show the existence of an infinite asymptotic expansion of a weighted heat trace as $t\to 0$ if $\partial\Omega$ is smooth. The constants in their asymptotic expansion agree with the constants given in our main result. A subsequent paper~\cite{MR2836585} by van den Berg, Gilkey and Kirsten generalizes the result for Dirichlet boundary conditions to mixed boundary conditions.
\end{remark}
\begin{remark}
    In the singular case $\alpha < -1$, there exists an order-sharp upper bound, proved by van den Berg in \cite{MR2480958}, for domains with finite volume and under the assumption that a strong Hardy inequality holds. Theorem \ref{thm:mainThmWithG} strengthens this, for bounded Lipschitz domains, to a leading-order asymptotic formula. Van den Berg's estimates for controlling the heat kernel are used to control its behavior near the boundary and constitute an important ingredient in the proof of Theorem~\ref{thm:mainThmWithG}. The restriction on the possible values of $\alpha$ in \cite{MR2480958} depends on the strong Hardy constant. By adapting his argument, we obtain a possibly larger range of admissible exponents $\alpha$, which instead depends on the weak Hardy constant.
\end{remark}

The idea of the proof of Theorem \ref{thm:mainThmWithG} is to write the integral $Z_{g,\alpha}(t)$ as
\begin{align}
\begin{aligned}
    Z_{g,\alpha}(t) &= \int_\Omega  g(x)\d{x}^{\alpha} \bigg(1-e^{-\frac{\d{x}^2}{t}}\bigg)\dd{x} \\
    &\hspace{4em}+ \int_\Omega g(x)\d{x}^{\alpha}\bigg((4\pi t)^{n/2}p_\Omega(x,x;t) - 1+e^{-\frac{\d{x}^2}{t}}\bigg)\dd{x}.
\end{aligned}
    \label{eq:decompInMainContrAndError}
\end{align}
We will show that the first integral is the main contribution for $\alpha < -1$, which turns out to be of order $t^{\frac{1+\alpha}{2}}$, and the second integral is a lower-order correction. That the second integral is of lower-order corresponds to the heat kernel of $\Omega$ locally being well-approximated by the heat kernel of suitably chosen half-spaces.
 
For $\alpha > -1$, we further split
\begin{align}
    \int_\Omega  g(x)\d{x}^{\alpha} \bigg(1-e^{-\frac{\d{x}^2}{t}}\bigg)\dd{x} = \int_\Omega g(x) \d{x}^{\alpha}\dd{x} - \int_\Omega g(x) \d{x}^{\alpha} e^{-\frac{\d{x}^2}{t}}\dd{x} \label{eq:splittingMainTerms}
\end{align}
and will show that the first term is finite and the second term has the attributed asymptotics. For these terms, the continuity of $g$ near the boundary is vital.

To analyze the second term in \eqref{eq:decompInMainContrAndError} we decompose $\Omega$, for all $\alpha > \alpha^\ast(c_w)$, into the following four regions and estimate the integral in these regions separately:
\begin{itemize}
    \item A bulk region, where $\d{x} \gg \sqrt{t}$.
    \item A good part of the boundary region, where $\d{x} \asymp \sqrt{t}$.
    \item A bad part of the boundary region, where $\d{x} \asymp \sqrt{t}$.
    \item A near-boundary region, where $\d{x} \ll \sqrt{t}$.
\end{itemize}

The good part of the boundary region will be constructed so that we can squeeze the corresponding part of the boundary locally between two hyperplanes. The heat kernel in that region will then be approximated by heat kernels of half-spaces corresponding to the two hyperplanes. This will result in the first term in Theorem \ref{thm:mainThmWithG} if $\alpha \leq-1$ and the second term for $\alpha > -1$. In the bulk region, the heat kernel is expected to behave as the free heat kernel. This follows from Kac's principle of not feeling the boundary, see Kac's published notes \cite{MR201237}. This is the first term for $\alpha > -1$. For $\alpha \leq -1$, the contribution of the bulk will be negligible.

The bad part is then given by all the points in the boundary region that do not belong to the good part. We will show that the measure of the bad part can be chosen small enough such that its contribution to the weighted heat trace will be insignificant.
Lastly, the heat kernel in the near-boundary region is small enough so that it does not contribute. We will ensure this by applying a weak Hardy inequality and relevant upper bounds for the heat kernel similar to what is done in \cite{MR2480958}. %

As we shall see in the upcoming sections, the continuity of $g$ will be used to derive the main terms in the asymptotic expansion, and the boundedness of $g$ to ensure that $g$ does not change the asymptotic order of the lower-order correction. Far away from the boundary, it is even possible to have a less regular weight function. There we only need an integrability condition for $g$, because the heat kernel in that part is already well approximated by the free heat kernel. This is summarized in the following extension.

\begin{corollary}
    \label{cor:introRelaxationOng}
    Let $n\in\N$, $n\geq 2$, and $\Omega \subset \R^n$ be an open and bounded set with Lipschitz boundary. Suppose $g\colon \overline{\Omega}\to \R$ can be written as $g = g_1 +g_2$ where $g_1$ satisfies the assumptions of Theorem~\ref{thm:mainThmWithG}, $g_2\in L^1(\overline{\Omega})$, and $\mathrm{supp}(g_2)$ is compactly contained in $\Omega$. Furthermore, let $c_w$ be the weak Hardy constant of $\Omega$ and $\alpha > \alpha^\ast(c_w)$. Then $Z_{g,\alpha}(t)$ exists for all $t>0$, and the same asymptotics as in Theorem \ref{thm:mainThmWithG} hold.
\end{corollary}

A separate class of functions $g$, which is relevant in the proof of Theorem \ref{thm:mainThmWithG}, consists of indicator functions of certain sets $A\subset\R^n$, that is, $g=\1_A$. We then write $Z_{A,\alpha} = Z_{g,\alpha}$. As this yields another interesting localization result, we state it as its own theorem. 

\begin{theorem}
\label{thm:mainThmWithIndicatorFunc}
    Let $n\in\N$, $n\geq 2$, $\Omega\subset\R^n$ be an open and bounded set with Lipschitz boundary, and $A\subset\R^n$ be open and bounded such that $\hausdBoundarySet{\partial\Omega \cap \partial A} = 0$. Let $c_w$ be the weak Hardy constant of $\Omega$ and $\alpha > \alpha^\ast(c_w)$. Then $Z_{A,\alpha}(t)$ exists for all $t > 0$ and satisfies the following asymptotics, as $t\to 0$, 
        \begin{align*}
             Z_{A,\alpha}(t) = \begin{cases}
                \begin{aligned}[c]
                    \int_{\Omega\cap A} \d{x}^{\alpha}\dd{x}- t^{\frac{1+\alpha}{2}} \frac{1}{2}\Gamma\bigg(\frac{1+\alpha}{2}\bigg) \hausdBoundarySet{\partial\Omega \cap A} + o\Big(t^\frac{1+\alpha}{2}\Big),
                \end{aligned}
                    &\qif \alpha > -1, \\[1.25em]
                \begin{aligned}
                     -\log (t)\cdot  \frac{1}{2} \hausdBoundarySet{\partial\Omega \cap A}  + o\Big(\abs{\log t}\Big) ,
                \end{aligned}
                       &\qif \alpha = -1, \\[1.25em]
                \begin{aligned}[c]
                            -t^{\frac{1+\alpha}{2}} \frac{1}{2}\Gamma\bigg(\frac{1+\alpha}{2}\bigg) \hausdBoundarySet{\partial\Omega \cap A} + o\Big(t^{\frac{1+\alpha}{2}}\Big)  ,
                \end{aligned}
                        &\qif \alpha < -1.
            \end{cases}
        \end{align*}
\end{theorem}

\begin{remark}
    The geometric condition on the set $A$, namely $\hausdBoundarySet{\partial\Omega \cap \partial A} = 0$, arises due to technical requirements needed to establish the convergence of certain integral terms to the localized Hausdorff measures. We will see in Remark \ref{rem:exForGeometricCondA}, that some kind of geometric condition is necessary. Nevertheless, in Remark \ref{rem:exForGeometricCondB} we present an example that illustrates that our geometric condition might not be optimal.
\end{remark}

Two applications of our main result are now presented. In the unweighted case, the short-time asymptotics of the heat trace are related to semiclassical asymptotics of the counting function and so-called Riesz means of order $\gamma\in [0,\infty)$, defined as
\begin{align*}
    N(\lambda) &\coloneq \#\{k\in\N \mid \lambda_k < \lambda\},\qand 
    \\
    \Tr(-\Delta-\lambda)_-^{\gamma} &\coloneq \sum_{k=1}^\infty (\lambda_k-\lambda)^\gamma_- = \sum_{k \colon \lambda_k < \lambda} (\lambda-\lambda_k)^\gamma,
\end{align*}
respectively. Here we use the convention that $x_\pm = \frac12(\abs{x}\pm x)$ and $x_-^0 = \1_{x<0}$, therefore the counting function can be interpreted as a Riesz mean with parameter $\gamma = 0$. As explained earlier, the heat trace and Riesz means are related through the Laplace transform. Karamata's Tauberian Theorem then relates the leading term in the short-time asymptotic expansion of the heat trace to the leading term in the semiclassical asymptotic expansion of Riesz means. For an overview of Tauberian Theorems we refer to \cite{MR2073637}.

Using our result for the weighted heat trace and Karamata's Tauberian Theorem, we find the following result for a weighted Riesz mean.

\begin{corollary}
\label{cor:weightedRieszMean}
    Suppose that the assumptions of Theorem \ref{thm:mainThmWithG} are satisfied. Let $\gamma \in [0,\infty)$. Then
    \begin{align*}
        R_{g,\alpha}^\gamma (\lambda) \coloneq \int_\Omega g(x)\d{x}^\alpha (-\Delta-\lambda)_-^\gamma(x,x) \dd{x}
    \end{align*}
    exists for all $\lambda > 0$. Furthermore, it satisfies the following asymptotics, as $\lambda\to \infty$,
    \begin{align*}
        R_{g,\alpha}^\gamma(\lambda) = \begin{cases}
                \begin{aligned}[c]
                    \lambda^{\gamma + \frac{n}{2}} \frac{\Gamma(\gamma+1)}{(4\pi)^{\frac{n}{2}} \Gamma(\gamma+1+\frac{n}{2})} \int_\Omega g(x) \d{x}^\alpha \dd{x} + o\Big(\lambda^{\gamma+\frac{n}{2}}\Big),
                \end{aligned}
                    &\qif \alpha > -1, \\[1.25em]
                \begin{aligned}
                     \lambda^{\gamma+\frac{n}{2}}\log (\lambda)  \frac{\Gamma(\gamma+1)}{2(4\pi)^\frac{n}{2} \Gamma(\gamma+1+\frac{n}{2})} \int_{\partial\Omega}g(x)\dd{\mathcal{H}^{n-1}(x)}  + o\Big(\lambda^{\gamma+\frac{n}{2}}\log \lambda\Big) ,
                \end{aligned}
                       &\qif \alpha = -1, \\[1.25em]
                \begin{aligned}[c]
                            -\lambda^{\gamma+\frac{n-1-\alpha}{2}}  \frac{\Gamma(\gamma+1)\Gamma(\frac{1+\alpha}{2})}{2(4\pi)^{\frac{n}{2}}\Gamma(\gamma+\frac{1+n-\alpha}{2})} \int_{\partial\Omega}g(x)\dd{\mathcal{H}^{n-1}(x)} + o\Big(\lambda^{\gamma+\frac{n-1-\alpha}{2}}\Big)  ,
                \end{aligned}
                        &\qif \alpha < -1.
            \end{cases}
    \end{align*}

\end{corollary}
\begin{remark}
    We write $(-\Delta-\lambda)_-^\gamma(x,y)$ for the integral kernel of the operator $(-\Delta-\lambda)_-^\gamma$, which is given in an eigenbasis $\{u_k\}_{k\in\N}$ of $-\Delta$ by
    \begin{align*}
        (-\Delta-\lambda)_-^\gamma(x,y) = \sum_{k=1}^\infty u_k(x)u_k(y) (\lambda_k-\lambda)_-^\gamma  .
    \end{align*}
    For fixed $\lambda$ this is just a finite sum, so there are no convergence issues. Therefore, the weighted Riesz mean can equivalently be formulated as
    \begin{align*}
        \int_\Omega g(x)\d{x}^\alpha (-\Delta-\lambda)_-^\gamma(x,x) \dd{x} &= \sum_{k=1}^\infty \int_\Omega g(x) \d{x}^\alpha u_k(x)^2 (\lambda_k-\lambda)^\gamma_- \dd{x} \\
        &= \sum_{k\colon \lambda_k < \lambda} (\lambda-\lambda_k)^\gamma  \int_\Omega g(x)\d{x}^\alpha u_k(x)^2 \dd{x}.
    \end{align*}
\end{remark}

\begin{remark}
    In the case of $\alpha = 0$ and $g=1$, we reproduce Weyl's law for Riesz means,
    \begin{align*}
        \Tr(-\Delta-\lambda)^\gamma_- = \int_\Omega (-\Delta-\lambda)_-^\gamma(x,x) \dd{x} = \lambda^{\gamma + \frac{n}{2}} \frac{\Gamma(\gamma+1)}{(4\pi)^{\frac{n}{2}} \Gamma(\gamma+1+\frac{n}{2})} \abs{\Omega}+ o\Big(\lambda^{\gamma+\frac{n}{2}}\Big).
    \end{align*}
\end{remark}

As a second application, we observe that the short-time asymptotics may be interpreted as convergence of measures. Recall that, for a family of (signed) measures $(\mu_t)_{t\in(0,\infty)}$, we say that the family converges weakly to the (signed) measure $\mu_0$ if 
\begin{align*}
    \int_{\R^n} g(x)\dd{\mu_t(x)} \overset{t\to 0}{\longrightarrow} \int_{\R^n} g(x) \dd{\mu_0(x)}, \text{ for all } g\in C_c(\R^n).
\end{align*}
Therefore, we deduce the following corollary of the main theorem.
\begin{corollary}
\label{cor:ConvergenceOfMeasures}
    Let $n\in \N$, $n\geq 2$, and $\Omega \subset \R^n$ be an open and bounded set with Lipschitz boundary. Let $c_w$ be the weak Hardy constant for $\Omega$ and $\alpha > \alpha^\ast(c_w)$. Then the family of measures
    \begin{align*}
        \dd\mu_t^\alpha(x) = \begin{cases}
         t^{\frac{n}{2}} p_\Omega(x,x;t) \d{x}^\alpha \1_\Omega\dd{x}, & \qif \alpha > -1, \\[1.25em]
        - t^{\frac{n}{2}} (\log t)^{-1} p_\Omega(x,x;t) \d{x}^\alpha \1_\Omega\dd{x}, & \qif \alpha = -1, \\[1.25em]
         t^{\frac{n-\alpha-1}{2}}p_\Omega(x,x;t) \d{x}^\alpha \1_\Omega\dd{x}, & \qif \alpha < -1,
        \end{cases}   
    \end{align*}
    converges weakly as $t\to 0$ to 
    \begin{align*}
        \mu_0^\alpha = \begin{cases}
            (4\pi)^{-\frac{n}{2}} \dOmega^\alpha \lambda\rvert_\Omega ,&\qif \alpha > -1, \\[1.25em]
            \frac{1}{2(4\pi)^{\frac{n}{2}}} \mathcal{H}^{n-1}\rvert_{\partial\Omega} ,&\qif \alpha = -1, \\[1.25em]
            -\frac{1}{2(4\pi)^{\frac{n}{2}}} \Gamma\big(\frac{1+\alpha}{2}\big) \mathcal{H}^{n-1}\rvert_{\partial\Omega}, &\qif \alpha < -1 ,
        \end{cases}
    \end{align*}
    respectively, where $\lambda$ is the Lebesgue measure.

    Furthermore, for $\alpha > -1$, the family 
    \begin{align*}
        \dd\Tilde{\mu}_t^\alpha(x) = t^{-\frac{\alpha+1}{2}} \bigg(t^{\frac{n}{2}}p_\Omega(x,x;t) - (4\pi)^{-\frac{n}{2}}\bigg)\d{x}^\alpha \1_\Omega \dd{x}
    \end{align*}
    of signed measures converges weakly as $t \to 0$ to
    \begin{align*}
        \Tilde{\mu}_0^\alpha = - \frac{1}{2(4\pi)^\frac{n}{2}} \Gamma\bigg(\frac{1+\alpha}{2}\bigg) \mathcal{H}^{n-1}\rvert_{\partial\Omega}.
    \end{align*}
\end{corollary}

The statement that the family of measures converges for $\alpha < -1$ to a multiple of the boundary surface measure can be interpreted as a localization result for the eigenfunctions of the Laplacian. By means of the representation \eqref{eq:heatKernelRepInBasis}, the measures $\mu_t^\alpha$ are an average of the $L^2$-mass densities $u_j(x)^2$ of the eigenfunctions with eigenvalues until the order $\lambda_j \sim t^{-1}$. The singular weight $\dOmega^\alpha$ then yields that the average is strongly localized near the boundary, so that only the boundary measure $\mathcal{H}^{n-1}\rvert_{\partial\Omega}$ remains in the limit $t\to 0$.

\begin{remark}
    A similar statement about convergence of measures can be obtained by using the asymptotics for weighted Riesz means in Corollary \ref{cor:weightedRieszMean}. Interpreted as weak convergence of measures, this yields, for $\alpha > -1$,
\begin{align*}
    \lambda^{-\gamma-\frac{n}{2}}\d{x}^\alpha (-\Delta-\lambda)_-^\gamma(x,x)\1_\Omega\dd{x} \rightharpoonup \frac{\Gamma(\gamma+1)}{(4\pi)^\frac{n}{2}\Gamma(\gamma+1+\frac{n}{2})} \d{x}^\alpha\1_\Omega\dd{x}    ,
\end{align*}
and for $\alpha \in (\alpha^\ast(c_w),-1)$,
\begin{align*}
     \lambda^{-\gamma-\frac{n-1-\alpha}{2}}\d{x}^\alpha (-\Delta-\lambda)_-^\gamma(x,x)\1_\Omega\dd{x} \rightharpoonup -\frac{\Gamma(\gamma+1)\Gamma(\frac{1+\alpha}{2})}{2(4\pi)^\frac{n}{2}\Gamma(\gamma+\frac{1+n-\alpha}{2})} \mathcal{H}^{n-1}\rvert_{\partial\Omega},
\end{align*}
as $\lambda \to \infty$.
The latter generalizes the result of Frank and Larson in \cite[Corollary 1.2]{MR4331812} to domains with Lipschitz boundary, general Riesz parameters $\gamma \geq 0$ and a wider range of singularities with exponent $\alpha > \alpha^\ast(c_w)$. We recover their result by choosing $\alpha = -2$ and $\gamma = 0$. 
\end{remark}

\subsection{Notation}
In the following, we fix a dimension $n\in\N$ with $n\geq 2$. For $x\in\R^n$ and $r>0$ we denote by $B_r(x)$ the ball of radius $r$ centered at $x$, that is, $$B_r(x) = \{y\in\R^n \mid \abs{x-y} < r\}.$$ For a set $A \subset \R^n$, $\abs{A}$ denotes the $n$-dimensional Lebesgue measure of $A$ and $\hausdBoundarySet{A}$ the $(n-1)$-dimensional Hausdorff measure of $A$. 

For $r> 0$, we abbreviate the set $\{x\in\Omega \mid \d{x} < r\}$ as $\{\d{x} < r\}$, if there is no confusion.

For non-negative functions $f$ and $g$, we write $f\lesssim g$ if there exists a constant $C > 0$ such that $f(x) \leq C g(x)$ for all $x$. The constant may depend on the relevant parameters. If we want to make the dependence of the constant $C$ on a parameter $c$ explicit, we write $f \lesssim_c g$. We write $f \asymp g$ if $f \lesssim g$ and $g \lesssim f$.

\subsection{Structure of the paper}
In Section \ref{sec:geomConstruction}, we first present various geometric constructions that are necessary for our estimates. These include the geometric construction of good points near the boundary and an introduction to inner and outer Minkowski contents, which yield a framework to analyze integrals over inner tubular regions around the boundary of Lipschitz sets. In Section \ref{sec:mainContribution} we find expressions for the main contribution according to the decomposition \eqref{eq:decompInMainContrAndError}, using the analysis of integrals over inner tubular regions developed before. Next, in Section \ref{sec:heatKernelConstructions}, we recall known identities and tools that are useful in analyzing heat kernels and prove an approximation result for the heat kernel at good points near the boundary. Subsequently, in Section \ref{sec:boundsForErrorTerms}, we analyze the second term in the decomposition \eqref{eq:decompInMainContrAndError} for different regions of the domain $\Omega$, as sketched before. With these calculations and estimates, we are able to prove our main theorems and their corollaries in Section \ref{sec:proofSection}.

\section{Geometric constructions for Lipschitz sets}
\label{sec:geomConstruction}

\subsection{Construction of good points near the boundary}
In the following, we will introduce the geometric construction of good points used by Brown in~\cite{MR1134755}. We follow the development of Frank and Larson in \cite{frank2026uniformboundsneumannheat}. Loosely speaking, good points on the boundary are points $p$ at which the boundary in the vicinity of $p$ can be enclosed in the complement of certain cones. Furthermore, the good points inside the set are points that lie in a narrower cone around the same symmetry axis pointing in the inward normal direction of a good boundary point.

\begin{definition}
    Let $\varepsilon \in \big(0,\frac12\big]$ and $r > 0$. A point $p\in\partial\Omega$ is called $(\varepsilon, r)$-good if the inner unit normal $\nu(p)$ exists at $p$ and
    \begin{align*}
        \partial\Omega\cap B_r(p) \subset \{x\in\R^n\mid \abs{(p-x)\cdot\nu(p)}< \varepsilon\abs{p-x}\}\cup\{p\}.
    \end{align*}
    A set $G \subset \partial\Omega$ is called $(\varepsilon, r)$-good if each point $p\in G$ is $(\varepsilon,r)$-good.
\end{definition}
This means that the boundary points in a ball of radius $r$ around $p$ are enclosed in the complement of cones with opening angle $\varphi = \arccos{\varepsilon}$ and symmetry axis $\nu(p)$.
\begin{definition}
\label{def:goodPointsOfDomain}
    For all $(\varepsilon,r)$-good $p\in \partial\Omega$, let 
    \begin{align*}
        \Gamma_{\varepsilon,r}(p) \coloneq \{x\in\Omega \mid (x-p)\cdot\nu(p)>\sqrt{1-\varepsilon^2}\abs{p-x}\}\cap B_{r/2}(p),
    \end{align*}
    and, if $G\subset \partial\Omega$ is $(\varepsilon,r)$-good, define
    \begin{align*}
        \mathcal{G} \coloneq \bigcup_{p\in G} \Gamma_{\varepsilon,r}(p).
    \end{align*}
    We call $\mathcal{G}$ the set of $(\varepsilon,r)$-good points of $\Omega$ associated to $G$ or the sawtooth-region associated to $G$. We may also just say good points if the dependence on $(\varepsilon,r)$ is clear.
\end{definition}

The geometrical interpretation is as follows: For a $(\varepsilon,r)$-good point $p\in \partial\Omega$, $\Gamma_{\varepsilon,r}(p)$ is a truncated cone with opening angle $\arccos{\sqrt{1-\varepsilon^2}} = \arcsin{\varepsilon}$ and the normal $\nu(p)$ as symmetry axis. We say that any point $x\in \Gamma_{\varepsilon,r}(p)$ is a good point near $p\in\partial\Omega$. Conversely, if $\mathcal{G} \subset \Omega$ is the $(\varepsilon,r)$-good subset associated to $G\subset \partial\Omega$, then for any good point $x\in \mathcal{G}$, there is a, not necessarily unique, associated good point $p\in G$ on the boundary so that $x\in\Gamma_{\varepsilon,r}(p)$. A visualization of the situation is given in Figure \ref{fig:basicConstrInclSawtooth}.

\begin{figure}[htbp]
    \centering
    \begin{tikzpicture}[
    xscale=1.5,
    yscale=1.5,
    boundary/.style={thick},
    rough/.style={thick},
    ray/.style={->, thick},
    guide/.style={dashed, thin},
    point/.style={circle, fill, inner sep=1.2pt}
]

\coordinate (p) at (0,0);
\def\R{2.9}
\def\a{27.5}
\def\L{3.3}

\fill[gray!25]
    (p) -- ++({-90-\a}:{\R/2})
    arc[start angle={-90-\a}, end angle={-90+\a}, radius={\R/2}]
    -- cycle;

\draw[boundary] (p) circle (\R);

\draw[line width=0.4pt] (p) circle ({\R/2});

\draw[rough]
    (-2.92, 1.78)
    -- (-2.74, 1.50)
    -- (-2.56, 1.12)
    -- (-2.35, 0.35)
    -- (-2.18,-0.72)
    -- (-1.88,-0.18)
    -- (-1.62, 0.61)
    -- (-1.31,-0.46)
    -- (-1.04,-0.08)
    -- (-0.79, 0.28)
    -- (-0.55,-0.22)
    -- (-0.31, 0.10)
    -- (-0.20, 0.00)
    -- (p)
    -- (0.18, 0.00)
    -- (0.37,-0.08)
    -- (0.63, 0.29)
    -- (0.96,-0.41)
    -- (1.18,-0.18)
    -- (1.46, 0.52)
    -- (1.76,-0.23)
    -- (2.02,-0.72)
    -- (2.24, 0.45)
    -- (2.43,-0.92)
    -- (2.58,-1.28)
    -- (2.70,-1.53)
    -- (2.80,-1.90);

\node at (-1.8,-2.0) {$\Omega$};
\node at (-2.95,1.2) {$\partial\Omega$};

\node[point] at (p) {};
\node[above=0.6pt] at (p) {$p$};

\coordinate (q) at ($(p)+(40:\R)$);
\draw[ray] (p) -- (q) node[pos=0.68, above left] {$r$};

\coordinate (qhalf) at ($(p)+(47:{\R/2})$);
\draw[ray] (p) -- (qhalf) node[pos=0.72, above left] {$r/2$};

\draw[ray] (p) -- ++(0,-0.85) node[below] {$\nu(p)$};

\draw[guide] ($(p)+({180+\a}:\L)$) -- ($(p)+(\a:\L)$);
\draw[guide] ($(p)+({180-\a}:\L)$) -- ($(p)+({-\a}:\L)$);

\draw[dashed] (p) ++(\a:0.88) arc[start angle=\a, end angle={180-\a}, radius=0.88];
\node at ($(p)+(100:0.55)$) {$\varphi$};

\draw[guide] (p) -- ++({-90-\a}:{\R/2});
\draw[guide] (p) -- ++({-90+\a}:{\R/2});

\end{tikzpicture}
    \caption{Sketch of the geometric construction of a $(\varepsilon,r)$-good point $p\in\partial\Omega$. The corresponding sawtooth $\Gamma_{\varepsilon,r}(p)$ is shaded. The opening angle of the conical regions defining good boundary points is given by $\varphi = \arccos\varepsilon$.}
    \label{fig:basicConstrInclSawtooth}
\end{figure}
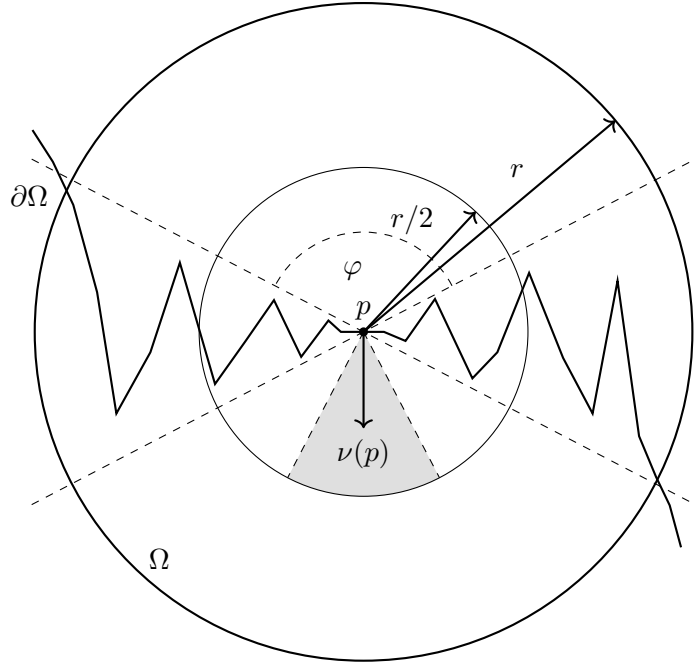

Let $p\in\partial\Omega$ be $(\varepsilon,r)$-good and $x\in \Gamma_{\varepsilon,r}(p)$. By continuity of the distance function, there exists $p^\ast \in \partial\Omega$ such that $\d{x} = \abs{x-p^\ast}$. In general, $p\neq p^\ast$ and $p^\ast$ might not even be $(\varepsilon,r)$-good. In other words, the closest point to $x$ on the boundary might not be the point in whose cone $x$ is lying. A priori, it is also unclear what the distance of $x$ to its corresponding boundary point is, apart from the bound $\frac{r}{2}$. This issue is solved by the following known lemma, which we cite for convenience. We will use this in Section \ref{sec:heatKernelConstructions} to approximate the heat kernel in the good part of the boundary.

\begin{lemma}[Lemma 3.3 in \cite{MR4145207}]
\label{lem:distOfGoodPointToItsSaddle}
    Let $p\in\partial\Omega$ be $(\varepsilon,r)$-good with $r> 0$ and $\varepsilon \in \big(0,\frac12\big]$. Then, for any $x\in \Gamma_{\varepsilon,r}(p)$,
    \begin{align*}
        \abs{x-p} \leq 2\d{x}.
    \end{align*}
\end{lemma}

Using the previous geometric construction, we find in the following proposition two half-spaces such that the boundary of the domain is locally squeezed between their respective boundaries.

\begin{proposition}
\label{prop:halfPlaneConstructionForGoodPart}
    Let $p\in \partial\Omega$ be $(\varepsilon,r)$-good and $s \in (0,r)$. Then, for all $\rho\geq s\varepsilon$, there exist two open parallel half-spaces $\mathbb{H}_+$ and $\mathbb{H}_-$ with $\rho = \dist{p}{ \partial\mathbb{H}_+} = \dist{p}{ \partial\mathbb{H}_-}$ and
\begin{align*}
    \mathbb{H}_- \cap B_s(p) \subset \Omega\cap B_s(p) \subset \mathbb{H}_+ \cap B_s(p).
\end{align*}
\end{proposition}

\begin{proof}%
    Since $p$ is $(\varepsilon,r)$-good, there exists an inner unit normal $\nu(p)$. Consider the open half-space 
    \begin{align*}
        \mathbb{H} = \{x\in\R^n\mid (x-p)\cdot\nu(p) > 0\}.
    \end{align*}
    Let $\rho \geq s\varepsilon$. Then
    \begin{align*}
        \mathbb{H}_+ = \mathbb{H} - \rho \cdot \nu(p) \qand 
        \mathbb{H}_- = \mathbb{H} + \rho \cdot \nu(p)
    \end{align*}
    are the desired half-spaces. This follows essentially from the fact that $p$ is $(\varepsilon,r)$-good and therefore
    \begin{align}
    \begin{aligned}
        \Omega \cap B_r(p) &\subset \{x\in \R^n \mid (x-p)\cdot \nu(p) > -\varepsilon \abs{x-p}\}, \qand \\
        \Omega^c \cap B_r(p) &\subset \{x\in \R^n \mid (x-p)\cdot \nu(p) <\varepsilon \abs{x-p}\}\cup \{p\}.
    \end{aligned}
    \label{eq:condForBallsAroundPointRegions}
    \end{align}
    
    For the first inclusion, let $x\in \mathbb{H}_- \cap B_s(p)$. Therefore, we have $x-\rho \nu(p) \in \mathbb{H}$, and hence,
    \begin{align*}
        0< (x-\rho \nu(p)-p)\cdot \nu(p) = (x-p)\cdot \nu(p) -\rho.
    \end{align*}
    Thus, by the assumption on $\rho$, we see
    \begin{align*}
        (x-p)\cdot \nu(p) > \rho \geq \varepsilon\cdot s > \varepsilon \abs{x-p},
    \end{align*}
    which implies $x \notin \Omega^c \cap B_r(p)$. But as $x\in B_s(p) \subset B_r(p)$ we have $x\in \Omega\cap B_r(p)$ and therefore $x\in \Omega\cap B_s(p)$.

    For the second inclusion, let $x\in \Omega\cap B_s(p)\subset \Omega\cap B_r(p)$. Then, we have with \eqref{eq:condForBallsAroundPointRegions}, that
    \begin{align*}
        (x+\rho \nu(p)-p)\cdot \nu(p) = (x-p)\cdot \nu(p) + \rho > - \varepsilon \abs{x-p} + \rho > -\varepsilon s+\rho \geq 0,
    \end{align*}
    and thus $x + \rho \nu(p) \in \mathbb{H}$.
\end{proof} 

In Section \ref{sec:heatKernelConstructions}, we use the construction and inclusions in Proposition \ref{prop:halfPlaneConstructionForGoodPart} to compare the heat kernel of $\Omega$ at certain good points within $\Omega$ with the heat kernel of the constructed half-spaces.

\subsection{Inner and outer Minkowski content}
Based on the notion of Minkowski content, we define the inner and outer Minkowski content and other relevant quantities. Afterwards, we prove that, under assumptions for the underlying set, these are equal to the $(n-1)$-dimensional Hausdorff measure of the boundary of the corresponding sets.

\begin{definition}
    For $\Omega\subset\R^n$, let
    \begin{align*}
        \Tilde{\mathrm{d}}_\Omega(x) = \dist{x}{\Omega} - \dist{x}{\Omega^c}
    \end{align*}
    be the signed distance of $x\in\R^n$ to the boundary of $\Omega$. Let 
    \begin{align*}
        \Omega_r = \{x\in\R^n\mid \Tilde{\mathrm{d}}_\Omega(x)<r\} \qand \Omega_{-r} = \{x \in\R^n \mid \Tilde{\mathrm{d}}_\Omega(x) \leq -r\}
    \end{align*}
    for $r >0$. 
\end{definition}
These definitions are chosen such that, if $\Omega$ is open, $\Omega_r$ and $\Omega\setminus \Omega_{-r}$ are open for $r>0$. Similarly to the usual definition of the Minkowski content, see \cite[Definition 2.100]{MR1857292}, we define the inner and outer Minkowski content in the following way.
\begin{definition}
    Let $\Omega\subset\R^n$ be open. We call
    \begin{align*}
        \mathcal{SM}_*^{out} (\Omega) = \liminf_{r\to 0^+}\frac{\abs{\Omega_r\setminus \Omega}}{r} \qand 
        {\mathcal{SM}^{out}}^* (\Omega) = \limsup_{r\to 0^+}\frac{\abs{\Omega_r\setminus \Omega}}{r} 
    \end{align*}
    the lower and upper outer Minkowski content. If $\mathcal{SM}_*^{out} (\Omega)$ = ${\mathcal{SM}^{out}}^* (\Omega)<\infty$ we denote this common value by $\mathcal{SM}^{out}(\Omega)$ and say that $\Omega$ admits an outer Minkowski content.
    Similarly, we call 
    \begin{align*}
        \mathcal{SM}_*^{in} (\Omega) = \liminf_{r\to 0^+}\frac{\abs{\Omega\setminus \Omega_{-r}}}{r} \qand 
        {\mathcal{SM}^{in}}^* (\Omega) = \limsup_{r\to 0^+}\frac{\abs{\Omega\setminus \Omega_{-r}}}{r} 
    \end{align*}
    the lower and upper inner Minkowski content. If $\mathcal{SM}_*^{in} (\Omega)$ = ${\mathcal{SM}^{in}}^* (\Omega)<\infty$ we denote this common value by $\mathcal{SM}^{in}(\Omega)$ and say that $\Omega$ admits an inner Minkowski content.
\end{definition}
The above definition of the outer Minkowski content comes from \cite{3454532120081201}. The inner Minkowski content is defined analogously, and a similar result for this are shown in the following theorem.

\begin{theorem}
\label{thm:innerMinkowskiContentConverges}
    For any open and bounded $\Omega \subset \R^n$ with Lipschitz boundary the inner and outer Minkowski contents exist and are equal to
    \begin{align*}
        \mathcal{SM}^{out}(\Omega) = \mathcal{SM}^{in}(\Omega) =
        \lim_{r\to 0}\frac{\abs{\{x\in\Omega \mid \d{x} < r\}}}{r} = \hausdBoundary.
    \end{align*}
\end{theorem}

\begin{proof}
    First, note that $\abs{\Omega\setminus \Omega_{-r}} =  \abs{(\partial \Omega)_r} - \abs{\Omega_r\setminus \Omega}$, where $(\partial\Omega)_r = \{x\in\R^n\mid \dist{x}{\partial\Omega} < r\}$. Comparing this identity with the definitions of the inner, outer, and usual Minkowski contents of $\partial \Omega$, we see that the existence of two of these implies the existence of the third. We use this idea in the following.

    By combining Proposition 1 and Theorem 5 in \cite{3454532120081201}, we see that $\Omega$ admits outer Minkowski content and $\mathcal{SM}^{out}(\Omega) = \hausdBoundary$, as $\Omega$ has a Lipschitz boundary and is bounded. Furthermore, we have
    \begin{align*}
        \frac{\abs{\Omega\setminus \Omega_{-r}}}{r} &= \frac{\abs{(\partial \Omega)_r}- \abs{\Omega_r\setminus \Omega}}{r} \\
        &= 2\frac{\abs{(\partial \Omega)_r}}{2r} - \frac{\abs{\Omega_r\setminus \Omega}}{r} \overset{r\to 0}{\longrightarrow} 2\hausdBoundary - \mathcal{SM}^{out}(\Omega) = \hausdBoundary,
    \end{align*}
    where the convergence of the first term follows from the definition of the Minkowski content and Proposition 1 in \cite{3454532120081201}. 
    Thus, as the right hand-side converges as $r\to 0$, the term on the left-hand side also converges and $\mathcal{SM}^{in}(\Omega)$ exists. Furthermore, $\mathcal{SM}^{in}(\Omega) = \hausdBoundary$.
\end{proof}

\begin{corollary}
\label{cor:behaviourOfThetaR}
    For any open and bounded $\Omega \subset \R^n$ with Lipschitz boundary, the quantity
    \begin{align*}
        \vartheta_\Omega (r ) = \frac{\abs{\{x\in\Omega \mid \d{x} < r\}}}{r\hausdBoundary} - 1 , \qfor r > 0,
    \end{align*}
    satisfies
    \begin{align*}
        \lim_{r\to 0} \vartheta_\Omega (r ) = 0 \qand \sup_{r>0}\abs{\vartheta_\Omega(r)} < \infty.
    \end{align*} 
\end{corollary}
\begin{proof}
    The first assertion follows directly from Theorem \ref{thm:innerMinkowskiContentConverges}. The second assertion also follows from that theorem combined with the boundedness of $\Omega$.
\end{proof}
\begin{remark}
    In Section \ref{sec:boundsForErrorTerms}, we will derive upper bounds for the weighted heat trace in different regimes that involve the function $\vartheta_\Omega$. Corollary \ref{cor:behaviourOfThetaR} often ensures that we obtain the correct asymptotic behavior. For general Lipschitz domains $\Omega$, there is no hope for a more quantitative behavior of $\vartheta_\Omega(r)$ as $r\to 0$. Thus, using our techniques, the little-$o$ term in Theorem \ref{thm:mainThmWithG} cannot be improved immediately. We believe that there are possible improvements for convex domains, similar to studies in e.g. \cite{MR4946250, frank2026uniformboundsneumannheat}, where more explicit bounds for $\vartheta_\Omega$ are known.
\end{remark}

The preceding statements are applied to show an upper bound for the size of an inner tubular region.
\begin{corollary}
\label{cor:boundednessOfTubularRegions}
    For any open and bounded $\Omega\subset \R^n$ with Lipschitz boundary, there exists $\theta_\Omega > 0$ such that for all $r> 0$, we have
    \begin{align*}
        \abs{\{x\in \Omega \mid \d{x} < r\}} \leq r \theta_\Omega \hausdBoundary.
    \end{align*}
\end{corollary}
\begin{proof}
    This corollary follows directly from the definition of $\vartheta_\Omega$ in Corollary \ref{cor:behaviourOfThetaR} and the boundedness of $\vartheta_\Omega$, with $\theta_\Omega = \norm{\vartheta_\Omega}_\infty + 1$.
\end{proof}

\subsection{Integrals over inner tubular regions}

In this subsection, we are interested in integrals of the form
\begin{align*}
    \int_{\{\d{x} < r\}} g(x) \dd{x}
\end{align*}
for suitable functions $g\colon \overline\Omega \to \R$ and their limit as $r\to 0$. These arise later in the analysis of the main contribution to the weighted heat trace. We call the sets $\{x\in\Omega \mid \d{x}<r\}$ for $r > 0$ inner tubular regions. The main result is summarized in the following theorem, which is proved at the end of this subsection. 
\begin{theorem}
\label{thm:boundaryIntegralConvergence}
    Let $\Omega \subset \R^n$ be an open and bounded set with Lipschitz boundary. Let $g\colon \overline{\Omega}\to \R$ be such that $g$ is continuous in a neighborhood of $\partial\Omega$. Then 
    \begin{align*}
        \frac{1}{r}\int_{\{\d{x} < r\}} g(x) \dd{x} \to \int_{\partial\Omega} g(x) \dd{\mathcal{H}^{n-1}(x)}, \qas r\to 0.
    \end{align*}
\end{theorem}
Observe that the case $g=1$ follows directly from Theorem \ref{thm:innerMinkowskiContentConverges}. To prove Theorem \ref{thm:boundaryIntegralConvergence}, we proceed in two steps. First, we show the statement for indicator functions of certain sets and then for functions that are continuous near the boundary. 

Before proving the theorem, we need to briefly introduce the notion of relative perimeter, following \cite[Section 3.3]{MR1857292}.
\begin{definition}
    Let $E\subset \R^n$ be measurable and $\Omega \subset \R^n$ be open. We denote by
    \begin{align*}
        P(E,\Omega) \coloneq \sup \bigg\{\int_E \text{div} \phi \dd{x} \mid \phi \in [C_c^1(\Omega)]^n, \norm{\phi}_\infty \leq 1\bigg\} %
    \end{align*}
    the perimeter of $E$ in $\Omega$. %
    If $\Omega=\R^n$, we write $P(E) = P(E,\R^n)$.
\end{definition}
\begin{remark}
    By Theorem 3.36 in \cite{MR1857292}, if $P(E)<\infty$, then the distributional derivative $D\1_E$ exists as a Radon measure, and $\abs{D\1_E}(\Omega) = P(E,\Omega)$ for every open $\Omega \subset \R^n$. As in \cite{3454532120081201}, this allows us to extend the notion of relative perimeter to Borel-measurable $A \subset \R^n$ by
    \begin{align*}
        P(E, A) \coloneq \abs{D\1_E}(A).
    \end{align*}
\end{remark}
In the following proposition, we provide a useful property of the perimeter.

\begin{proposition}
\label{prop:factsAboutPerimeter}
    Let $E \subset \R^n$ be compact with Lipschitz boundary and $A \subset \R^n$ be Borel-measurable. Then $P(E,A) = \hausdBoundarySet{\partial E \cap A}$.
\end{proposition}
\begin{proof}
    Note that $\1_E = \1_{E^\circ}$ almost everywhere and therefore $D\1_E = D\1_{E^\circ}$ as Radon measures. Thus, we have
    \begin{align*}
        P(E,A) = \abs{D\1_E}(A) = \abs{D\1_{E^\circ}}(A) = P(E^\circ,A).
    \end{align*}
    By Proposition 3.62 in \cite{MR1857292}, $P(E^\circ,A) = \hausdBoundarySet{\partial (E^\circ) \cap A} = \hausdBoundarySet{\partial E \cap A}$, which implies the claim.
\end{proof}

The preceding proposition allows us to interpret the perimeter of $E$ in $A$ as the $(n-1)$-dimensional measure of the parts of $\partial E$ that lie inside $A$.

The following theorem summarizes the first case of the boundary integral, which heuristically states that the boundary integral is only influenced by parts close to the boundary in a normal direction.

\begin{theorem}
\label{thm:boundaryIntegralConvergenceForIndicator}
    Let $\Omega\subset\R^n$ be an open and bounded set with Lipschitz boundary. Let $A\subset\R^n$ be open and bounded, and assume that $\hausdBoundarySet{\partial\Omega\cap\partial A} = 0$. Then we have
    \begin{align}
        \lim_{r\to 0} \frac{1}{r}\int_{\{\d{x} < r\}} \1_{A}(x) \dd{x} = \hausdBoundarySet{\partial\Omega \cap A}. \label{eq:boundaryIntegralConvergenceForIndicator}
    \end{align}
\end{theorem}

\begin{proof}
    First, note that
\begin{align*}
    \frac{1}{r}\int_{\{\d{x} <r\}} \1_A(x)\dd{x} = \frac{\abs{\{\d{x} < r\}\cap A}}{r} = \frac{\abs{\left(\Omega\setminus\Omega_{-r}\right)\cap A}}{r}.
\end{align*}

To analyze this expression, let $R> 0$ be so that $\Omega\cup A \subset B_R(0)$ and set $E = \overline{B_{2R}(0)}\setminus \Omega$. Then clearly $\partial E \cap A = \partial\Omega\cap A$ and
\begin{align*}
    \left(\Omega\setminus \Omega_{-r}\right) \cap A = \left(E_r\setminus E\right) \cap A
\end{align*}
for $r$ small enough. In other words, the inner parallel set of $\Omega$ can be seen as an outer parallel set of $E$. %

Note that, by construction, $E$ is closed and bounded with Lipschitz boundary. By \cite[Corollary~1]{3454532120081201}, $\mathcal{SM}^{out}(E) = P(E)$, and therefore by \cite[Proposition 2]{3454532120081201} we have that
\begin{align*}
    \lim_{r\to 0} \frac{\abs{(E_r\setminus E)\cap A}}{r} = P(E,A),
\end{align*}
if $P(E, \partial A) = 0$. By Proposition \ref{prop:factsAboutPerimeter}, we see that indeed
\begin{align*}
    P(E, \partial A) = \hausdBoundarySet{\partial E \cap \partial A}= \hausdBoundarySet{\partial \Omega \cap \partial A} = 0.
\end{align*}
Thus, we have
\begin{align*}
    \lim_{r\to 0} \frac{\abs{\left(\Omega\setminus\Omega_{-r}\right)\cap A}}{r} = \lim_{r\to 0} \frac{\abs{(E_r\setminus E)\cap A}}{r} = P(E,A) = \hausdBoundarySet{\partial E\cap A} = \hausdBoundarySet{\partial \Omega\cap A},
\end{align*}
where we used again Proposition \ref{prop:factsAboutPerimeter} in the second to last step. 
\end{proof}

\begin{figure}[ht]
\centering

\begin{subfigure}[t]{0.48\textwidth}
\centering
\begin{tikzpicture}[x=4cm,y=4cm,line cap=round,line join=round]
  \path (-0.18,-0.62) rectangle (1.17,1.12);

  \draw[->,gray!55] (-0.10,0) -- (1.12,0) node[right] {$x$};
  \draw[->,gray!55] (0,-0.56) -- (0,1.08) node[above] {$y$};

  \fill[gray!12] (0,0) rectangle (1,1);

\fill[blue!25,opacity=.45] (0,-.5) rectangle (1,.5);

\fill[blue!55,opacity=.35] (0,0) rectangle (1,.5);

  \draw[blue!70!black,dashed,thick] (0,-.5) rectangle (1,.5);

  \draw[black,very thick] (0,0) rectangle (1,1);

  \draw[red!80!black,very thick] (0,0) -- (1,0);

  \fill[white] (0,0) circle (.012);
  \draw[red!80!black,thick] (0,0) circle (.012);
  \fill[white] (1,0) circle (.012);
  \draw[red!80!black,thick] (1,0) circle (.012);

  \foreach \x/\lab in {0/0,1/1}{
    \draw[gray!70] (\x,.012) -- (\x,-.012);
    \node[gray!80] at ([xshift=-5pt,yshift=-10pt]\x,0) {$\lab$};
  }

  \foreach \y/\lab in {-.5/-\frac12,.5/\frac12,1/1}{
    \draw[gray!70] (.012,\y) -- (-.012,\y);
    \node[left,gray!80] at (-.025,\y) {$\lab$};
  }

\node[black] at (.5,.75) {$\Omega=(0,1)^2$};
\node[blue!70!black] at (.5,.25) {$A\cap\Omega$};
\node[blue!70!black] at (.5,-.32) {$A$};
\node[red!80!black] at (.5,-.1) {$A\cap\partial\Omega$};
\end{tikzpicture}
\caption{Configuration according to Remark \ref{rem:exForGeometricCondA}.}
\label{fig:configuration-a-geomCond}
\end{subfigure}
\hfill
\begin{subfigure}[t]{0.48\textwidth}
\centering
\begin{tikzpicture}[x=4cm,y=4cm,line cap=round,line join=round]
  \path (-0.18,-0.62) rectangle (1.17,1.12);

  \draw[->,gray!55] (-0.10,0) -- (1.12,0) node[right] {$x$};
  \draw[->,gray!55] (0,-0.56) -- (0,1.08) node[above] {$y$};

  \fill[gray!12] (0,0) rectangle (1,1);

\fill[blue!25,opacity=.45]
  (0,-.5) -- (1,-.5) -- (1,0) -- (.75,0) --
  (.75,.5) -- (.25,.5) -- (.25,0) -- (0,0) -- cycle;

\fill[blue!55,opacity=.35] (.25,0) rectangle (.75,.5);

  \draw[blue!70!black,dashed,thick]
    (0,-.5) -- (1,-.5) -- (1,0) -- (.75,0) --
    (.75,.5) -- (.25,.5) -- (.25,0) -- (0,0) -- cycle;

  \draw[black,very thick] (0,0) rectangle (1,1);

  \draw[red!80!black,very thick] (.25,0) -- (.75,0);

  \fill[white] (.25,0) circle (.012);
  \draw[red!80!black,thick] (.25,0) circle (.012);
  \fill[white] (.75,0) circle (.012);
  \draw[red!80!black,thick] (.75,0) circle (.012);

  \foreach \x/\lab in {0/0,.25/\frac14,.75/\frac34,1/1}{
    \draw[gray!70] (\x,.012) -- (\x,-.012);
    \node[gray!80] at ([xshift=-5pt,yshift=-10pt]\x,0) {$\lab$};
  }

  \foreach \y/\lab in {-.5/-\frac12,.5/\frac12,1/1}{
    \draw[gray!70] (.012,\y) -- (-.012,\y);
    \node[left,gray!80] at (-.025,\y) {$\lab$};
  }

\node[black] at (.5,.75) {$\Omega=(0,1)^2$};
\node[blue!70!black] at (.5,.25) {$A\cap\Omega$};
\node[blue!70!black] at (.5,-.32) {$A$};
\node[red!80!black] at (.5,-.1) {$A\cap\partial\Omega$};
\end{tikzpicture}
\caption{Configuration according to Remark \ref{rem:exForGeometricCondB}.}
\label{fig:configuration-b-geomCond}
\end{subfigure}

\caption{Illustration of two configurations of \(A\) relative to \(\Omega=(0,1)^2\).}
\label{fig:two-configurations}
\end{figure}
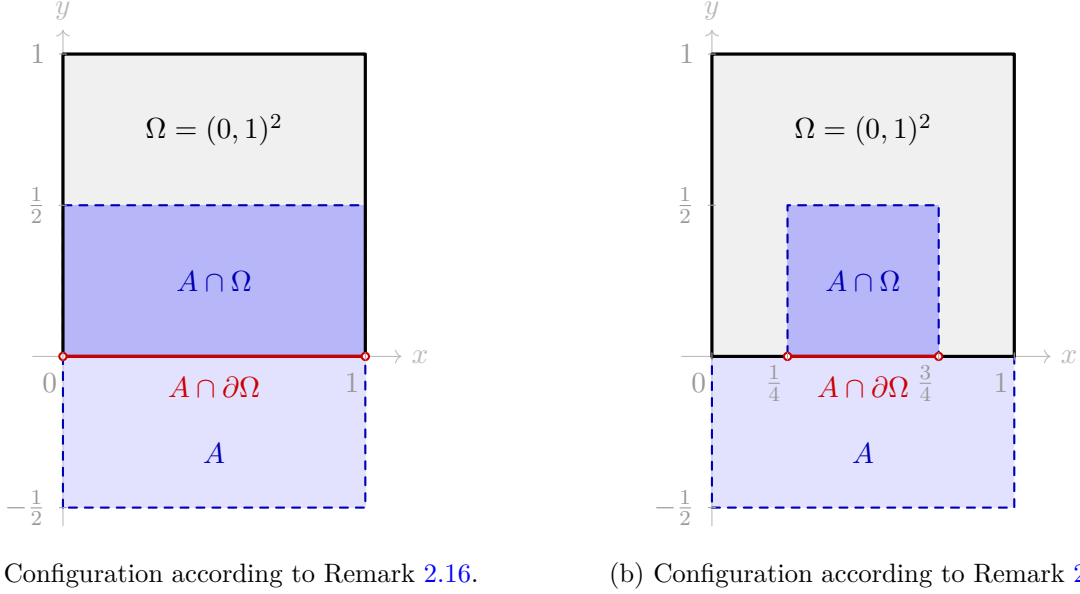

\begin{remark}
\label{rem:exForGeometricCondA}
    The geometric condition $\hausdBoundarySet{\partial\Omega \cap \partial A} = 0$ in the preceding theorem ensures that $\partial A$ does not have too much overlap with $\partial\Omega$. As an illustration that some geometric control is necessary, we give an explicit example where \eqref{eq:boundaryIntegralConvergenceForIndicator} fails if the geometric condition is not satisfied. A sketch of this situation is given in Figure \ref{fig:configuration-a-geomCond}. Consider $\Omega = (0,1)^2 \subset \R^2$ and $A = (0,1)\times \big(-\frac12,\frac12\big)$. Then $\partial\Omega\cap A = (0,1)\times \{0\}\neq \emptyset$ and $\partial\Omega \cap \partial A = \Big(\{0\}\times \Big[0,\frac12\Big] \Big)\cup \Big(\{1\}\times\Big[0,\frac12\Big]\Big)$, which implies $\mathcal{H}^1(\partial\Omega \cap \partial A) = 1$. Hence, the geometric condition of Theorem \ref{thm:boundaryIntegralConvergenceForIndicator} is not fulfilled. Furthermore, 
    \begin{align*}
        \int_{\{\d{x}<r\}} \1_{A} \dd{x} = \frac{r}{2}\cdot 2 + r\cdot (1-2r) = 2r-2r^2,
    \end{align*}
    if $r < \frac{1}{2}$. Hence,
    \begin{align*}
         \lim_{r\to0} \frac1r \int_{\{\d{x}<r\}} \1_A(x)\dd{x} = 2.
    \end{align*}
    On the other hand,
    \begin{align*}
        \hausdBoundarySet{\partial\Omega \cap A} = 1.
    \end{align*}
    Therefore, \eqref{eq:boundaryIntegralConvergenceForIndicator} fails in this example.
\end{remark}
\begin{remark}
\label{rem:exForGeometricCondB}
    However, the condition $\hausdBoundarySet{\partial A\cap\partial\Omega} = 0$ in Theorem \ref{thm:boundaryIntegralConvergenceForIndicator} is not necessary. To see this, consider again $\Omega = (0,1)^2 \subset \R^2$ and $A = \Big(\Big(\frac14,\frac34\Big)\times\Big(-\frac12,\frac12\Big)\Big)\cup \Big((0,1)\times \Big(-\frac12, 0\Big)\Big)$, see also the sketch in Figure \ref{fig:configuration-b-geomCond}. Then we have $\partial A\cap\partial\Omega = \Big(\Big[0,\frac14\Big]\times \{0\}\Big)\cup \Big( \Big[\frac34,1\Big]\times\{0\}\Big)$, and therefore
    \begin{align*}
         \hausdBoundarySet{\partial A\cap \partial\Omega} = \frac12 > 0.
    \end{align*}
    Hence, the geometric assumptions of Theorem \ref{thm:boundaryIntegralConvergenceForIndicator} are not fulfilled. Nevertheless, it holds that $A\cap\partial\Omega = \Big(\frac14,\frac34\Big)\times\{0\}$ and $A\cap\{\d{x}<r\} = \Big(\frac14,\frac34\Big)\times (0,r)$ for $r< \frac14$, which implies 
     \begin{align*}
        \lim_{r\to 0}\frac1r \int_{\{\d{x}<r\}}\1_A(x)\dd{x} = \frac{1}{2} = \hausdBoundarySet{A\cap \partial\Omega} .
     \end{align*}
\end{remark}

Due to the following corollary, Theorem \ref{thm:boundaryIntegralConvergenceForIndicator} can be applied to balls of almost every radius centered at a boundary point of $\Omega$.
\begin{corollary}
\label{cor:IntegralBallResult}
    Let $\Omega \subset \R^n$ be an open and bounded set with Lipschitz boundary. Let $y\in\partial\Omega$. Then, for a.e. $r>0$, the ball $B_r(y)$ satisfies
    \begin{align*}
        B_r(y)\cap \partial\Omega \neq \emptyset \qand \hausdBoundarySet{\partial\Omega\cap\partial B_r(y)} = 0.
    \end{align*}
\end{corollary}
\begin{proof}
    Note that $\Tilde{\mathrm{d}}_{\{y\}}(\cdot) \colon \R^n \to \R$ is a Lipschitz function and $\mathcal{H}^n(\partial\Omega) = 0$. Thus, applying \cite[Lemma 2.95]{MR1857292},
    \begin{align*}
        \hausdBoundarySet{\partial\Omega \cap \Tilde{\mathrm{d}}_{\{y\}}^{-1}(r)} = 0, \qfor \text{a.e. } r>0.
    \end{align*}
    Now, the claim follows by noting that $\partial B_r(y) = \{x\in\R^n \mid \Tilde{\mathrm{d}}_{\{y\}}(x) = r\}$.
\end{proof}

Finally, we are able to prove the main theorem of this subsection by approximating continuous functions by specific simple functions.

\begin{proof}[Proof of Theorem \ref{thm:boundaryIntegralConvergence}]
    For $\eta > 0$, let $N_\eta = \{x\in\overline{\Omega} \mid \d{x}<\eta\}$. As $g$ is continuous in a neighborhood of the boundary, there exists an $\eta_0 > 0$ such that $g$ is continuous on $N_{2\eta_0}$.
    As $\Omega$ is bounded, so is $\overline{N}_{\eta_0}$, and $g$ is therefore uniformly continuous on $\overline{N}_{\eta_0}$. Fix $\varepsilon > 0$. Then there exists $\eta \in (0,\eta_0)$ such that for $x,y\in N_{\eta_0}$ with $\abs{x-y}<\eta$ we have $\abs{g(x)-g(y)}< \varepsilon$. As $\partial \Omega$ is compact, we can find a finite set of points $\{y_i\}_{i=1}^N\subset\partial\Omega$ such that the balls $B_i = B_{\eta/2}(y_i)$ are a finite cover of $\partial \Omega$, that is,
\begin{align*}
    \partial\Omega\subset\bigcup_{i=1}^N B_i,  
\end{align*}
and 
\begin{align}
    \mathcal{H}^{n-1}\bigg(\partial\Omega\cap\partial B_\eta(y_i)\bigg) = 0 \label{eq:choosingOfBallsNearBoundary}
\end{align}
for all $i\in\{1,\dots,N\}$ by Corollary \ref{cor:IntegralBallResult}.
For $r < \frac{\eta}{2}$, we have by the triangle inequality that $\{x \in \overline{\Omega} \mid \d{x} < r\} \subset \bigcup_{i=1}^N \Tilde{B}_i$ with $\Tilde{B}_i = B_\eta(y_i)$.

We define a locally constant version $\Tilde{g}$ of $g$ around the boundary by 
\begin{align*}
    \Tilde{g}(x) = \sum_{i=1}^N g(y_i) \1_{\Tilde{B}_i\setminus \bigcup_{j=1}^{i-1} \Tilde{B}_j}(x),
\end{align*}
for $x\in \R^n$. Note that $\Tilde{g}$ has support also outside of $\Omega$. The function $\Tilde{g}$ is constructed so that it is a good approximation for $g$ on $\{x\in\overline{\Omega}\mid\d{x}<r\}$. Indeed, let $x\in \overline{\Omega}$ with $\d{x}<r$. Then there exists $i\in\{1,\dots,N\}$ such that $x\in \Tilde{B}_i$, but $x\notin \Tilde{B}_j$ for all $j< i$. Thus,
\begin{align*}
    \abs{g(x)-\Tilde{g}(x)} = \abs{g(x)-g(y_i)}<\varepsilon
\end{align*}
by the uniform continuity of $g$ established above. Then, we have
\begin{align}
\begin{aligned}
    \abs{\frac{1}{r}\int_{\{\d{x} < r\}}g(x)\dd x - \int_{\partial\Omega}g(x)\dd\mathcal{H}^{n-1}(x)} &\leq \abs{\frac{1}{r}\int_{\{\d{x} < r\}}g(x)\dd x - \frac{1}{r}\int_{\{\d{x} < r\}}\Tilde{g}(x)\dd x} \\
    &\hspace{2em}+ \abs{\frac{1}{r}\int_{\{\d{x} < r\}}\Tilde{g}(x)\dd x - \int_{\partial\Omega}\Tilde{g}(x)\dd\mathcal{H}^{n-1}(x)}  \\
    &\hspace{2em}+ \abs{\int_{\partial\Omega}\Tilde{g}(x)\dd\mathcal{H}^{n-1}(x) - \int_{\partial\Omega}g(x)\dd\mathcal{H}^{n-1}(x)}.
\end{aligned}
\label{eq:mainTermInLocalizationResultLevelSetIntegral}
\end{align}
The first and third terms are small due to $\Tilde{g}$ being an approximation of $g$, which can be seen by the following estimates where we apply Corollary \ref{cor:behaviourOfThetaR},
\begin{align*}
    \abs{\frac{1}{r}\int_{\{\d{x} < r\}}g(x)\dd x - \frac{1}{r}\int_{\{\d{x} < r\}}\Tilde{g}(x)\dd x}  &\leq \frac{1}{r}\int_{\{\d{x} < r\}}\abs{g(x)-\Tilde{g}(x)}\dd x \\
    &\leq \varepsilon \frac{\abs{\{\d{x}<r\}}}{r} \\
    & = \varepsilon \hausdBoundary\big(1+\vartheta_\Omega(r)\big), \qand\\
    \abs{\int_{\partial\Omega}\Tilde{g}(x)\dd\mathcal{H}^{n-1}(x) - \int_{\partial\Omega}g(x)\dd\mathcal{H}^{n-1}(x)} &\leq \int_{\partial\Omega} \abs{\Tilde{g}(x)-g(x)}\dd\mathcal{H}^{n-1}(x) \\
    &\leq \varepsilon \hausdBoundary.
\end{align*}

The middle term in \eqref{eq:mainTermInLocalizationResultLevelSetIntegral} is handled by the considerations for integrals over indicator functions, that is,
\begin{align*}
    \frac{1}{r}\int_{\{\d{x} < r\}}\Tilde{g}(x)\dd x &= \sum_{i=1}^N \frac{1}{r}g(y_i)\int_{\{\d{x}<r\}} \1_{\Tilde{B}_i\setminus \bigcup_{j=1}^{i-1} \Tilde{B}_j} \dd{x} \\
    &= \sum_{i=1}^Ng(y_i) \frac{\abs{\Big(\Tilde{B}_i\setminus \bigcup_{j=1}^{i-1} \Tilde{B}_j\Big)\cap \{\d{x}<r\}}}{r} \\
    &= \sum_{i=1}^N g(y_i) \mathcal{H}^{n-1}\Bigg(\partial\Omega\cap \bigg(\Tilde{B}_i\setminus \bigcup_{j=1}^{i-1} \Tilde{B}_j\bigg)\Bigg) \\
    &+ \sum_{i=1}^N g(y_i) \Bigg[\frac{\abs{\Big(\Tilde{B}_i\setminus \bigcup_{j=1}^{i-1} \Tilde{B}_j\Big)\cap \{\d{x}<r\}}}{r}-\mathcal{H}^{n-1}\Bigg(\partial\Omega\cap \bigg(\Tilde{B}_i\setminus \bigcup_{j=1}^{i-1} \Tilde{B}_j\bigg)\Bigg)\Bigg] .
    \end{align*}
The first sum can now be rewritten as
    \begin{align*}
    \sum_{i=1}^N g(y_i) \mathcal{H}^{n-1}\Bigg(\partial\Omega\cap \big(\Tilde{B}_i\setminus \bigcup_{j=1}^{i-1} \Tilde{B}_j\big)\Bigg)&= \sum_{i=1}^N g(y_i) \int_{\partial\Omega} \1_{\big(\Tilde{B}_i\setminus \bigcup_{j=1}^{i-1} \Tilde{B}_j\big)}\dd\mathcal{H}^{n-1} \\
    &= \int_{\partial\Omega} \Tilde{g} \dd\mathcal{H}^{n-1}.
\end{align*}
For the second sum we apply Theorem \ref{thm:boundaryIntegralConvergenceForIndicator} with $A = \Tilde{B}_i\setminus \bigcup_{j=1}^{i-1} \overline{\Tilde{B}_j}$ to obtain 
\begin{align*}
    &\lim_{r\to 0}\sum_{i=1}^N g(y_i) \left[\frac{\abs{\Big(\Tilde{B}_i\setminus \bigcup_{j=1}^{i-1} \Tilde{B}_j\Big)\cap \{\d{x}<r\}}}{r}-\mathcal{H}^{n-1}\Bigg(\partial\Omega\cap \Big(\Tilde{B}_i\setminus \bigcup_{j=1}^{i-1} \Tilde{B}_j\Big)\Bigg)\right] \\
    &\hspace{1em}= \lim_{r\to 0}\sum_{i=1}^N g(y_i) \left[\frac{\abs{\Big(\Tilde{B}_i\setminus \bigcup_{j=1}^{i-1} \overline{\Tilde{B}_j}\Big)\cap \{\d{x}<r\}}}{r}-\mathcal{H}^{n-1}\Bigg(\partial\Omega\cap \Big(\Tilde{B}_i\setminus \bigcup_{j=1}^{i-1} \overline{\Tilde{B}_j}\Big)\Bigg)\right]  = 0,
\end{align*}
if $\mathcal{H}^{n-1}\bigg(\partial\Omega \cap \partial\Big(\Tilde{B}_i\setminus \bigcup_{j=1}^{i-1} \overline{\Tilde{B}_j}\Big)\bigg) = 0$. This follows by noting that
\begin{align*}
    \partial\bigg(\Tilde{B}_i\setminus \bigcup_{j=1}^{i-1} \overline{\Tilde{B}_j}\bigg) \subset \bigcup_{j=1}^i \partial \Tilde{B_j} \subset \bigcup_{j=1}^N \partial \Tilde{B_j}
\end{align*}
and therefore
\begin{align*}
    \mathcal{H}^{n-1}\Bigg(\partial\Omega \cap \partial\bigg(\Tilde{B}_i\setminus \bigcup_{j=1}^{i-1} \overline{\Tilde{B}_j}\bigg)\Bigg) \leq \sum_{j=1}^N\mathcal{H}^{n-1}\bigg(\partial\Omega\cap\partial\Tilde{B}_j\bigg) =0,
\end{align*}
where the last equality follows from \eqref{eq:choosingOfBallsNearBoundary}.

Hence,
\begin{align*}
    \limsup_{r\to 0} \abs{\frac{1}{r}\int_{\{\d{x} < r\}}g(x)\dd x - \int_{\partial\Omega}g(x)\dd\mathcal{H}^{n-1}(x)} \leq 2\varepsilon \hausdBoundary.
\end{align*}
Noting that $ \hausdBoundary < \infty$ and letting $\varepsilon \to 0$ yields the result.

\end{proof}

\section{Main contribution to the weighted heat trace}
\label{sec:mainContribution}

This section is devoted to the analysis of the first integral in the decomposition \eqref{eq:decompInMainContrAndError}. This yields precisely the main contribution in Theorems \ref{thm:mainThmWithG} and \ref{thm:mainThmWithIndicatorFunc}. We distinguish three different cases for the exponent $\alpha$, namely we have $\alpha \in (-3,-1)$, $\alpha \in (-1, \infty)$ and $\alpha = -1$. Note that in the case of $\alpha \in (-1, \infty)$, there are two relevant terms. In the proofs, we rely heavily on the analysis of the integrals over inner tubular regions established in the previous section.

\begin{theorem}
\label{thm:mainTermAlphaLarger1}
    Let $\Omega \subset \R^n$ be an open and bounded set with Lipschitz boundary. Let 
    \begin{itemize}
        \item $g\colon\overline{\Omega}\to \R$ be a bounded function that is continuous in a neighborhood of the boundary, or
        \item $g = \1_A$ for an open and bounded $A\subset \R^n$ such that $\hausdBoundarySet{\partial\Omega \cap \partial A} = 0$.
    \end{itemize}
    Then, for $\alpha \in (-3,-1)$, we have
    \begin{align*}
        \int_\Omega g(x)\d{x}^{\alpha}\bigg(1-e^{-\frac{\d{x}^2}{t}}\bigg)\dd{x} &= -t^{\frac{1+\alpha}{2}} \frac{1}{2}\Gamma\bigg(\frac{1+\alpha}{2}\bigg) \int_{\partial\Omega} g(x) \dd\mathcal{H}^{n-1}(x) + o\Big(t^{\frac{1+\alpha}{2}}\Big),
    \end{align*}
    as $t\to 0$.
\end{theorem}

\begin{proof}
    Note that
    \begin{align*}
        \int_\Omega g(x)\d{x}^{\alpha}\bigg(1-e^{-\frac{\d{x}^2}{t}}\bigg)\dd{x} = t^{\frac{\alpha}{2}} \int_\Omega g(x)f\bigg(\frac{\d{x}}{\sqrt{t}}\bigg)\dd{x}
    \end{align*}
    with $f(s) = s^\alpha\big(1-e^{-s^2}\big)$.
    Using the fact that $f(s)$ is differentiable on $(0,\infty)$ and tends to $0$ as $s\to\infty$, we have
    \begin{align*}
        f(s) = -\int_s^\infty f^\prime(w)\dd{w},
    \end{align*}
    and therefore
    \begin{align}
    \begin{aligned}
        t^{\frac{\alpha}{2}} \int_\Omega g(x)f\bigg(\frac{\d{x}}{\sqrt{t}}\bigg)\dd{x} &= -t^{\frac{\alpha}{2}}\int_\Omega g(x) \int_{\frac{\d{x}}{\sqrt{t}}}^\infty f^\prime(w)\dd{w}\dd{x} \\
        &= -t^{\frac{\alpha}{2}}\int_0^\infty f^\prime(w) \int_\Omega g(x) \1_{\{w > \frac{\d{x}}{\sqrt{t}}\}} \dd{x}\dd{w} \\
        &= -t^{\frac{1+\alpha}{2}}\int_0^\infty f^\prime(w)w \frac{1}{w\sqrt{t}}\int_\Omega g(x) \1_{\{w\sqrt{t} > \d{x}\}} \dd{x}\dd{w} .
    \end{aligned}
    \label{eq:localizationRewrittenIntegralMainContribution}
    \end{align}
    Note that we can interchange the integrals due to Fubini's theorem as 
    \begin{align}
        \int_0^\infty\int_\Omega \abs{g(x)f^\prime(w)\1_{\{\d{x}<w\sqrt{t}\}}}\dd{x}\dd{w} &\leq \norm{g}_\infty \int_0^\infty \abs{f^\prime(w)}\cdot \abs{\{\d{x}<w\sqrt{t}\}}\dd{w} \nonumber \\
        &\leq \norm{g}_\infty \sqrt{t}\theta_\Omega \hausdBoundary \int_0^\infty \abs{f^\prime(w)}\cdot w\dd{w} < \infty, \label{eq:interchangeIntegralsDueToFubini}
    \end{align}
    where we used Corollary \ref{cor:boundednessOfTubularRegions}. The finiteness of the integral in the last step will be shown at the end of the proof.
    We denote
    \begin{align*}
        G(s) \coloneq \frac{1}{s} \int_\Omega g(x) \1_{\{\d{x} < s\}}\dd{x}
    \end{align*}
    for $s>0$.
    We saw in Theorems \ref{thm:boundaryIntegralConvergence} and \ref{thm:boundaryIntegralConvergenceForIndicator} that for each of the two assumptions on $g$ we have
    \begin{align}
        \lim_{s\to 0} G(s) = \int_{\partial \Omega} g(x) \dd\mathcal{H}^{n-1}(x) \eqqcolon G, \label{eq:convOfGsToG}
    \end{align}
    so we write
    \begin{align}
        \int_0^\infty f^\prime(w)w G(w\sqrt{t})\dd{w} = 
        G\int_0^\infty f^\prime(w)w\dd{w} + \int_0^\infty f^\prime(w)w \Big(G(w\sqrt{t})- G\Big)\dd{w}. \label{eq:proofBoundaryIntegralDecomposition}
    \end{align} 
    The first integral is now independent of $\Omega$. To evaluate it, we integrate by parts and substitute $u=w^2$ to obtain
    \begin{align*}
        \int_0^\infty  f^\prime(w)w\dd{w} = f(w)w\Big\rvert_0^\infty - \int_0^\infty f(w)\dd{w} &= -\int_0^\infty w^{\alpha}(1-e^{-w^2})\dd{w} \\
        &= -\frac{1}{2}\int_0^\infty u^{\frac{\alpha-1}{2}}(1-e^{-u}) \dd{u},
    \end{align*}
    where the boundary term vanishes because of the decay behavior of $f$ at both endpoints.
    Furthermore, we use $1-e^{-u} = u\int_0^1e^{-vu}\dd{v}$ and interchange the integrals due to Fubini's theorem to get
    \begin{align*}
        -\frac{1}{2}\int_0^\infty u^{\frac{\alpha-1}{2}}(1-e^{-u}) \dd{u} &= -\frac{1}{2}\int_0^\infty u^{\frac{\alpha+1}{2}}\int_0^1e^{-vu}\dd{v} \dd{u} \\
        &= -\frac{1}{2} \int_0^1\int_0^\infty u^{\frac{\alpha+1}{2}}e^{-vu}\dd{u}\dd{v}.
    \end{align*}
    The inner integral is evaluated due to a final substitution $z = vu$ to
    \begin{align*}
        \int_0^\infty u^{\frac{\alpha+1}{2}}e^{-vu}\dd{u} = v^{-\frac{\alpha+1}{2}-1}\int_0^\infty z^{\frac{\alpha+1}{2}}e^{-z}\dd{z} = v^{-\frac{\alpha+3}{2}}\Gamma\bigg(\frac{\alpha+3}{2}\bigg).
    \end{align*}
    Thus, we obtain
    \begin{align}
        \int_0^\infty f^\prime(w)w \dd{w}  &= -\frac{1}{2}\Gamma\bigg(\frac{\alpha+3}{2}\bigg)\int_0^1 v^{-\frac{\alpha+3}{2}}\dd{v} \nonumber\\
        &= \frac{1}{1+\alpha}\Gamma\bigg(\frac{3+\alpha}{2}\bigg) \nonumber\\
        &= \frac{1}{2}\Gamma\bigg(\frac{1+\alpha}{2}\bigg). \label{eq:finitenessOfGammaIntegral1}
    \end{align}

    It remains to show that the latter integral in \eqref{eq:proofBoundaryIntegralDecomposition} is $o(1)$ as $t\to 0$ and that \eqref{eq:interchangeIntegralsDueToFubini} holds. Regarding the first claim, due to \eqref{eq:convOfGsToG}, the integrand is $o(1)$ as $t\to 0$. Note that 
    \begin{align*}
        \abs{G(s)} \leq \frac{1}{s} \norm{g}_\infty \abs{\{\d{x}<s\}} \leq \norm{g}_\infty \theta_\Omega\hausdBoundary,
    \end{align*}
    by Corollary \ref{cor:boundednessOfTubularRegions}, and $\abs{G} \leq \norm{g}_\infty \hausdBoundary$. Thus, the integrand can, for any $w\in(0,\infty)$, be bounded by $\abs{f^\prime(w)}w\norm{g}_\infty \hausdBoundary\big(1+\theta_\Omega\big)$. We now show that this bound is integrable.

    If $\alpha \in (-2,-1)$, $f(s)$ has a maximum at some $s^\ast \in (0,\infty)$, is monotonically increasing until $s^\ast$ and monotonically decreasing after $s^\ast$. Thus, we can write
    \begin{align*}
        \int_0^\infty \abs{f^\prime(w)} w \dd{w} = \int_0^{s^\ast} f^\prime(w) w\dd{w} - \int_{s^\ast}^\infty f^\prime(w) w\dd{w} = 2\int_0^{s^\ast} f^\prime(w) w\dd{w} - \int_0^\infty f^\prime(w) w\dd{w},
    \end{align*}
    and both integrals are finite. The first integral is finite due to the integrability near $0$ and the second integral is finite due to \eqref{eq:finitenessOfGammaIntegral1}.
    If $\alpha \in (-3,-2]$, $f(s)$ is monotonically decreasing and therefore
    \begin{align*}
        \int_0^\infty \abs{f^\prime(w)} w \dd{w} = -\int_0^\infty f^\prime(w) w \dd{w},
    \end{align*}
    which is finite by \eqref{eq:finitenessOfGammaIntegral1}.

    Thus, in both cases $\abs{f^\prime(w)}w\norm{g}_\infty \hausdBoundary\big(1+\theta_\Omega\big)$ is an integrable dominant. The second term in \eqref{eq:proofBoundaryIntegralDecomposition} is therefore $o(1)$ as $t\to 0$ by dominated convergence. Furthermore, we have shown that \eqref{eq:interchangeIntegralsDueToFubini} holds.
\end{proof}

In the case $\alpha \in (-1,\infty)$, we need to understand the two integrals on the right-hand side of \eqref{eq:splittingMainTerms}. First, we show that the first one is indeed finite. Afterwards, we prove the asymptotics for the second integral.

\begin{lemma}
\label{lem:mainTermAlphaLess1ExtraTerm}
    Let $\Omega \subset \R^n$ be an open and bounded set with Lipschitz boundary and $g\in L^\infty(\overline{\Omega})$. If $\alpha > -1$, then
    \begin{align*}
        \int_\Omega g(x) \d{x}^\alpha \dd{x}
    \end{align*}
    is absolutely convergent.
\end{lemma}
\begin{proof}
    First, note that 
    \begin{align*}
        \int_\Omega \abs{g(x) \d{x}^\alpha} \dd{x} \leq \norm{g}_\infty \int_\Omega \d{x}^\alpha \dd{x}.
    \end{align*}
    Hence, we can restrict ourselves to the case $g = 1$.

    If $\alpha \geq 0$, the integral is finite by the boundedness of $\Omega$, as
    \begin{align*}
        \abs{\int_\Omega  \d{x}^{\alpha}\dd{x}} \leq  \int_\Omega \d{x}^{\alpha}\dd{x} &\leq  \varrho^\alpha \abs{\Omega}<\infty,
    \end{align*}
    with $\varrho = \sup_{x\in\Omega}\d{x} < \infty$. 
    
    If $\alpha\in(-1,0)$, we write
    \begin{align*}
        \d{x}^\alpha = -\alpha\int_{\d{x}}^\infty s^{\alpha-1}\dd{s}
    \end{align*}
    and therefore
    \begin{align*}
        \int_\Omega \d{x}^\alpha \dd{x} &= -\alpha \int_\Omega \int_0^\infty \1_{\{\d{x}<s\}} s^{\alpha-1}\dd{s}\dd{x}  \\
        &= -\alpha\int_0^\infty s^{\alpha-1}\abs{\{\d{x}< s\}}\dd{s} \\
        &= -\alpha\int_0^\varrho s^{\alpha-1}\abs{\{\d{x}< s\}}\dd{s} -\alpha\abs{\Omega}\int_\varrho^\infty s^{\alpha-1}\dd{s}.
    \end{align*}
    For the first term, we apply Corollary \ref{cor:boundednessOfTubularRegions} to obtain
    \begin{align*}
        0 \leq -\alpha\int_0^\varrho s^{\alpha-1}\abs{\{\d{x}< s\}}\dd{s} \leq - \alpha\theta_\Omega\hausdBoundary \int_0^\varrho s^\alpha \dd{s} = \theta_\Omega\hausdBoundary\frac{-\alpha}{\alpha+1} \varrho^{\alpha +1},
    \end{align*}
    and the second term evaluates to $\abs{\Omega}\varrho^{\alpha}$. This shows that, for $\alpha\in(-1,0)$,
    \begin{align*}
        \abs{\int_\Omega \d{x}^\alpha \dd{x}} \leq \frac{-\alpha}{1+\alpha}\theta_\Omega\hausdBoundary \varrho^{\alpha +1} + \abs{\Omega}\varrho^\alpha < \infty.
    \end{align*}
\end{proof}

\begin{theorem}
\label{thm:mainTermAlphaLess1}
    Let $\Omega \subset \R^n$ be an open and bounded set with Lipschitz boundary. Let 
    \begin{itemize}
        \item $g\colon\overline{\Omega}\to \R$ be a bounded function that is continuous in a neighborhood of the boundary, or
        \item $g = \1_A$ for an open and bounded $A\subset \R^n$ such that $\hausdBoundarySet{\partial\Omega \cap \partial A} = 0$.
    \end{itemize}
    Then, for $\alpha \in (-1,\infty)$, we have
    \begin{align*}
        \int_\Omega g(x)\d{x}^{\alpha}e^{-\frac{\d{x}^2}{t}}\dd{x} &= t^{\frac{1+\alpha}{2}} \frac{1}{2}\Gamma\bigg(\frac{1+\alpha}{2}\bigg) \int_{\partial\Omega} g(x) \dd\mathcal{H}^{n-1}(x) + o\Big(t^\frac{1+\alpha}{2}\Big), 
    \end{align*}
    as $t\to 0$.
\end{theorem}

\begin{proof}
    Let $G(s)$ and $G$ be as in the proof of Theorem \ref{thm:mainTermAlphaLarger1}. We proceed similarly to that proof and approximate $G(\sqrt{t}w)$ by $G$ in the resulting integral. Using \eqref{eq:localizationRewrittenIntegralMainContribution}, now with $f(s)=s^\alpha e^{-s^2}$, we obtain
    \begin{align}
        \int_\Omega g(x) \d{x}^{\alpha} e^{-\frac{\d{x}^2}{t}}\dd{x} &= t^{\frac{\alpha}{2}} \int_\Omega g(x)f\bigg(\frac{\d{x}}{\sqrt{t}}\bigg)\dd{x} \nonumber \\
        &=-t^{\frac{1+\alpha}{2}}\int_0^\infty f^\prime(w)w G\big(w\sqrt{t}\big)\dd{w} \nonumber\\
        &= -t^{\frac{1+\alpha}{2}}G\int_0^\infty f^\prime(w)w \dd{w} \nonumber\\
        &\hspace{5em}-t^{\frac{1+\alpha}{2}}\int_0^\infty f^\prime(w)w \Big(G\big(w\sqrt{t}\big)-G\Big)\dd{w}. \label{eq:proofBoundaryIntegralDecompositionAlphaLarger-1}
    \end{align}
    Again, the interchange of the integrals is justified by
    \begin{align}
        \int_0^\infty \int_\Omega \abs{g(x)f^\prime(w)\1_{\{\d{x}<w\sqrt{t}\}}}\dd{x}\dd{w} &\leq \norm{g}_\infty \int_0^\infty \abs{f^\prime(w)}\cdot \abs{\{\d{x}<w\sqrt{t}\}}\dd{w} \nonumber\\
        &\leq \norm{g}_\infty \sqrt{t}\theta_\Omega \hausdBoundary \int_0^\infty \abs{f^\prime(w)}w\dd{w} < \infty, \label{eq:interchangeOfIntsFubiniAlphaLargerMinusOne}
    \end{align}
    where the last assertion is again postponed until the end of the proof.
    We calculate the remaining integral explicitly by using integration by parts and the substitution $u = w^2$, to obtain
    \begin{align}
        \int_0^\infty  f^\prime(w) w\dd{w} = f(w)w\bigg\rvert_0^\infty -\int_0^\infty w^{\alpha} e^{-w^2}\dd{w} = -\frac{1}{2}\int_0^\infty u^{\frac{\alpha-1}{2}} e^{-u} \dd{u} = -\frac{1}{2}\Gamma\bigg(\frac{1+\alpha}{2}\bigg). \label{eq:finitenessOfGammaIntegral2}
    \end{align}
    
    Note that, as in the proof of Theorem \ref{thm:mainTermAlphaLarger1}, we have
    \begin{align*}
        \abs{f^\prime(w)w\Big(G(w\sqrt{t})- G\Big)} \leq \abs{f^\prime(w)}w\norm{g}_\infty \hausdBoundary (1+\theta_\Omega),
    \end{align*}
    for all $w\in (0,\infty)$. If $\alpha \in (0,\infty)$, $f(s)$ has a maximum at some $s^\ast \in (0,\infty)$, is monotonically increasing until $s^\ast$ and monotonically decreasing after $s^\ast$. Therefore,
    \begin{align*}
        \int_0^\infty \abs{f^\prime(w)} w\dd{w} = \int_0^{s^\ast} f^\prime(w)w\dd{w} - \int_{s^\ast}^\infty f^\prime(w) w\dd{w} = 2\int_0^{s^\ast} f^\prime(w)w\dd{w} - \int_0^\infty f^\prime(w) w\dd{w},
    \end{align*} 
    and both integrals are finite due to integrability around $0$ and \eqref{eq:finitenessOfGammaIntegral2}. If $\alpha \in (-1,0]$, $f(s)$ is monotonically decreasing, and therefore
    \begin{align*}
        \int_0^\infty \abs{f^\prime(w)}w \dd{w} = - \int_0^\infty f^\prime(w)w\dd{w},
    \end{align*}
    which is finite by \eqref{eq:finitenessOfGammaIntegral2}.
    Thus, in both cases, $\abs{f^\prime(w)}w\norm{g}_\infty \hausdBoundary (1+\theta_\Omega)$ is an integrable dominant.
    The second term in \eqref{eq:proofBoundaryIntegralDecompositionAlphaLarger-1} is therefore $o\Big(t^{\frac{1+\alpha}{2}}\Big)$ as $t\to 0$ by dominated convergence. This also proves the assertion in \eqref{eq:interchangeOfIntsFubiniAlphaLargerMinusOne}.
\end{proof}

\begin{theorem}
\label{thm:mainTermAlphaEquals1}
    Let $\Omega \subset \R^n$ be an open and bounded set with Lipschitz boundary. Let 
    \begin{itemize}
        \item $g\colon\overline{\Omega}\to \R$ be a bounded function that is continuous in a neighborhood of the boundary, or
        \item $g = \1_A$ for an open and bounded $A\subset \R^n$ such that $\hausdBoundarySet{\partial\Omega \cap \partial A} = 0$.
    \end{itemize}
    Then we have
    \begin{align*}
        \int_\Omega g(x)\d{x}^{-1}\bigg(1-e^{-\frac{\d{x}^2}{t}}\bigg)\dd{x} =  -\frac{\log{t}}{2} \int_{\partial\Omega} g(x) \dd\mathcal{H}^{n-1}(x) + o\Big(\abs{\log t}\Big), \qas t\to 0.
    \end{align*}
\end{theorem}

\begin{proof}
    As in the proof of Theorem \ref{thm:mainTermAlphaLarger1}, note that
    \begin{align*}
        \int_\Omega g(x)\d{x}^{-1}\bigg(1-e^{-\frac{\d{x}^2}{t}}\bigg)\dd{x} = t^{-\frac{1}{2}} \int_\Omega g(x)f\bigg(\frac{\d{x}}{\sqrt{t}}\bigg)\dd{x},
    \end{align*}
    with $f(s) = s^{-1}\big(1-e^{-s^2}\big)$.
    We write
    \begin{align}
        t^{-\frac{1}{2}} \int_\Omega g(x)f\bigg(\frac{\d{x}}{\sqrt{t}}\bigg)\dd{x} &= -t^{-\frac{1}{2}}\int_\Omega g(x) \int_0^\infty f^\prime(w) \1_{\{w>\frac{\d{x}}{\sqrt{t}}\}}\dd{w}\dd{x}. \label{eq:eq:proofAlpha1DoubleIntegral}
    \end{align}
    Let $\varrho = \sup_{x\in \Omega}\d{x} < \infty$, by the boundedness of $\Omega$. We will split the inner integral at $w = \varrho/\sqrt{t}$, that is,
    \begin{align*}
        \int_0^\infty f^\prime(w) \1_{\{w>\frac{\d{x}}{\sqrt{t}}\}}\dd{w} = \int_0^{\varrho/\sqrt{t}} f^\prime(w) \1_{\{w>\frac{\d{x}}{\sqrt{t}}\}}\dd{w} + \int_{\varrho/\sqrt{t}}^\infty f^\prime(w) \1_{\{w>\frac{\d{x}}{\sqrt{t}}\}}\dd{w}.
    \end{align*}
    
    Observe that for $w > \varrho/\sqrt{t}$ we have
    \begin{align*}
        \{x\in\Omega \mid \d{x} < w\sqrt t \} = \Omega,
    \end{align*}
    and therefore the expression \eqref{eq:eq:proofAlpha1DoubleIntegral} simplifies to
    \begin{align*}
    -t^{-\frac{1}{2}}\int_\Omega g(x) \int_{\varrho/\sqrt{t}}^\infty f^\prime(w) \1_{\{w>\frac{\d{x}}{\sqrt{t}}\}}\dd{w}\dd{x}  &= -t^{-\frac12}\int_\Omega g(x) \dd{x} \int_{\varrho/\sqrt{t}}^\infty f^\prime(w) \dd{w} \\
        &= t^{-\frac12}\int_\Omega g(x) \dd{x} f\bigg(\frac{\varrho}{\sqrt{t}}\bigg) \\
        &= \int_\Omega g(x) \dd{x} \frac{1}{\varrho} \bigg(1-e^{-\frac{\varrho^2}{t}}\bigg).
    \end{align*}
    Note that this term is bounded as $t\to 0$.

    For $0<w < \varrho/\sqrt{t}$, we let $G(s)$ and $G$ be as in the proof of Theorem \ref{thm:mainTermAlphaLarger1} and write
    \begin{align}
    &-t^{-\frac{1}{2}}\int_\Omega g(x) \int_0^{\varrho/\sqrt{t}} f^\prime(w) \1_{\{w>\frac{\d{x}}{\sqrt{t}}\}}\dd{w}\dd{x} \nonumber\\
        &\hspace{7em}=-t^{-\frac{1}{2}}\int_0^{\varrho/\sqrt{t}} f^\prime(w) \int_\Omega g(x) \1_{\{w > \frac{\d{x}}{\sqrt{t}}\}} \dd{x}\dd{w} \nonumber\\
        &\hspace{7em}= -\int_0^{\varrho/\sqrt{t}} f^\prime(w)w G(w\sqrt{t})\dd{w} \nonumber\\
        &\hspace{7em}= -G\int_0^{\varrho/\sqrt{t}} f^\prime(w)w\dd{w} - \int_0^{\varrho/\sqrt{t}} f^\prime(w)w \Big(G(w\sqrt{t})- G\Big)\dd{w}. \label{eq:proofAlpha1SplitGsLowerPart}
    \end{align}
    As in the previous proofs, we will justify the interchange of the integrals at the end of the proof.
    By integration by parts, the first integral is
    \begin{align*}
        -\int_0^{\varrho/\sqrt{t}} f^\prime(w)w\dd{w} = -\bigg(1-e^{-\frac{\varrho^2}{t}}\bigg) + \int_0^{\varrho/\sqrt{t}} \frac{1-e^{-w^2}}{w}\dd{w},
    \end{align*}
    and by the substitution $u = w^2$, we see
    \begin{align*}
        \int_0^{\varrho/\sqrt{t}} \frac{1-e^{-w^2}}{w}\dd{w} = \frac12 \int_0^{\varrho^2/t}\frac{1-e^{-u}}{u}\dd{u} = \frac12 \text{Ein}\bigg(\frac{\varrho^2}{t}\bigg),
    \end{align*}
    where $\text{Ein}$ is an exponential integral, see \cite[Chapter 5.1]{MR167642}. Note that $\text{Ein}(s) = \log{s} + \order{1}$ as $s\to \infty$. Thus, we obtain
    \begin{align*}
        -\int_0^{\varrho/\sqrt{t}} f^\prime(w)w\dd{w} = -\frac{1}{2}\log{t} + \order{1}, \qas t\to 0.
    \end{align*}

    To show that the second integral in \eqref{eq:proofAlpha1SplitGsLowerPart} is $o\big(\abs{\log t}\big)$, observe that
    \begin{align*}
        f^\prime(s) = - \frac{1-e^{-s^2}}{s^2} + 2e^{-s^2}. 
    \end{align*}
    Thus, there exists $R> 0$ such that $\abs{f^\prime(s)} \leq 2s^{-2}$ for $s\in (R,\infty)$. Take $t$ small enough such that $R< \varrho/\sqrt{t}$. Then, by splitting the second integral in \eqref{eq:proofAlpha1SplitGsLowerPart} at $R$, we obtain
    \begin{align*}
        \abs{\int_0^{\varrho/\sqrt{t}} f^\prime(w)w \Big(G(w\sqrt{t})- G\Big)\dd{w}} &\leq \int_0^R \abs{f^\prime(w)}w \abs{G(w\sqrt{t}) -G}\dd{w} \\
        &\hspace{5em}+ \int_R^{\varrho/\sqrt{t}} \abs{f^\prime(w)}w \abs{G(w\sqrt{t}) -G}\dd{w}.
    \end{align*}
    Since $f^\prime(s)s$ is bounded on $(0,R)$, $\abs{f^\prime(w)}w \norm{g}_\infty \hausdBoundary(1+\theta_\Omega)$ is an integrable dominant and by dominated convergence we have
    \begin{align*}
        \int_0^R \abs{f^\prime(w)}w \abs{G(w\sqrt{t}) -G}\dd{w} \to 0, \qas t\to 0.
    \end{align*}
    For the second term, we obtain, with the substitution $u=w\sqrt{t}$, and $\abs{f^\prime(s)} \leq 2s^{-2}$ for $s > R$ that
    \begin{align*}
        \int_R^{\varrho/\sqrt{t}} \abs{f^\prime(w)}w \abs{G(w\sqrt{t}) -G}\dd{w} \leq 2\int_R^{\varrho/\sqrt{t}}\frac{1}{w}\abs{G(w\sqrt{t}) -G}\dd{w} = 2 \int_{R\sqrt{t}}^{\varrho}\frac{1}{u}\abs{G(u) -G}\dd{u}.
    \end{align*}
    Let $\varepsilon > 0$. By the definition of $G(u)$ and $G$, and \eqref{eq:convOfGsToG}, there exists $\delta \in (R\sqrt{t}, \varrho)$ such that $\abs{G(u) - G} < \varepsilon$ for $u \in (0,\delta)$, by possibly taking a smaller $t$. Hence, we obtain
    \begin{align*}
        \int_{R\sqrt{t}}^{\varrho}\frac{1}{u}\abs{G(u) -G}\dd{u} &= \int_{R\sqrt{t}}^{\delta}\frac{1}{u}\abs{G(u) -G}\dd{u} + \int_{\delta}^{\varrho}\frac{1}{u}\abs{G(u) -G}\dd{u} \\
        &\leq \varepsilon \int_{R\sqrt{t}}^\delta \frac{\dd{u}}{u} + \int_\delta^\varrho\frac{1}{u}\abs{G(u) -G} \dd{u} \\
        &= -\frac{\varepsilon}{2} \log t + \order{1}, \qas t \to 0.
    \end{align*}
    As $\varepsilon> 0$ is chosen arbitrarily, we obtain that the second integral in \eqref{eq:proofAlpha1SplitGsLowerPart} is $o\big(\abs{\log t}\big)$.

    Lastly, to show that the interchange of the integrals is valid, we need to prove that
    \begin{align*}
        \int_0^{\varrho/\sqrt{t}} \int_\Omega \abs{g(x) f^\prime(w)}\1_{\{\d{x}<w\sqrt{t}\}}\dd{x}\dd{w} &\leq \norm{g}_\infty \int_0^{\varrho/\sqrt{t}} \abs{f^\prime(w)}\abs{\{\d{x}<w\sqrt{t}\}}\dd{w} < \infty.
    \end{align*}
    By Corollary \ref{cor:boundednessOfTubularRegions}, we have
    \begin{align*}
        \int_0^{\varrho/\sqrt{t}}\abs{f^\prime(w)}\abs{\{\d{x}<w\sqrt{t}\}}\dd{w} \leq \sqrt{t}\theta_\Omega \hausdBoundary \int_0^{\varrho/\sqrt{t}}\abs{f^\prime(w)}w\dd{w}.
    \end{align*}
    Furthermore, $\abs{f^\prime(s)} \leq 1$ for all $s>0$. Therefore, we obtain
    \begin{align*}
        \int_0^{\varrho/\sqrt{t}}\abs{f^\prime(w)}\abs{\{\d{x}<w\sqrt{t}\}}\dd{w} \leq \frac{1}{2}\theta_\Omega \hausdBoundary  \frac{\varrho^2}{\sqrt{t}} < \infty,
    \end{align*}
    which justifies the use of Fubini's theorem.
\end{proof}

\begin{remark}
    In the case of $g=1$, the boundary integral in the preceding theorems simplifies to
    \begin{align*}
        G = \int_{\partial\Omega} g(x)\dd\mathcal{H}^{n-1}(x) = \hausdBoundary.
    \end{align*}
    Furthermore, due to Corollary \ref{cor:behaviourOfThetaR}, we have 
        \begin{align*}
        G(s) - G = \frac{\abs{\{\d{x} < s\}}}{s} - \hausdBoundary = \vartheta_\Omega(s) \hausdBoundary,
    \end{align*}
    for $s>0$, and the lower-order terms in Theorems \ref{thm:mainTermAlphaLarger1} and \ref{thm:mainTermAlphaLess1} can be written explicitly in terms only depending on the domain $\Omega$, specifically as,
    \begin{align*}
        -t^{\frac{1+\alpha}{2}}\int_0^\infty f^\prime(w)w \Big(G(w\sqrt{t})- G\Big)\dd{w} = -t^{\frac{1+\alpha}{2}}\hausdBoundary \int_0^\infty \vartheta_\Omega(\sqrt{t} w) wf^\prime(w)\dd{w},
    \end{align*}
    for $\alpha\in (-3,\infty)\setminus\{-1\}$.
\end{remark}

\section{Heat kernel constructions}
\label{sec:heatKernelConstructions}

In this section, we will briefly summarize some well-known properties and inequalities regarding the heat kernel. Afterwards, we prove a specific approximation of the diagonal heat kernel for certain points near the boundary that will be used for estimates of the weighted heat trace in the good boundary region.

\subsection{Known inequalities for the heat kernel}

We will use the following results as basic ingredients to derive bounds in the rest of the paper. For proofs, we refer to the corresponding cited literature.

\begin{proposition}[Eq. (1.9.1) in \cite{MR1103113}]
\label{prop:boundednessOfHeatKernel}
    Let $\Omega \subset \R^n$ be open. Then
    \begin{align*}
        0 \leq p_\Omega(x,x;t) \leq (4\pi t)^{-n/2}, \quad &\forall (x,t) \in \Omega\times(0,\infty).
    \end{align*}
\end{proposition}
The following proposition is a quantitative version of the principle of not feeling the boundary and justifies the heuristic interpretation that far away from the boundary the heat kernel may be approximated by the heat kernel of the whole space. 
\begin{proposition}[Eq. (11) in \cite{MR640650}]
\label{prop:boundBulkKernel}
    Let $\Omega \subset \R^n$ be open. Then, for all $x \in \Omega$, 
    \begin{align}
        \abs{(4\pi t)^{-n/2}- p_\Omega(x,x;t)}\leq (4\pi t)^{-n/2} 2ne^{-\frac{\d{x}^2}{nt}}, \quad \forall t\in (0,\infty). \label{eq:notFeelTheBoundary}
    \end{align}
\end{proposition}

\begin{proposition}[Cor. 2.2 in \cite{MR1014655}]
\label{prop:heatkernelOfSubsetsAreOrdered}
    Let $\Omega_1,\, \Omega_2 \subset \R^n$ be open sets such that $\Omega_1 \subset \Omega_2$. Then
    \begin{align*}
        p_{\Omega_1}(x,x;t) \leq p_{\Omega_2}(x,x;t), \quad \forall (x,t) \in \Omega_1\times(0,\infty).
    \end{align*}
\end{proposition}

\subsection{Approximation of the heat kernel near good points of the boundary}

In the following, we will locally compare the heat kernel of $\Omega$ for points in the good part with the heat kernel of a half-space. For a half-space $\mathbb{H}$, the heat kernel satisfies
\begin{align}
    p_{\mathbb{H}}(x,x;t) = (4\pi t)^{-n/2} \bigg(1-e^{-\frac{\d[\mathbb{H}]{x}^2}{t}}\bigg),\quad \text{for all } (x,t)\in \mathbb{H}\times(0,\infty),\label{eq:heatKernelHalfPlane}
\end{align}
which can be obtained by using the heat kernel of $\R^n$ and applying the principle of mirror images.

For the proof of the relevant theorem of this section, we will use a probabilistic interpretation of the heat kernel, which is justified by \cite[Theorem 3]{SIMON1978268}, and is described as follows. For fixed $x\in \Omega$, let $A\colon [0, t] \to \R^n$ be a Brownian loop starting at $x$, that is, a Brownian motion such that $A(0) = A(t) = x$. Then 
\begin{align}
    (4\pi t)^{n/2}p_\Omega(x,x;t) = \Pr(A(\tau)\in \Omega ,\forall  \tau \in [0,t] \mid A(0) = A(t) = x). \label{eq:interpretationOfHeatKernelAsProbability}
\end{align}
That is, $p_\Omega(x,x;t)$ is the probability that a Brownian motion starting at time $0$ at $x$ and ending at time $t$ at $x$ does not leave the set $\Omega$, up to some known factor. From this representation, the Propositions \ref{prop:boundednessOfHeatKernel} and \ref{prop:heatkernelOfSubsetsAreOrdered} are immediate. 

Proposition \ref{prop:boundBulkKernel} can, using the above characterization, be formulated from a point of view, where we are interested in whether the Brownian loop leaves the domain or, equivalently, hits the boundary. In particular, for any open $\Omega$, Proposition \ref{prop:boundBulkKernel} implies that
\begin{align}
\begin{aligned}
    &\Pr(\exists \tau \in [0,t] \colon A(\tau)\notin \Omega \mid A(0) = A(t) = x) \\
    &\hspace{8em}= \Pr(\exists \tau \in [0,t]\colon A(\tau)\in \partial\Omega \mid A(0) = A(t) = x) \\
    &\hspace{8em}\leq 2ne^{-\frac{\d{x}^2}{nt}}, 
\end{aligned}
\label{eq:quantNotFeelingBoundary}        
\end{align}
for all $(x,t)\in \Omega\times(0,\infty)$. Note that this uses path continuity of the Brownian motion.

Using the probabilistic interpretation, we are able to obtain in the following theorem an upper and lower bound for the heat kernel for points where the nearby boundary can locally be squeezed between two hyperplanes. Afterwards, we apply the geometric constructions from Section \ref{sec:geomConstruction} to prove that this is indeed possible for good points in the domain.

\begin{theorem}
\label{thm:goodPartOfBoundary}
    Let $\Omega \subset \R^n$ be an open and bounded set with Lipschitz boundary. Let $\ell> 0$ and $x\in\Omega$ be such that $\d{x} < \ell$. Suppose that there exist $p\in\partial\Omega$, $\rho > 0$ and two open parallel half-spaces $\mathbb{H}_- \subset\mathbb{H}_+$ such that $\abs{x-p} < \ell$, $\rho = \frac{1}{2}\dist{\partial\mathbb{H}_+}{\partial\mathbb{H}_-}$ and
\begin{align*}
    \mathbb{H}_- \cap B_{2\ell}(p) \subset \Omega\cap B_{2\ell}(p) \subset \mathbb{H}_+ \cap B_{2\ell}(p).
\end{align*}
Then
\begin{align}
    1-e^{-\frac{(\d{x} - 2\rho)_+^2}{t}} - 2ne^{-\frac{\ell^2}{nt}} \leq (4\pi t)^{n/2}p_\Omega(x,x;t) \leq 1-e^{-\frac{(\d{x} + 2\rho)^2}{t}} + 2ne^{-\frac{\ell^2}{nt}} \label{eq:ineqForKernelInGoodPartGeneral}
\end{align}
for all $t > 0$.
\end{theorem}
\begin{remark}
    As $a_+ = \max\{0,a\}$ is the positive part of a quantity $a$, the lower bound in the previous theorem is trivial if $\rho \geq \frac{1}{2}\d{x}$.
\end{remark}

\begin{proof}%

We show the lower and upper bound separately through the analogy with a Brownian loop. Therefore, let $A\colon [0,t]\to \R^n$ be a Brownian loop centered at $x$. For simplicity, we write throughout the proof
\begin{align*}
    \Pr\nolimits_x(\cdot) = \Pr(\cdot \mid A(0) = A(t) = x)
\end{align*}
for the conditional probability.

For the lower bound, we can assume without loss of generality that $\d{x}> 2\rho$. This implies
\begin{align*}
    \dist{x}{\partial\mathbb{H}_+} \geq \d{x}> 2\rho,
\end{align*}
and therefore $x\in \mathbb{H}_-$. Thus, we obtain, using the inclusion $\mathbb{H}_-\cap B_{2\ell}(p) \subset \Omega\cap B_{2\ell}(p) \subset \Omega$, that
\begin{align*}
    (4\pi t)^{n/2} p_\Omega(x,x;t) &\geq \Pr\nolimits_x(A(\tau)\in \mathbb{H}_-\cap B_{2\ell}(p) ,\forall \tau\in[0,t]) \\
    &\geq \Pr\nolimits_x(A(\tau)\notin \partial\mathbb{H}_- ,\forall \tau\in[0,t]) \\
    &\hspace{1cm} - \Pr\nolimits_x(\exists\tau\in[0,t]\colon A(\tau)\in\mathbb{H}_-\cap \partial B_{2\ell}(p)).
\end{align*}
The first probability is now given by \eqref{eq:heatKernelHalfPlane} as
\begin{align*}
    \Pr\nolimits_x(A(\tau) \notin \partial\mathbb{H}_- ,\forall \tau\in[0,t]) = (4\pi t)^{n/2} p_{\mathbb{H}_-}(x,x;t) = 1-e^{-\frac{\dist{\partial\mathbb{H}_-}{x}^2}{t}},
\end{align*}
and for the second term, we find
\begin{align*}
    \Pr\nolimits_x(\exists\tau\in[0,t]\colon A(\tau)\in\mathbb{H}_-\cap \partial B_{2\ell}(p))\leq \Pr\nolimits_x(\exists\tau\in[0,t]\colon A(\tau)\in\partial B_{2\ell}(p)) \leq 2ne^{-\frac{\ell^2}{nt}},
\end{align*}
as $ \dist{x}{\partial B_{2\ell}(p)} > \ell$ and by using the estimate in \eqref{eq:quantNotFeelingBoundary} with domain $B_{2\ell}(p)$. By the geometric assumption, we find as before
\begin{align*}
    \dist{x}{\partial \mathbb{H}_-} = \dist{x}{\partial \mathbb{H}_+} - 2\rho \geq \d{x} - 2\rho
\end{align*}
from which the lower bound in \eqref{eq:ineqForKernelInGoodPartGeneral}, that is,
\begin{align*}
    (4\pi t)^{n/2} p_\Omega(x,x;t) &\geq 1-e^{-\frac{(\d{x} - 2\rho)^2_+}{t}} - 2ne^{-\frac{\ell^2}{nt}}
\end{align*}
follows.

For the upper bound, observe that
\begin{align*}
    \Pr\nolimits_x(A(\tau)\in \Omega,\forall \tau\in[0,t]) &= \Pr\nolimits_x(A(\tau)\notin \partial\Omega,\forall\tau\in[0,t]) \\
    &\leq \Pr\nolimits_x(A(\tau)\notin\partial\Omega\cap B_{2\ell}(p),\forall \tau\in[0,t]) .
\end{align*}
To obtain an upper bound for the preceding probability, we define the following events, which are all conditioned on the case that $A$ is a Brownian loop starting at $x$:
\begin{align*}
    A_1 &= \{A(\tau)\notin\partial\Omega\cap B_{2\ell}(p), \forall \tau\in[0,t]\}, \qand \\
    A_2 &= \{A(\tau)\notin\partial B_{2\ell}(p), \forall \tau\in[0,t]\}.
\end{align*}
Then we find
\begin{align*}
    A_1\cap A_2 &= \{A(\tau)\notin \partial(\Omega\cap B_{2\ell}(p)), \forall \tau\in[0,t]\} \\
    &= \{A(\tau) \in \Omega\cap B_{2\ell}(p), \forall \tau\in[0,t]\} \\
    &\subset \{A(\tau) \in \mathbb{H}_+\cap B_{2\ell}(p), \forall \tau\in[0,t]\}, \qand\\
    A_1\setminus A_2 &\subset A_2^c =  \{\exists \tau \in[0,t]\colon A(\tau)\in \partial B_{2\ell}(p)\}.
\end{align*}
Thus, using $\Pr\nolimits_x(A_1)=\Pr\nolimits_x(A_1\cap A_2) + \Pr\nolimits_x(A_1\setminus A_2)$, we can see
\begin{align}
    \Pr\nolimits_x(A_1) &\leq \Pr\nolimits_x(A(\tau)\in \mathbb{H}_+\cap B_{2\ell}(p), \forall \tau\in[0,t]) + \Pr\nolimits_x(\exists \tau \in[0,t]\colon A(\tau)\in \partial B_{2\ell}(p)). \label{eq:decompositionOfPrA1}
\end{align}
The first term can be bounded from above by the heat kernel of a half-space, that is,
\begin{align*}
    \Pr\nolimits_x(A(\tau)\in \mathbb{H}_+\cap B_{2\ell}(p), \forall \tau\in[0,t]) \leq \Pr\nolimits_x(A(\tau)\in \mathbb{H}_+, \forall \tau\in[0,t]) = 1-e^{-\frac{\dist{x}{\partial \mathbb{H}_+}^2}{t}}.
\end{align*}
If $x\notin \mathbb{H}_-$, then $\dist{x}{\partial\mathbb{H}_+} \leq 2\rho \leq \d{x} + 2\rho $. If $x\in\mathbb{H}_-$, then $\dist{x}{\partial\mathbb{H}_+} \leq \dist{x}{\partial\mathbb{H}_-} + 2\rho \leq \d{x} + 2\rho$. Hence, in both cases, we obtain the second term on the right-hand side of \eqref{eq:ineqForKernelInGoodPartGeneral}.
Finally, we can bound the second term in \eqref{eq:decompositionOfPrA1} as before by
\begin{align*}
    \Pr\nolimits_x(\exists \tau\in[0,t]\colon A(\tau)\in \partial B_{2\ell}(p)) \leq 2ne^{-\frac{\ell^2}{nt}},
\end{align*}
which implies the claim.
\end{proof}

\begin{figure}[htbp]
    \centering
    \begin{tikzpicture}[
    xscale=1.5,
    yscale=1.5,
    boundary/.style={thick},
    rough/.style={thick},
    ray/.style={->, thick},
    guide/.style={dashed, thin},
    point/.style={circle, fill, inner sep=1.2pt}
]

\coordinate (p) at (0,0);
\def\R{2.9}
\def\a{27.5}
\def\L{4.23}
\def\s{2.0}
\pgfmathsetmacro{\dH}{\s*sin(\a)}
\pgfmathsetmacro{\dHplus}{\dH+0.18}
\pgfmathsetmacro{\dHminus}{-\dH-0.18}

\fill[gray!25]
    (p) -- ++({-90-\a}:{\s/4})
    arc[start angle={-90-\a}, end angle={-90+\a}, radius={\s/4}]
    -- cycle;

\draw[boundary] (p) circle (\R);

\draw[line width=0.4pt] (p) circle ({\s});

\draw[thin] ($(p)+(-\L,0)$) -- ($(p)+(\L,0)$);
\draw[thin] ($(p)+(-\L,\dH)$) -- ($(p)+(\L,\dH)$);
\draw[thin] ($(p)+(-\L,-\dH)$) -- ($(p)+(\L,-\dH)$);

\draw[<->, thin] ($(p)+(3.25,0)$) -- ($(p)+(3.25,\dH)$)
    node[midway,right] {$\rho$};

\draw[<->, thin] ($(p)+(3.25,0)$) -- ($(p)+(3.25,-\dH)$)
    node[midway,right] {$\rho$};

\node at ($(p)+(3.75,0.22)$) {$\partial\mathbb{H}_{\,\,}$};
\node at ($(p)+(3.75,\dHplus)$) {$\partial\mathbb{H}_{+}$};
\node at ($(p)+(3.75,\dHminus)$) {$\partial\mathbb{H}_{-}$};

\draw[rough]
    (-2.92, 1.78)
    -- (-2.74, 1.50)
    -- (-2.56, 1.12)
    -- (-2.35, 0.35)
    -- (-2.18,-0.72)
    -- (-1.88,-0.18)
    -- (-1.62, 0.61)
    -- (-1.31,-0.46)
    -- (-1.04,-0.08)
    -- (-0.79, 0.28)
    -- (-0.55,-0.22)
    -- (-0.31, 0.10)
    -- (-0.20, 0.00)
    -- (p)
    -- (0.18, 0.00)
    -- (0.37,-0.08)
    -- (0.63, 0.29)
    -- (0.96,-0.41)
    -- (1.18,-0.18)
    -- (1.46, 0.52)
    -- (1.76,-0.23)
    -- (2.02,-0.72)
    -- (2.24, 0.45)
    -- (2.43,-0.92)
    -- (2.58,-1.28)
    -- (2.70,-1.53)
    -- (2.80,-1.90);

\node at (-1.8,-2.0) {$\Omega$};
\node at (-2.95,1.2) {$\partial\Omega$};

\node[point] at (p) {};
\node[above left=0.7pt] at (p) {$p$};

\coordinate (q) at ($(p)+(40:\R)$);
\draw[ray] (p) -- (q) node[pos=0.80, above left] {$r$};

\coordinate (qhalf) at ($(p)+(47:{\s})$);
\draw[ray] (p) -- (qhalf) node[pos=0.72, above left] {$2\ell$};

\draw[ray] (p) -- ++(0,-0.85) node[below] {$\nu(p)$};

\draw[guide] ($(p)+({180+\a}:\L)$) -- ($(p)+(\a:\L)$);
\draw[guide] ($(p)+({180-\a}:\L)$) -- ($(p)+({-\a}:\L)$);

\draw[guide] (p) -- ++({-90-\a}:{\s/4});
\draw[guide] (p) -- ++({-90+\a}:{\s/4});

\end{tikzpicture}
    \caption{Sketch of the half-space construction for points in $(\varepsilon,r)$-good regions of the domain. The shaded area represents all points near a fixed $p\in\partial\Omega$ for which Corollary \ref{cor:boundsForGoodPartOfBoundaryKernelOnly} is applicable.}
    \label{fig:halfPlanesWithSawtooth}
\end{figure}
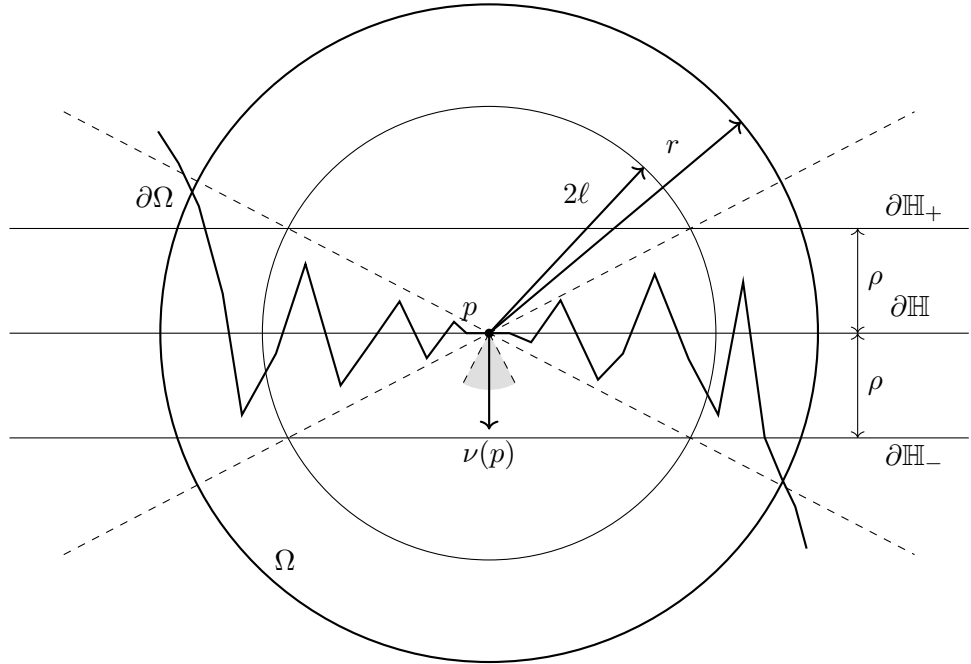

The preceding theorem is applied in the following corollary to the good part of the domain to obtain a uniform upper and lower bound for the heat kernel. Furthermore, a sketch of the geometric configuration is given in Figure \ref{fig:halfPlanesWithSawtooth}.
\begin{corollary}
\label{cor:boundsForGoodPartOfBoundaryKernelOnly}
    Let $\Omega \subset \R^n$ be an open and bounded set with Lipschitz boundary. Let $\varepsilon\in\big(0,\frac12\big], r > 0$ and $\ell \in \big(0, \frac{r}{2}\big)$. Then for all $(\varepsilon,r)$-good regions $\mathcal{G} \subset \Omega$ and for all $x\in \mathcal{G}$ with $\d{x} < \frac{\ell}{2}$, we have
        \begin{align}
    1-e^{-\frac{(\d{x} - 2\rho)_+^2}{t}} - 2ne^{-\frac{\ell^2}{nt}} \leq (4\pi t)^{n/2}p_\Omega(x,x;t) \leq 1-e^{-\frac{(\d{x} + 2\rho)^2}{t}} + 2ne^{-\frac{\ell^2}{nt}} \label{eq:ineqForKernelInGoodPart}
\end{align}
for all $t > 0$, with $\rho = 2\ell \varepsilon$.
\end{corollary}
\begin{proof}
    Let $x\in\mathcal{G}$ such that $\d{x}<\frac{\ell}{2}$. Then, there exists an $(\varepsilon,r)$-good $p\in \partial\Omega$ such that $x\in\Gamma_{\varepsilon,r}(p)$ and by Lemma \ref{lem:distOfGoodPointToItsSaddle} we have $\abs{x-p}<\ell$. According to Proposition \ref{prop:halfPlaneConstructionForGoodPart} with $s=2\ell\in(0,r)$, there exist two half-spaces $\mathbb{H}_+$ and $\mathbb{H}_-$ that satisfy the assumptions of Theorem \ref{thm:goodPartOfBoundary}, and hence \eqref{eq:ineqForKernelInGoodPartGeneral} holds for $x$ with $\rho = 2\ell\varepsilon$. 
\end{proof}

\section{Approximations for the weighted heat trace in different parts of the domain}
\label{sec:boundsForErrorTerms}
In this section, we analyze the integral
\begin{align}
    \int_D \d{x}^{\alpha}\abs{(4\pi t)^{n/2}p_\Omega(x,x;t) - 1+e^{-\frac{\d{x}^2}{t}}}\dd{x}
    \label{eq:weightedIntegralErrorTerm}
\end{align}
in the asymptotic limit $t\to 0$ for appropriate subsets $D\subset\Omega$. 
These subsets are chosen according to the following disjoint decomposition of $\Omega$.

\begin{definition}
\label{def:decompositionOfDomain}
    Let $r, \ell_1,\ell_2 > 0$ and $\varepsilon \in \Big(0,\frac14\Big]$ be such that $\ell_1 < \ell_2$. Let $\mathcal{G}_{\varepsilon,r}$ be a $(\varepsilon,r)$-good region according to Definition \ref{def:goodPointsOfDomain}. We define the following sets:
    \begin{itemize}
        \item The bulk of the domain: $\Omega_{bulk} = \{x\in\Omega \mid \d{x}\geq \ell_2\}$.
        \item The near-boundary part of the domain: $\Omega_{near} = \{x\in\Omega \mid \d{x}\leq \ell_1\}$.
        \item The good part of the domain: $\Omega_{good} = \{x\in\Omega \mid \ell_2 > \d{x} > \ell_1\}\cap \mathcal{G}_{\varepsilon,r}$.
        \item The bad part of the domain: $\Omega_{bad} = \{x\in\Omega \mid \ell_2 > \d{x} > \ell_1\}\setminus \mathcal{G}_{\varepsilon,r}$.
    \end{itemize} 
\end{definition}
\begin{remark}
    The sets in the preceding definition are chosen so that they form a disjoint decomposition of $\Omega$.
    In Section \ref{sec:badRegion} we will choose a family of sets $\{\mathcal{G}_{\varepsilon,r}\}_{\varepsilon,r}$ in a specific way. The interpretation of the different parameters is as follows:
    \begin{itemize}
        \item $\varepsilon$ corresponds to the angle of the cones in the approximation of the good boundary points,
        \item $r$ is the radius of the balls in which the cones centered at good boundary points enclose the boundary,
        \item $\ell_1$ denotes the size of the near-boundary region, and
        \item $\ell_2$ denotes the size of the bulk region.
    \end{itemize}
    We will later use $\ell_1 = \varepsilon_1\sqrt{t}$ and $\ell_2 = \frac{\sqrt{t}}{\varepsilon_2}$, so that $\varepsilon, \varepsilon_1, \varepsilon_2$ and $r$ are independent of $t$. Thus, the length scales of the sizes of the sets are given in terms of $\sqrt{t}$, which is the natural length scale for the heat equation.
\end{remark}

In the following sections, we will determine bounds for the integrals over the different parts of the domain. In fact, if $\alpha \in (\alpha^\ast(c_w),0)$, so that the weight near the boundary is singular, we will use Proposition \ref{prop:boundForBulkPartNegAlpha} for the bulk region, Propositions \ref{prop:boundForNearPartKernelOnly} and \ref{prop:boundForNearPartExpTerm} for the near boundary region, and Proposition \ref{prop:boundForBadPart} for the bad region to prove that each of the integrals in the following decompositions is already small enough to be insignificant for our claimed asymptotic expansions:
\begin{align}
&\begin{aligned}
    &\int_{\Omega_{bulk}} \d{x}^{\alpha}\abs{(4\pi t)^{n/2}p_\Omega(x,x;t) - 1+e^{-\frac{\d{x}^2}{t}}}\dd{x} \\
    &\hspace{4em}\leq \int_{\Omega_{bulk}} \d{x}^{\alpha}\abs{(4\pi t)^{n/2}p_\Omega(x,x;t) - 1}\dd{x} + \int_{\Omega_{bulk}} \d{x}^{\alpha}e^{-\frac{\d{x}^2}{t}}\dd{x}, \label{eq:errorDecompBulkPart}
\end{aligned}\\
&\begin{aligned}
    &\int_{\Omega_{near}} \d{x}^{\alpha}\abs{(4\pi t)^{n/2}p_\Omega(x,x;t) - 1+e^{-\frac{\d{x}^2}{t}}}\dd{x} \\
    &\hspace{4em}\leq (4\pi t)^{n/2}\int_{\Omega_{near}} \d{x}^{\alpha}p_\Omega(x,x;t)\dd{x} + \int_{\Omega_{near}} \d{x}^{\alpha}\abs{e^{-\frac{\d{x}^2}{t}}-1}\dd{x}, \qand \label{eq:errorDecompNearPart2}
\end{aligned}\\
     &\begin{aligned}
         &\int_{\Omega_{bad}} \d{x}^{\alpha}\abs{(4\pi t)^{n/2}p_\Omega(x,x;t) - 1+e^{-\frac{\d{x}^2}{t}}}\dd{x} \\
    &\hspace{4em}\leq (4\pi t)^{n/2}\int_{\Omega_{bad}} \d{x}^{\alpha}p_\Omega(x,x;t)\dd{x} + \int_{\Omega_{bad}} \d{x}^{\alpha}e^{-\frac{\d{x}^2}{t}}\dd{x} + \int_{\Omega_{bad}} \d{x}^{\alpha}\dd{x}. \nonumber
     \end{aligned}
\end{align}
This means that the heat kernel in the bulk is well approximated by the free heat kernel and in the near region we have a good boundary behavior. Only for the integral over the good region do all three terms in the integrand need to be considered together in order to show that the integral \eqref{eq:weightedIntegralErrorTerm} is a lower-order contribution. This is established in Proposition \ref{prop:boundForGoodPart}.

In the case of $\alpha \geq 0$, we will instead apply Proposition \ref{prop:boundForBulkPartPositiveAlpha} for the bulk and Proposition \ref{prop:nearBoundaryRegionIsSmallForPosAlpha} for the near boundary region, to conclude that each of the integrals in the decomposition \eqref{eq:errorDecompBulkPart} and 
\begin{align*}
\begin{aligned}
    &\int_{\Omega_{near}} \d{x}^{\alpha}\abs{(4\pi t)^{n/2}p_\Omega(x,x;t) - 1+e^{-\frac{\d{x}^2}{t}}}\dd{x} \\
    &\hspace{4em}\leq (4\pi t)^{n/2}\int_{\Omega_{near}} \d{x}^{\alpha}p_\Omega(x,x;t)\dd{x} + \int_{\Omega_{near}} \d{x}^{\alpha}e^{-\frac{\d{x}^2}{t}}\dd{x} + \int_{\Omega_{near}} \d{x}^{\alpha}\dd{x},
\end{aligned}
\end{align*}
respectively, are small enough as $t\to 0$.
For the good and bad regions, the same estimates as before apply when $\alpha \geq 0$.

\begin{remark}
    Throughout this section, we assume that $\Omega \subset \R^n$ is an open and bounded set with Lipschitz boundary, unless explicitly stated otherwise. Additionally, we assume that the parameters in Definition \ref{def:decompositionOfDomain} are fixed and we apply the decomposition of the domain according to Definition \ref{def:decompositionOfDomain} without restating it in the following theorems, propositions, and corollaries. 
\end{remark}

\subsection{The contribution of the bulk region}
In the bulk region, the Dirichlet heat kernel behaves like the free heat kernel, according to Kac's principle of not feeling the boundary. A quantitative version of this principle was stated earlier in Proposition \ref{prop:boundBulkKernel}.
Using that result, we find a bound for the integral in \eqref{eq:weightedIntegralErrorTerm} over the bulk region.

\begin{proposition}
\label{prop:boundForBulkPartNegAlpha}
    Let $\alpha < 0$. Then, for all $t > 0$,
    \begin{align*}
        &\int_{\Omega_{bulk}}  \d{x}^{\alpha}\abs{(4\pi t)^{n/2}p_\Omega(x,x;t) - 1 }\dd{x} \\
        &\hspace{5em} +\int_{\Omega_{bulk}}\d{x}^{\alpha}e^{-\frac{\d{x}^2}{t}}\dd{x} \leq \frac{(2n^{\frac{3}{2}}+1)\sqrt{\pi}}{2} \ell_2^{\alpha}\sqrt{t} \theta_\Omega \hausdBoundary.
    \end{align*}
\end{proposition}
\begin{remark}
    Since $\alpha <0$, the right-hand side in the preceding proposition is $o(t^{\frac{1+\alpha}{2}})$ if $\ell_2 \gg \sqrt{t}$. Hence, we see quantitatively that the heat kernel in the bulk region is sufficiently well approximated for our purposes by the free heat kernel. This is also commonly used in estimates for the unweighted heat trace. It meets the expectation that this approximation is also sufficient here, as the factor $\d{x}^{\alpha}$ for $\alpha < 0$ puts more weight on the region close to the boundary and less weight on the bulk region.
\end{remark}

\begin{proof}
    Applying Proposition \ref{prop:boundBulkKernel}, we immediately see that 
    \begin{align*}
        \int_{\Omega_{bulk}}  \d{x}^{\alpha}\abs{(4\pi t)^{n/2}p_\Omega(x,x;t) - 1 }\dd{x} \leq  2n\int_{\{\d{x}>\ell_2\}}\d{x}^{\alpha}e^{-\frac{\d{x}^2}{nt}}\dd{x}.
    \end{align*}
    Therefore, it suffices to consider exponential terms in the integral, that is, the integral
    \begin{align*}
        \int_{\{\d{x} > \ell_2\}} e^{-\d{x}^2/t} \dd{x}.
    \end{align*}
    Rewriting the exponential term as a Gaussian-type integral, we obtain
    \begin{align*}
     \int_{\{\d{x} > \ell_2\}} e^{-\d{x}^2/t} \dd{x} &= 2  \int_{\{\d{x} > \ell_2\}} \int_{\d{x}/\sqrt{t}}^\infty ue^{-u^2} \dd{u}\dd{x} \\
     &= 2\int_{\{\d{x} > \ell_2\}} \int_0^\infty \1_{\{u > \d{x}/\sqrt{t}\}}ue^{-u^2} \dd{u} \dd{x} \\
     &= 2\int_0^\infty \int_\Omega \1_{\{\ell_2 < \d{x} < \sqrt{t}u\}} \dd{x} ue^{-u^2} \dd{u}.
    \end{align*}
    Note that 
    \begin{align*}
        \{\ell_2 < \d{x}<\sqrt{t}u\} \subset \{\d{x} < \sqrt{t}u\}\cap \{\ell_2 < \sqrt{t}u\}
    \end{align*}
    and thus
    \begin{align*}
        \1_{\{\ell_2 < \d{x}<\sqrt{t}u\}} \leq \1_{\{\d{x} < \sqrt{t}u\}} \1_{\{\ell_2 < \sqrt{t}u\}}.
    \end{align*}
    By Corollary \ref{cor:boundednessOfTubularRegions}, we have the bound 
    \begin{align*}
        \abs{\{\d{x} < \sqrt{t}u\}} \leq  \sqrt{t}u \theta_\Omega\hausdBoundary.
    \end{align*}
    Using this bound, we find
    \begin{align*}
        2\int_0^\infty \int_\Omega \1_{\{\ell_2 < \d{x} < \sqrt{t}u\}} \dd{x} ue^{-u^2} \dd{u} &\leq  2\int_0^\infty  \abs{\{\d{x} < \sqrt{t}u\}} ue^{-u^2} \1_{\{\ell_2<\sqrt{t}u\}} \dd{u} \\
        &\leq 2 \sqrt{t} \theta_\Omega\hausdBoundary \int_0^\infty u^2 e^{-u^2}\1_{\{\ell_2<\sqrt{t}u\}} \dd{u} \\
        &\leq \frac{\sqrt{\pi t}}{2}\theta_\Omega  \hausdBoundary. %
    \end{align*}
    Finally, we get
    \begin{align*}
        &\int_{\Omega_{bulk}}  \d{x}^{\alpha}\abs{(4\pi t)^{n/2}p_\Omega(x,x;t) - 1 }\dd{x} +\int_{\Omega_{bulk}}\d{x}^{\alpha}e^{-\frac{\d{x}^2}{t}}\dd{x} \\
        &\hspace{5em}\leq 2n \ell_2^{\alpha}\int_{\Omega_{bulk}}e^{-\frac{\d{x}^2}{nt}}\dd{x} + \ell_2^{\alpha}\int_{\Omega_{bulk}}e^{-\frac{\d{x}^2}{t}}\dd{x} \\
        &\hspace{5em}\leq 2n\sqrt{n} \ell_2^{\alpha}\frac{\sqrt{\pi t}}{2}\theta_\Omega \hausdBoundary + \ell_2^{\alpha}\frac{\sqrt{\pi t}}{2}\theta_\Omega \hausdBoundary.%
    \end{align*}
\end{proof}

For nonnegative values of $\alpha$, we cannot simply pull out a factor of $\ell_2^\alpha$ as in the preceding proof. We have to be more careful in estimating the integrals. This is done in the following proposition.
\begin{proposition}
\label{prop:boundForBulkPartPositiveAlpha}
    Let $\alpha \geq 0$. Then, for all $t >0$, 
    \begin{align*}
        &\int_{\Omega_{bulk}}  \d{x}^{\alpha}\abs{(4\pi t)^{n/2}p_\Omega(x,x;t) - 1}\dd{x} \\
        &\hspace{4em} + \int_{\Omega_{bulk}}\d{x}^{\alpha}e^{-\frac{\d{x}^2}{t}}\dd{x} \leq C(\alpha)(2n^\frac{\alpha+3}{2}+1) t^{\frac{\alpha}{2}} \bigg(\sqrt{\frac{\pi t}{2}} + \ell_2\bigg) e^{-\frac{\ell_2^2}{2nt}} \theta_\Omega \hausdBoundary,
    \end{align*}
    where $C(\alpha) = \alpha^{\alpha/2}e^{-\alpha/2}$ if $\alpha >0$ and $C(0) = 1$.
\end{proposition}
\begin{proof}
    Applying Proposition \ref{prop:boundBulkKernel} and rescaling, it again suffices to consider an exponential term of the form
    \begin{align*}
        \int_{\Omega_{bulk}} \d{x}^\alpha e^{-\frac{\d{x}^2}{t}} \dd{x}.
    \end{align*}
    Let $f(s) = s^\alpha e^{-s^2}$ and note that $s^\alpha e^{-\frac{s^2}{2}} \leq C(\alpha)$ for all $s\in (0,\infty)$ with $C(\alpha) = \alpha^{\alpha/2}e^{-\alpha/2}$ if $\alpha > 0$ and $C(0)=1$. Thus, $f(s)\leq C(\alpha) e^{-s^2/2}$ for all $s\in(0,\infty)$ and we obtain
    \begin{align*}
        \int_{\Omega_{bulk}} \d{x}^\alpha e^{-\frac{\d{x}^2}{t}} \dd{x} &= t^\frac{\alpha}{2} \int_{\Omega_{bulk}}f\bigg(\frac{\d{x}}{\sqrt{t}}\bigg)\dd{x} \\
        &\leq C(\alpha)t^\frac{\alpha}{2}  \int_{\Omega_{bulk}} e^{-\frac{\d{x}^2}{2t}}\dd{x}.
    \end{align*}
    For the remaining integral, let $h(s)= e^{-s^2}$. Then we have
    \begin{align*}
        \int_{\Omega_{bulk}} e^{-\frac{\d{x}^2}{2t}} \dd{x} &= -\int_{\Omega_{bulk}} \int_{\frac{\d{x}}{\sqrt{2t}}}^\infty h^\prime(w)\dd{w}\dd{x} \\
        &= - \int_{0}^\infty \int_{\Omega}\1_{\{\ell_2 < \d{x} < w\sqrt{2t}\}}\dd{x}h^\prime(w)\dd{w},
    \end{align*}
    Note that the indicator function is zero if $w < \ell_2/\sqrt{2t}$. Thus, we can start the outer integral at $\ell_2/\sqrt{2t}$. Therefore, using Corollary \ref{cor:boundednessOfTubularRegions} and the fact that $h^\prime$ is negative on $(0,\infty)$, we obtain
    \begin{align*}
        - \int_{0}^\infty \int_{\Omega}\1_{\{\ell_2 < \d{x} < w\sqrt{2t}\}}\dd{x}h^\prime(w)\dd{w} \leq -\sqrt{2t}\theta_\Omega \hausdBoundary \int_{\frac{\ell_2}{\sqrt{2t}}}^\infty wh^\prime(w)\dd{w}.
    \end{align*}
    We get, after integration by parts,
    \begin{align*}
        -\int_{\frac{\ell_2}{\sqrt{2t}}}^\infty wh^\prime(w)\dd{w} = \frac{\ell_2}{\sqrt{2t}}e^{-\frac{\ell_2^2}{2t}} + \int_{\frac{\ell_2}{\sqrt{2t}}}^\infty e^{-w^2}\dd{w} = \frac{\ell_2}{\sqrt{2t}}e^{-\frac{\ell_2^2}{2t}} + \frac{\sqrt{\pi}}{2}\text{erfc}\bigg(\frac{\ell_2}{\sqrt{2t}}\bigg),
    \end{align*}
    where $\text{erfc}$ denotes the complementary error function. Using the elementary inequality $\text{erfc}(x) \leq e^{-x^2}$ for all $x\in (0,\infty)$, see \cite[7.1.13]{MR167642}, we obtain
    \begin{align*}
        -\int_{\frac{\ell_2}{\sqrt{2t}}}^\infty wh^\prime(w)\dd{w} \leq \bigg(\frac{\sqrt{\pi}}{2}+\frac{\ell_2}{\sqrt{2t}}\bigg) e^{-\frac{\ell_2^2}{2t}}.
    \end{align*}
    This implies
    \begin{align*}
        \int_{\Omega_{bulk}} \d{x}^\alpha e^{-\frac{\d{x}^2}{t}} \dd{x} &\leq C(\alpha) t^{\frac{\alpha+1}{2}} \bigg(\sqrt{\frac{\pi}{2}} + \frac{\ell_2}{\sqrt{t}}\bigg)e^{-\frac{\ell_2^2}{2t}} \theta_\Omega \hausdBoundary \\
        &\leq C(\alpha) t^{\frac{\alpha+1}{2}} \bigg(\sqrt{\frac{\pi}{2}} + \frac{\ell_2}{\sqrt{t}}\bigg)e^{-\frac{\ell_2^2}{2nt}} \theta_\Omega \hausdBoundary.
    \end{align*}
    The general statement now follows by adding the rescaled bound
    \begin{align*}
        2n \int_{\Omega_{bulk}} \d{x}^\alpha e^{-\frac{\d{x}^2}{nt}} \dd{x} &\leq 2C(\alpha) t^{\frac{\alpha+1}{2}}n^{\frac{\alpha+3}{2}} \bigg(\sqrt{\frac{\pi}{2}} + \frac{\ell_2}{\sqrt{nt}}\bigg)e^{-\frac{\ell_2^2}{2nt}} \theta_\Omega \hausdBoundary \\
        &\leq 2C(\alpha) t^{\frac{\alpha+1}{2}}n^{\frac{\alpha+3}{2}} \bigg(\sqrt{\frac{\pi}{2}} + \frac{\ell_2}{\sqrt{t}}\bigg)e^{-\frac{\ell_2^2}{2nt}} \theta_\Omega \hausdBoundary.
    \end{align*}
\end{proof}

\subsection{The contribution of the near boundary region}

In the following, we will show that the contribution of the region closest to the boundary, that is $\Omega_{near}$, is negligible. This statement is straightforward for $\alpha \geq 0$, which we will present first. For $\alpha \in (\alpha^\ast(c_w),0)$, we will then follow the strategy of \cite{MR2480958} to quantify the decay of the heat kernel near the boundary. %

\begin{proposition}
\label{prop:nearBoundaryRegionIsSmallForPosAlpha}
    Let $\alpha \geq 0$. Then, for all $t > 0$, %
    \begin{align*}
        \int_{\Omega_{near}} \d{x}^{\alpha}\abs{(4\pi t)^{n/2}p_\Omega(x,x;t) - 1+e^{-\frac{\d{x}^2}{t}}}\dd{x} \leq 3\ell_1^{\alpha+1} \theta_\Omega \hausdBoundary. 
    \end{align*}
\end{proposition}
\begin{proof}
    As $\alpha \geq 0$ and each of the terms inside the absolute value is bounded by $1$, we see
    \begin{align*}
        \int_{\Omega_{near}} \d{x}^{\alpha}\abs{(4\pi t)^{n/2}p_\Omega(x,x;t) - 1+e^{-\frac{\d{x}^2}{t}}}\dd{x} \leq 3\ell_1^\alpha \abs{\Omega_{near}}. 
    \end{align*}
    The claim now follows by using Corollary \ref{cor:boundednessOfTubularRegions}.
\end{proof}

For negative $\alpha$, the strategy is to show that, for our purposes, it suffices to bound 
\begin{align}
    \int_{\Omega_{near}}u^2{\d{x}}^{\alpha}\dd{x}, \label{eq:nearRegionuIntegral}
\end{align}
where $u$ is the weak solution to the initial value problem
\begin{align}
\begin{cases}
    (\partial_t - \Delta) u(x;t) &= 0, \quad \forall (x,t)\in \Omega\times(0,\infty), \\
    u(x;t) &= 0, \quad \forall (x,t)\in \partial\Omega\times(0,\infty), \\
    u(x;0) &= 1, \quad \forall x \in \Omega.
\end{cases}
\label{eq:def_HeatContent}
\end{align}
Afterwards, we prove the relevant bound for \eqref{eq:nearRegionuIntegral} under suitable assumptions on the domain $\Omega$ and apply it to the weighted heat trace.

By the initial value condition of $u(x;t)$, we find the relation
    \begin{align}
        u(x;t) = \int_\Omega p_\Omega(x,y;t)\dd{y}. \label{eq:conditionHeatContent}
    \end{align}
Furthermore, the probabilistic interpretation allows us to understand $u$ as 
\begin{align}
    u(x;t) = \Pr\nolimits_x(\tau_\Omega > t), \label{eq:repOfHeatContent}
\end{align}
where $\tau_\Omega$ is the first time that a Brownian motion starting at $x$ exits the domain $\Omega$, see \cite[Equation~(1.6)]{MR1022816}. Using this representation, it is easy to see that $t\mapsto u(\cdot;t)$ is monotonically decreasing and $u(x;t) \in [0,1]$ for all $(x,t)\in \Omega\times[0,\infty)$.

\begin{proposition}
\label{prop:heatTraceBoundByHeatContentBound}
    We have, for all $x\in\Omega$, that
    \begin{align*}
        p_\Omega(x,x;t) \leq \bigg(\frac{3}{4\pi t}\bigg)^\frac{n}{2} u(x;t/3)^2, \quad \forall t\in(0,\infty).
    \end{align*}
\end{proposition}
\begin{proof}
    The proof closely follows \cite[Prop. 1]{MR2480958}.
    First, note that for all $y,z\in\Omega$ and $t > 0$ we have, by using the representation \eqref{eq:heatKernelRepInBasis}, that
    \begin{align*}
        p_\Omega(y,z;t) = \sum_{k=1}^\infty e^{-\lambda_k \frac{t}{2}} u_k(y) \cdot e^{-\lambda_k \frac{t}{2}}u_k(z) \leq \left(\sum_{k=1}^\infty e^{-\lambda_k t}u_k(y)^2\right)^\frac12\left(\sum_{k=1}^\infty e^{-\lambda_k t}u_k(z)^2\right)^\frac12 \leq (4\pi t)^{-\frac{n}{2}},
    \end{align*}
    where we used the Cauchy-Schwarz inequality and Proposition \ref{prop:boundednessOfHeatKernel}.
    Therefore, by the semigroup property, we have
    \begin{align*}
        p_\Omega(x,x;t) &= \int_\Omega \int_\Omega p_\Omega(x,y;t/3) p_\Omega(y,z;t/3) p_\Omega(z,x;t/3) \dd{z}\dd{y} \\
        &\leq \bigg(4\pi \frac{t}{3}\bigg)^{-\frac{n}{2}} \int_\Omega \int_\Omega p_\Omega(x,y;t/3) p_\Omega(z,x;t/3) \dd{z}\dd{y}\\
        &= \bigg(\frac{3}{4\pi t}\bigg)^\frac{n}{2} \left(\int_\Omega p_\Omega(x,y;t/3) \dd{y}\right)^2,
    \end{align*}
    where we used the symmetry of $p_\Omega$ in the last step. With \eqref{eq:conditionHeatContent}, we obtain
    \begin{align*}
        p_\Omega(x,x;t) \leq \bigg(\frac{3}{4\pi t}\bigg)^\frac{n}{2} u(x;t/3)^2,
    \end{align*}
    which is the desired result.
\end{proof}

To find a bound for the weighted integral over $u$ in \eqref{eq:nearRegionuIntegral}, we use a Hardy inequality, defined as follows.
\begin{definition}
\label{def:weakHardyInequality}
    We say that a Hardy inequality holds for $\Omega$ or $\Omega$ has the Hardy property if there exist nonnegative constants $c$ and $\Lambda$ such that
    \begin{align}
        \int_\Omega \frac{\abs{f(x)}^2}{{\d{x}}^2} \dd{x} \leq c^2 \int_\Omega \abs{\nabla f(x)}^2 \dd{x} + \Lambda \int_\Omega \abs{f(x)}^2 \dd{x}, \label{eq:weakHardyInequality}
    \end{align}
    for all $f\in C_c^\infty(\Omega)$. If \eqref{eq:weakHardyInequality} holds with $\Lambda = 0$, then the infimum over all possible $c$ is called the strong Hardy constant. The infimum over all possible $c$ such that there exists a $\Lambda < \infty$ such that \eqref{eq:weakHardyInequality} holds is called the weak Hardy constant.
\end{definition}
\begin{remark}
\label{rem:weakHardyConstForLipschitz}
    It is evident from the definition that the weak Hardy constant is well defined if the strong Hardy constant is. In that case, the weak Hardy constant is always less than or equal to the strong Hardy constant. It is worth noting that any bounded Lipschitz domain has the Hardy property. This follows from the existence of a strong Hardy constant, see \cite[Theorem 8.4]{MR802206} with $\varepsilon = 0$. A more extensive treatment of Hardy inequalities and Hardy constants is given, for example, in \cite{MR1366614}.
\end{remark}

\begin{remark}
    Note that in the definition of the Hardy inequality we demand $f\in C_c^\infty(\Omega)$. Due to a density argument, if $\Omega$ has the Hardy property, \eqref{eq:weakHardyInequality} also holds for $f\in H^1_0(\Omega)$.
\end{remark}

In the following theorem, we will prove a similar result as \cite[Theorem 2]{MR2480958} under the assumption of a weak Hardy inequality instead of a strong Hardy inequality. This allows a larger range of parameters, with the downside of picking up two extra terms. We will see later that these additional terms are of lower order in $t$. For convenience, from now on, we write $\gamma = -\alpha > 0$ in the relevant theorems concerning \eqref{eq:nearRegionuIntegral}.

\begin{theorem}
\label{thm:weightedIntegralOverHeatContent}
    Let $\Omega\subset \R^n$ be open and assume that $\Omega$ has the Hardy property with weak Hardy constant $c_w$. 
    If 
    \begin{align*}
        2 < \gamma < \min{\left(2(1+c_w^{-1}), 3\right)} = -\alpha^\ast(c_w),
    \end{align*}
    then, for any pair of constants $c\geq c_w$ and $\Lambda\geq0$ such that \eqref{eq:weakHardyInequality} is true and $c < (\gamma/2-1)^{-1}$, it holds for every $t>0$, that
    \begin{align*}
        \int_\Omega u(x;t)^2{\d{x}}^{-\gamma}\dd{x} &\leq 4c^\gamma \left(1-c^2(1-\gamma/2)^2\right)^{-1} t^{-\gamma/2}\int_\Omega (1-u)\dd{x} \\
        &\hspace{3em}+ 4c^2\left(1-c^2(1-\gamma/2)^2\right)^{-1}\Lambda^\frac{\gamma-2}{2}\abs{\Omega}^\frac{\gamma-2}{2}t^{-1}\Big(\int_\Omega (1-u)\dd{x}\Big)^\frac{4-\gamma}{2} \\
        &\hspace{3em}+ \left(1-c^2(1-\gamma/2)^2\right)^{-1} \Lambda \int_\Omega \d{x}^{2-\gamma}\dd{x}.%
    \end{align*}
\end{theorem}

\begin{remark}
    If $\Omega$ has Lipschitz boundary, then the integral $\int_\Omega \d{x}^{2-\gamma}\dd{x}$ is finite for $\gamma < 3$, see Lemma \ref{lem:mainTermAlphaLess1ExtraTerm}. The explicit restriction to $\gamma < 3$ is necessary, as, to our knowledge, there is no known universal lower bound for the weak Hardy constant $c_w$, apart from $c_w \geq 0$ by the definition.
\end{remark}

\begin{proof}
    We will essentially follow the proof of \cite[Theorem 2]{MR2480958}, but we will obtain two additional terms, which arise due to the application of the weak Hardy inequality. 
    
    Note that, if the integral $\int_\Omega \d{x}^{2-\gamma}\dd{x}$ is not finite, then the statement is trivial. Therefore, we can restrict to the case in which the integral is finite. Let $c$ and $\Lambda$ be as in the assumption. Then $\gamma < 2(1+c^{-1})$.
    Let $v\colon \Omega \to [0,\infty)$ be a bounded function with bounded weak derivative $\nabla v$. Then
    \begin{align*}
        -\dv{t}\int_\Omega(uv)^2\dd{x} &= -2\int_\Omega uv^2\pdv{u}{t}\dd{x} = -2\int_\Omega uv^2\Delta u \dd{x} \\
        &= 2\int_\Omega \big(v^2\abs{\nabla u}^2 + 2uv\nabla u\cdot\nabla v\big)\dd{x},
 \end{align*}
 which follows from integration by parts and using \eqref{eq:def_HeatContent}. With
 \begin{align*}
     \int_\Omega \abs{\nabla(uv)}^2\dd{x} = \int_\Omega \big(v^2\abs{\nabla u}^2+2uv\nabla u\cdot\nabla v + u^2\abs{\nabla v}^2\big)\dd{x}
 \end{align*}
 we rewrite this as
 \begin{align*}
     -\dv{t}\int_\Omega(uv)^2\dd{x} = 2\int_\Omega\abs{\nabla(uv)}^2\dd{x} - 2\int_\Omega u^2\abs{\nabla v}^2\dd{x}.
 \end{align*}
 In the first term, we apply the weak Hardy inequality \eqref{eq:weakHardyInequality} to obtain
 \begin{align}
     -\dv{t}\int_\Omega(uv)^2\dd{x} \geq \frac{2}{c^2}\int_\Omega \frac{(uv)^2}{{\d{x}}^2}\dd{x} - \frac{2\Lambda}{c^2}\int_\Omega (uv)^2\dd{x} - 2\int_\Omega u^2\abs{\nabla v}^2\dd{x} \label{eq:proofWeakHardyStep1}
 \end{align}
 Let $v_n\colon \Omega \to [0,\infty)$ with $ v_n(x) = \min(n, \d{x}^{\frac{2-\gamma}{2}})$. Clearly, $v_n \nearrow {\d{x}}^{\frac{2-\gamma}{2}}$ and each $v_n$ is bounded. Furthermore, $\nabla v_n$ exists and 
 \begin{align*}
     \abs{\nabla v_n(x)}^2  
     =\begin{cases}
         \left(\frac{2-\gamma}{2}\right)^2 \d{x}^{-\gamma}= \left(\frac{2-\gamma}{2}\right)^2 v_n^2\d{x}^{-2}, & \d{x}^{\frac{2-\gamma}{2}}<n,\\
          0, & \d{x}^{\frac{2-\gamma}{2}}>n,
     \end{cases}
 \end{align*}
 holds almost everywhere.
 In particular, $\abs{\nabla v_n}$ is bounded by $\frac{\gamma-2}{2}n^{\frac{\gamma}{\gamma-2}}$. Thus, we can apply \eqref{eq:proofWeakHardyStep1} with $v_n$ instead of $v$. Hence,
 \begin{align*}
     -\dv{t}\int_\Omega(uv_n)^2\dd{x} &\geq \frac{2}{c^2}\int_\Omega \frac{(uv_n)^2}{{\d{x}}^2}\dd{x} - \frac{2\Lambda}{c^2}\int_\Omega (uv_n)^2\dd{x} - 2\int_\Omega u^2\abs{\nabla v_n}^2\dd{x} \\
     &\geq 2\left(c^{-2} - \left(1-\frac{\gamma}{2}\right)^2\right)\int_\Omega \frac{(uv_n)^2}{{\d{x}}^2}\dd{x} - \frac{2\Lambda}{c^2}\int_\Omega (uv_n)^2\dd{x}.
 \end{align*}
 Note that $t\mapsto u(\cdot;t)$ is monotonically decreasing and $u \geq 0$ by the probabilistic representation \eqref{eq:repOfHeatContent}, so the left-hand side is positive and increasing in $n$, which means we can write
 \begin{align*}
     -\dv{t}\int_\Omega u^2{\d{x}}^{2-\gamma}\dd{x} \geq -\dv{t}\int_\Omega(uv_n)^2\dd{x} \geq 2\left(c^{-2} - \left(1-\frac{\gamma}{2}\right)^2\right)\int_\Omega \frac{(uv_n)^2}{{\d{x}}^2}\dd{x} - \frac{2\Lambda}{c^2}\int_\Omega (uv_n)^2\dd{x}
 \end{align*}
 for any $n\in \N$. Taking the limit $n\to\infty$, we conclude by monotone convergence
 \begin{align*}
     -\dv{t}\int_\Omega u^2{\d{x}}^{2-\gamma}\dd{x} \geq 2\left(c^{-2} - \left(1-\frac{\gamma}{2}\right)^2\right)\int_\Omega \frac{u^2}{{\d{x}}^\gamma}\dd{x} - \frac{2\Lambda}{c^2}\int_\Omega u^2 {\d{x}}^{2-\gamma}\dd{x},
 \end{align*}
 which implies the upper bound
 \begin{align}
 \begin{aligned}
     \int_\Omega \frac{u^2}{{\d{x}}^\gamma}\dd{x} &\leq -2^{-1}\left(c^{-2} - \left(1-\frac{\gamma}{2}\right)^2\right)^{-1}\dv{t}\int_\Omega u^2{\d{x}}^{2-\gamma}\dd{x} \\
     &\hspace{3em} + \frac{\Lambda}{c^2}\left(c^{-2} - \left(1-\frac{\gamma}{2}\right)^2\right)^{-1}\int_\Omega u^2{\d{x}}^{2-\gamma}\dd{x}. 
 \end{aligned}
    \label{eq:proofWeakHardyUpperBoundBeforeIntegration}
 \end{align}
 Now we integrate in $t$ over the interval $[t,2t]$ and use that $u$ is decreasing on the left hand side of~\eqref{eq:proofWeakHardyUpperBoundBeforeIntegration} to get
 \begin{align}
 \begin{aligned}
     t\int_\Omega \frac{u(x;2t)^2}{{\d{x}}^\gamma}\dd{x} &\leq 2^{-1}\left(c^{-2} - \left(1-\frac{\gamma}{2}\right)^2\right)^{-1}\int_\Omega \left(u(x;t)^2-u(x;2t)^2\right){\d{x}}^{2-\gamma}\dd{x} \\
     &\hspace{3em} + \frac{\Lambda}{c^2}\left(c^{-2} - \left(1-\frac{\gamma}{2}\right)^2\right)^{-1}\int_t^{2t}\int_\Omega u(x;\tau)^2{\d{x}}^{2-\gamma}\dd{x}\dd{\tau}.
 \end{aligned}
     \label{eq:proofweakHardyNearSol}
 \end{align}
 
 For the second term, we again use the monotonicity of $t\mapsto u(\cdot;t)$ to get
 \begin{align}
     \int_t^{2t}\int_\Omega u(x;\tau)^2{\d{x}}^{2-\gamma}\dd{x}\dd{\tau} \leq t \int_\Omega u(x;t)^2{\d{x}}^{2-\gamma}\dd{x} \leq t\int_\Omega {\d{x}}^{2-\gamma}\dd{x} = C_{\Omega,\gamma} t, \label{eq:proofweakHardyLastTerm}
 \end{align}
 with constant $C_{\Omega,\gamma} = \int_\Omega \d{x}^{2-\gamma}\dd{x}$ being finite.

For the first term in \eqref{eq:proofweakHardyNearSol}, observe that
\begin{align*}
    u(x;t)^2-u(x;2t)^2 = \big(u(x;t)+u(x;2t)\big) \big(u(x;t)-u(x;2t)\big) \leq 2 u(x;t) \big(1-u(x;2t)\big),
\end{align*}
where we used that $u$ is decreasing in $t$ and $u\leq 1$. Hence,
\begin{align*}
    &\int_\Omega \big(u(x;t)^2-u(x;2t)^2\big) \d{x}^{2-\gamma}\dd{x} \\
    &\hspace{8em}\leq 2\int_\Omega u(x;t)\big(1-u(x;2t)\big)\d{x}^{2-\gamma}\dd{x} \\
    &\hspace{8em}\leq 2\Big(\int_\Omega u(x;t)^\frac{2}{\gamma-2}\d{x}^{-2}\dd{x}\Big)^\frac{\gamma-2}{2}\Big(\int_\Omega \big(1-u(x;2t)\big)^\frac{2}{4-\gamma}\dd{x}\Big)^\frac{4-\gamma}{2},
\end{align*}
by Hölder's inequality. Applying Bernoulli's inequality in the last factor, which is possible because $0\leq u \leq 1$ and $\frac{2}{4-\gamma} > 1$, yields
\begin{align}
\begin{aligned}
    &\int_\Omega \big(u(x;t)^2-u(x;2t)^2\big) \d{x}^{2-\gamma}\dd{x} \\
    &\hspace{5em}\leq 2^{(6-\gamma)/2} (4-\gamma)^\frac{\gamma-4}{2}\Big(\int_\Omega u(x;t)^\frac{2}{\gamma-2}\d{x}^{-2}\dd{x}\Big)^\frac{\gamma-2}{2}\Big(\int_\Omega \big(1-u(x;2t)\big)\dd{x}\Big)^\frac{4-\gamma}{2}.
\end{aligned}
     \label{eq:proofWeakHardyStep2}
\end{align}
To bound the first factor, set $p=\frac{2}{\gamma-2}$. Then
\begin{align*}
    -\dv{t} \int_\Omega u(x;t)^p\dd{x} = -p\int_\Omega u(x;t)^{p-1} \pdv{u(x;t)}{t} \dd{x} &= -p\int_\Omega u(x;t)^{p-1} \Delta u(x;t) \dd{x} \\
    &= \frac{4(p-1)}{p} \int_\Omega \abs{\nabla(u^\frac{p}{2})}^2 \dd{x}.
\end{align*}
Applying the weak Hardy inequality \eqref{eq:weakHardyInequality}, we obtain
\begin{align*}
    -\dv{t} \int_\Omega u(x;t)^p\dd{x} \geq \frac{4(p-1)}{pc^2}\left[\int_\Omega \frac{u^p}{\d{x}^2}\dd{x} - \Lambda \int_\Omega u^p\dd{x}\right].
\end{align*}
Finally, rearranging and integrating in time from $0$ to $t$ yields
\begin{align}
    t\int_\Omega u(x;t)^p\d{x}^{-2}\dd{x} &\leq \int_0^t \int_\Omega u(x;s)^p\d{x}^{-2}\dd{x} \nonumber\\
    &\leq \frac{pc^2}{4(p-1)} \int_\Omega \big(1-u(x;t)^p\big)\dd{x} + \Lambda \int_0^t \int_\Omega u(x;s)^p\dd{x}\dd{s} \nonumber \\
    &\leq \frac{p^2c^2}{4(p-1)} \int_\Omega\big(1-u(x;t)\big)\dd{x} + \Lambda \abs{\Omega}t, \label{eq:generOfvdBLemma4}
\end{align}
where we used again that $0\leq u\leq 1$ and Bernoulli's inequality. Therefore, we obtain for the first integral on the right-hand side of \eqref{eq:proofWeakHardyStep2} that
\begin{align}
    \Big(\int_\Omega u(x;t)^\frac{2}{\gamma-2}\d{x}^{-2}\dd{x}\Big)^\frac{\gamma-2}{2} &\leq \Big(t^{-1}\frac{p^2c^2}{4(p-1)} \int_\Omega\big(1-u(x;t)\big)\dd{x} + \Lambda \abs{\Omega}\Big)^\frac{\gamma-2}{2} \nonumber\\
    &\leq \Big(t^{-1}\frac{p^2c^2}{4(p-1)} \int_\Omega\big(1-u(x;2t)\big)\dd{x}\Big)^\frac{\gamma-2}{2} + \Lambda^\frac{\gamma-2}{2} \abs{\Omega}^\frac{\gamma-2}{2}, \label{eq:proofWeakHardyStep3}
\end{align}
where the last inequality holds as $0\leq \frac{\gamma-2}{2}\leq 1$ and $u$ is decreasing in time. Note that $$\frac{p^2}{4(p-1)} = \big((4-\gamma)(\gamma-2)\big)^{-1}.$$

Therefore, plugging \eqref{eq:proofWeakHardyStep3} into \eqref{eq:proofWeakHardyStep2}, and combining this result with \eqref{eq:proofweakHardyLastTerm} in \eqref{eq:proofweakHardyNearSol}, we obtain
\begin{align*}
     t\int_\Omega \frac{u(x;2t)^2}{{\d{x}}^\gamma}\dd{x} &\leq \left(c^{-2} - \left(1-\frac{\gamma}{2}\right)^2\right)^{-1}2^\frac{4-\gamma}{2} (4-\gamma)^{-1}t^\frac{2-\gamma}{2}c^{\gamma-2}(\gamma-2)^\frac{2-\gamma}{2} \int_\Omega \big(1-u(x;2t)\big)\dd{x} \\
     &\hspace{2em}+ \left(c^{-2} - \left(1-\frac{\gamma}{2}\right)^2\right)^{-1} \bigg(\frac{2}{4-\gamma}\bigg)^\frac{4-\gamma}{2} \Lambda^\frac{\gamma-2}{2} \abs{\Omega}^\frac{\gamma-2}{2}\bigg(\int_\Omega \big(1-u(x;2t)\big)\dd{x}\bigg)^\frac{4-\gamma}{2} \\
     &\hspace{2em} + \frac{\Lambda}{c^2}\left(c^{-2} - \left(1-\frac{\gamma}{2}\right)^2\right)^{-1}C_{\Omega,\gamma} t.
\end{align*}
Finally, using $(4-\gamma)^{-1}(\gamma-2)^{\frac{2-\gamma}{2}} \leq 1$ in the first term and $2^{\frac{4-\gamma}{2}}(4-\gamma)^{\frac{\gamma-4}{2}} \leq 2$ in the second term, as $2<\gamma<3$, the claim follows.

\end{proof}

\begin{remark}
    Equation \eqref{eq:generOfvdBLemma4} is a generalization of Equation (39) in \cite[Lemma 4]{MR2480958} to accommodate a weak Hardy inequality instead of a strong Hardy inequality. 
\end{remark}

As a corollary of the preceding theorem, we obtain an upper bound for the weighted integral of the solution of \eqref{eq:def_HeatContent}, which is independent of the solution itself.
\begin{corollary}
\label{cor:boundForWeightedIntegralHeatContent}
   Suppose the assumptions of Theorem \ref{thm:weightedIntegralOverHeatContent} are satisfied. Then for any pair of constants $c, \Lambda\geq0$, such that \eqref{eq:weakHardyInequality} is true and $c<(\gamma/2-1)^{-1}$, we have for all $t > 0$, that
    \begin{align*}
        \int_\Omega \frac{u(x;t)^2}{\d{x}^{\gamma}} \dd{x} \leq \Big(1-c^2(1-\gamma/2)^2\Big)^{-1}&\Bigg(2^\frac{n+7}{2}\sqrt{\pi} c^\gamma t^{\frac{1-\gamma}{2}} \theta_\Omega \hausdBoundary \\
        &\hspace{2em}+ 2^\frac{n+7}{2}\sqrt{\pi}c^2t^{-\frac{\gamma}{4}}\Lambda^{\frac{\gamma-2}{2}}\varrho^{\frac{\gamma-2}{2}}\theta_\Omega\hausdBoundary  + C_{\Omega,\gamma} \Lambda \Bigg),
    \end{align*}
    where $C_{\Omega,\gamma} = \int_\Omega \d{x}^{2-\gamma}\dd{x}$ and $\varrho=\sup_{x\in\Omega}\d{x}$.
\end{corollary}
\begin{proof}
    By \cite[Lemma 5]{MR2214584} we have 
    \begin{align*}
        u(x;t) \geq 1 - 2^{\frac{n+2}{2}} e^{-\frac{\d{x}^2}{8t}}.
    \end{align*}
    Therefore, Corollary \ref{cor:boundednessOfTubularRegions} gives
    \begin{align*}
        \int_\Omega\big(1-u(x;t)\big) \dd{x} &\leq 2^{\frac{n+2}{2}} \int_\Omega e^{-\frac{\d{x}^2}{8t}}\dd{x} \\
        &= 2^{\frac{n+4}{2}}\int_\Omega \int_0^\infty \1_{\{\d{x}<s\sqrt{8t}\}} se^{-s^2}\dd{s}\dd{x} \\
        &\leq  2^{\frac{n+4}{2}}\sqrt{8t} \theta_\Omega\hausdBoundary \int_0^\infty s^2e^{-s^2}\dd{s} \\
        &= 2^{\frac{n+3}{2}}\sqrt{\pi t} \theta_\Omega\hausdBoundary.
    \end{align*}
    Hence, applying Theorem \ref{thm:weightedIntegralOverHeatContent},
    \begin{align*}
        \int_\Omega u(x;t)^2{\d{x}}^{-\gamma}\dd{x} &\leq 4c^\gamma \left(1-c^2(1-\gamma/2)^2\right)^{-1} t^{\frac{1-\gamma}{2}}2^\frac{n+3}{2}\sqrt{\pi}\theta_\Omega\hausdBoundary \\
        &\hspace{3em}+ 4c^2\left(1-c^2(1-\gamma/2)^2\right)^{-1}\Lambda^\frac{\gamma-2}{2}\abs{\Omega}^\frac{\gamma-2}{2}t^{-\frac{\gamma}{4}}\Big(2^\frac{n+3}{2}\sqrt{\pi}\theta_\Omega\hausdBoundary\Big)^\frac{4-\gamma}{2} \\
        &\hspace{3em}+ \left(1-c^2(1-\gamma/2)^2\right)^{-1} \Lambda \int_\Omega \d{x}^{2-\gamma}\dd{x}.
    \end{align*}
    Note that, from Corollary \ref{cor:boundednessOfTubularRegions} with $\varrho = \sup_{x\in\Omega}\d{x}$, 
    \begin{align}
        \abs{\Omega} \leq \varrho \theta_\Omega \hausdBoundary,
    \end{align}
    and therefore the claim follows by using $2<\gamma<3$.

\end{proof}

Using the preceding theorem and corollary, we find in the following proposition the relevant bound for the first integral in \eqref{eq:errorDecompNearPart2}. To follow the notation in that expression, we state the following proposition again with the parameter $\alpha$.

\begin{proposition}
\label{prop:boundForNearPartKernelOnly}
    Let $c_w$ be the weak Hardy constant of $\Omega$ and let $\alpha \in (\alpha^\ast(c_w),0)$. Then, for all $\alpha^\prime < \alpha$ such that $\alpha^\prime \in (\alpha^\ast(c_w),-2)$, there are corresponding non-negative constants $c$ and $\Lambda$ such that \eqref{eq:weakHardyInequality} is true and, for all $t > 0$, 
    \begin{align*}
        \int_{\{\d{x}\leq \ell_1\}} \d{x}^{\alpha}p_\Omega(x,x;t)\dd{x} &\leq C  t^{-\frac{n}{2}}\ell_1^{\alpha - \alpha^\prime} \left(t^{\frac{1+\alpha^\prime}{2}} \theta_\Omega\hausdBoundary+ t^\frac{\alpha^\prime}{4}\theta_\Omega\hausdBoundary +C_{\Omega,-\alpha^\prime}\Lambda\right),
    \end{align*}
    with an explicit constant $C = C\left(n, \alpha^\prime, c, \Lambda,\varrho\right) > 0$ and $C_{\Omega,-\alpha^\prime}$ as before.
\end{proposition}
\begin{proof}
Since $\alpha^\prime < \alpha$, Corollary \ref{cor:boundForWeightedIntegralHeatContent} with $\gamma = -\alpha^\prime$, implies that there exist $c \geq c_w$ and $\Lambda \geq 0$ such that %
\begin{align*}
    \int_{\{\d{x} \leq \ell_1\}}u^2{\d{x}}^{\alpha}\dd{x} &\leq \ell_1^{\alpha - \alpha^\prime} \int_{\{\d{x} \leq \ell_1\}}u^2{\d{x}}^{\alpha^\prime}\dd{x} \\
    &\leq \ell_1^{\alpha - \alpha^\prime} \int_{\Omega}u^2{\d{x}}^{\alpha^\prime}\dd{x} \\
    &\lesssim_n \ell_1^{\alpha - \alpha^\prime} \Big(1-{c}^2(1+\alpha^\prime/2)^2\Big)^{-1}\Bigg({c}^{-\alpha^\prime}  t^{\frac{1+\alpha^\prime}{2}}  \theta_\Omega\hausdBoundary \\
    & \hspace{7em}+c^2t^\frac{\alpha^\prime}{4} \Lambda^{-\frac{\alpha^\prime+2}{2}}\varrho^{-\frac{\alpha^\prime+2}{2}}\theta_\Omega\hausdBoundary + C_{\Omega,-\alpha^\prime} \Lambda\Bigg).
\end{align*}
Applying the pointwise bound from Proposition \ref{prop:heatTraceBoundByHeatContentBound}, we obtain for the weighted heat trace near the boundary that
\begin{align*}
    \int_{\{\d{x}\leq \ell_1\}} \d{x}^{\alpha}p_\Omega(x,x;t)\dd{x}&\leq \frac{3^{n/2}}{(4\pi t)^{n/2}} \int_{\{\d{x} \leq \ell_1\}}u(x;t/3)^2{\d{x}}^{\alpha}\dd{x} \\
    &\leq C(n,\alpha^\prime,c,\Lambda,\varrho) t^{-\frac{n}{2}}\ell_1^{\alpha - \alpha^\prime} \Bigg(t^{\frac{1+\alpha^\prime}{2}} \theta_\Omega\hausdBoundary \\
    & \hspace{10em}+ t^\frac{\alpha^\prime}{4}\theta_\Omega\hausdBoundary +C_{\Omega,-\alpha^\prime}\Lambda\Bigg). %
\end{align*}
\end{proof}

The moral of the preceding proof is that we have enough room to pull out a decaying factor because we are integrating near the boundary, and then use an asymptotic estimate that holds for the integration on the whole domain.

That the second term in \eqref{eq:errorDecompNearPart2} can be suitably bounded, is shown in the following proposition.
\begin{proposition}
\label{prop:boundForNearPartExpTerm}
    Let $\Omega \subset \R^n$ be an open and bounded set with Lipschitz boundary and let $\alpha \in (-3,0)$. Then, for all $t > 0$,
    \begin{align*}
        \int_{\{\d{x} \leq \ell_1\}}\d{x}^{\alpha} \bigg(1-e^{-\frac{\d{x}^2}{t}}\bigg)\dd{x} \leq  C(\alpha)t^{-1}\ell_1^{3+\alpha}\theta_\Omega\hausdBoundary,
    \end{align*}
    where $C(\alpha)= \max\big(1,\left(3+\alpha\right)^{-1}\big)$.
\end{proposition}

\begin{proof}
Using
\begin{align*}
    1-e^{-\frac{\d{x}^2}{t}} = \int_0^\frac{\d{x}^2}{t} e^{-u}\dd{u} \leq \frac{\d{x}^2}{t}
\end{align*}
as $e^{-u} \leq 1$ for $u\in [0,\infty)$, we bound
\begin{align*}
    \int_{\{{\d{x}} \leq \ell_1\}}{\d{x}}^{\alpha} \bigg(1-e^{-\frac{{\d{x}}^2}{t}}\bigg)\dd{x} \leq \int_{\{{\d{x}} \leq \ell_1\}}{\d{x}}^{\alpha} \frac{{\d{x}}^2}{t}\dd{x}.
\end{align*}
For $\alpha\in [-2,0)$, we apply Corollary \ref{cor:boundednessOfTubularRegions} to obtain
\begin{align*}
    \int_{\{\d{x}\leq \ell_1\}} \d{x}^{2+\alpha}\dd{x} \leq \ell_1^{2+\alpha} \abs{\{\d{x}\leq \ell_1\}} \leq \ell_1^{3+\alpha}\theta_\Omega\hausdBoundary.
\end{align*}
For $\alpha \in (-3,-2)$, we rewrite
\begin{align*}
    \d{x}^{2+\alpha} = -(2+\alpha) \int_{\d{x}}^\infty s^{1+\alpha}\dd{s}.
\end{align*}
Hence,
\begin{align*}
    \int_{\{\d{x}\leq \ell_1\}} \d{x}^{2+\alpha}\dd{x} & = -(2+\alpha)\int_\Omega \int_0^\infty s^{1+\alpha} \1_{\{\d{x} \leq \ell_1\}} \1_{\{\d{x}\leq s\}}\dd{s}\dd{x} \\
    &= - (2+\alpha)\int_\Omega \int_0^{\ell_1}s^{1+\alpha} \1_{\{\d{x}\leq s\}}\dd{s}\dd{x} \\
    &\hspace{4em}- (2+\alpha) \int_{\Omega}\int_{\ell_1}^\infty s^{1+\alpha} \1_{\{\d{x}\leq \ell_1\}}\dd{s}\dd{x} \\
    &= - (2+\alpha) \int_0^{\ell_1}s^{1+\alpha} \abs{\{\d{x}\leq s\}}\dd{s} \\
    &\hspace{4em}- (2+\alpha) \int_{\ell_1}^\infty s^{1+\alpha} \abs{\{\d{x}\leq \ell_1\}}\dd{s}.
\end{align*}
Applying Corollary \ref{cor:boundednessOfTubularRegions} to both integrals, we obtain
\begin{align*}
    \int_0^{\ell_1}s^{1+\alpha} \abs{\{\d{x}\leq s\}}\dd{s} \leq \theta_\Omega\hausdBoundary \int_0^{\ell_1}s^{2+\alpha}\dd{s} = \ell_1^{3+\alpha}\frac{ \theta_\Omega\hausdBoundary}{3+\alpha} ,
\end{align*}
and 
\begin{align*}
    \int_{\ell_1}^\infty s^{1+\alpha} \abs{\{\d{x}\leq \ell_1\}}\dd{s} \leq\theta_\Omega \hausdBoundary \ell_1 \int_{\ell_1}^\infty s^{1+\alpha}\dd{s} = -\ell_1^{3+\alpha}\frac{\theta_\Omega\hausdBoundary}{2+\alpha } .
\end{align*}
Hence,
\begin{align*}
    \int_{\{\d{x}\leq \ell_1\}} \d{x}^{2+\alpha}\dd{x} \leq  \ell_1^{3+\alpha}\left(-\frac{2+\alpha}{3+\alpha} + 1\right) \theta_\Omega\hausdBoundary = \ell_1^{3+\alpha}\frac{\theta_\Omega\hausdBoundary}{3+\alpha} .
\end{align*}
\end{proof}

\subsection{The contribution of the good region}

As discussed in Section \ref{sec:geomConstruction}, we can approximate the heat kernel in the good part by the heat kernel of appropriate half-spaces. Using the approximation of the heat kernel, the following approximation result for the weighted integral holds.

\begin{proposition}
\label{prop:boundForGoodPart}
    Assume that $4\ell_2 < r$, $ 8\ell_2\varepsilon < \ell_1$, and $32\ell_2^2\varepsilon < t$. Then, for any $\alpha \in \R$,
    \begin{align*}
        &\int_{\Omega_{good}}{\d{x}}^{\alpha}\abs{1-(4\pi t)^{n/2}p_\Omega(x,x;t) -e^{-\frac{{\d{x}}^2}{t}}}\dd{x} \\
        &\hspace{4em}\leq   \ell_2 \ell_i^{\alpha} \bigg(2ne^{-\frac{4\ell_2^2}{nt}} +  16et^{-1} e^{-\frac{\ell_1^2}{t}}\ell_2^2\varepsilon(1+4\varepsilon)\bigg)\theta_\Omega\hausdBoundary, %
    \end{align*}
    where $i = 1$ if $\alpha < 0$ and $i= 2$ if $\alpha \geq 0$.
\end{proposition}

\begin{proof}

By Corollary \ref{cor:boundsForGoodPartOfBoundaryKernelOnly} with $\ell = 2\ell_2$, we have the pointwise bounds for $x\in\Omega_{good}$ as follows
\begin{align*}
    e^{-\frac{({\d{x}}+2\rho)^2}{t}}-2ne^{-\frac{4\ell_2^2}{nt}} -e^{-\frac{{\d{x}}^2}{t}} &\leq 1-(4\pi t)^{n/2}p_\Omega(x,x;t) -e^{-\frac{{\d{x}}^2}{t}} \\
    &\hspace{4em}\leq  e^{-\frac{({\d{x}}-2\rho)^2}{t}}+2ne^{-\frac{4\ell_2^2}{nt}} -e^{-\frac{{\d{x}}^2}{t}}, 
\end{align*}
where $\rho = 4\ell_2\varepsilon$ and the upper bound holds as $\d{x} > \ell_1 > 8\ell_2\varepsilon = 2\rho$. This implies
\begin{align*}
    \abs{1-(4\pi t)^{n/2}p_\Omega(x,x;t) -e^{-\frac{{\d{x}}^2}{t}}} \leq 2ne^{-\frac{4\ell_2^2}{nt}} + \max_{\pm}\abs{e^{-\frac{({\d{x}}\pm2\rho)^2}{t}}-e^{-\frac{{\d{x}}^2}{t}}}.
\end{align*}

Therefore, we will consider the two terms
\begin{align*}
   \int_{\Omega_{good}}{\d{x}}^{\alpha}e^{-\frac{4\ell_2^2}{nt}}\dd{x} \qand \int_{\Omega_{good}}{\d{x}}^{\alpha}\abs{e^{-\frac{({\d{x}}\pm2\rho)^2}{t}} -e^{-\frac{{\d{x}}^2}{t}}}\dd{x}
\end{align*}
separately. For the first term, we obtain
\begin{align*}
     \int_{\Omega_{good}}{\d{x}}^{\alpha}e^{-\frac{4\ell_2^2}{nt}}\dd{x} \leq \ell_i^{\alpha} \abs{\Omega_{good}} e^{-\frac{4\ell_2^2}{nt}} \leq  \ell_2 \ell_i^{\alpha} \theta_\Omega\hausdBoundary e^{-\frac{4\ell_2^2}{nt}}, %
\end{align*}
where we used
\begin{align*}
    \abs{\Omega_{good}} \leq \abs{\{x\in\Omega \mid {\d{x}} \leq \ell_2\}} \leq  \ell_2 \theta_\Omega\hausdBoundary
\end{align*}
by Corollary \ref{cor:boundednessOfTubularRegions}. 

For the second term, note that the first exponential is a small perturbation in $\rho$ of the second exponential. The goal is to show that the perturbation is also small for the weighted integral. The difference of the exponential terms can be written as
\begin{align*}
    \abs{e^{-\frac{({\d{x}}\pm2\rho)^2}{t}} -e^{-\frac{{\d{x}}^2}{t}}} = e^{-\frac{{\d{x}}^2}{t}}\abs{e^{\frac{\mp4\rho {\d{x}} - 4\rho^2}{t}} -1} %
    \leq e\cdot e^{-\frac{{\d{x}}^2}{t}} \frac{4\rho\ell_2 + 4\rho^2}{t},
\end{align*}
where the inequality $\abs{e^\xi - 1} \leq e\abs{\xi}$ for $\xi \in (-1,1)$ holds as 
\begin{align*}
    \abs{\xi} = \abs{\frac{\mp4\rho {\d{x}} - 4\rho^2}{t}}  \leq \frac{4\rho\ell_2+4\rho^2}{t} = \frac{16\ell_2^2\varepsilon(1+4\varepsilon)}{t} \leq \frac{32\ell_2^2\varepsilon}{t} < 1,
\end{align*}
by assumption. 
Therefore, we get
\begin{align*}
    \int_{\Omega_{good}}{\d{x}}^{\alpha}\abs{e^{-\frac{({\d{x}}\pm2\rho)^2}{t}} -e^{-\frac{{\d{x}}^2}{t}}}\dd{x} \leq 4e\bigg(\frac{\rho\ell_2 + \rho^2}{t}\bigg)\int_{\Omega_{good}}{\d{x}}^{\alpha}e^{-\frac{{\d{x}}^2}{t}}\dd{x} .
\end{align*}
The resulting integral is roughly bounded by
\begin{align*}
    \int_{\Omega_{good}}{\d{x}}^{\alpha}e^{-\frac{{\d{x}}^2}{t}}\dd{x} \leq \ell_i^\alpha \abs{\Omega_{good}} e^{-\frac{\ell_1^2}{t}}\leq \ell_i^{\alpha}\ell_2 e^{-\frac{\ell_1^2}{t}} \theta_\Omega \hausdBoundary,
\end{align*}
so we get 
\begin{align*}
    \int_{\Omega_{good}}{\d{x}}^{\alpha}\abs{e^{-\frac{({\d{x}}\pm2\rho)^2}{t}} -e^{-\frac{{\d{x}}^2}{t}}}\dd{x} \leq 4e t^{-1}\ell_2\ell_i^{\alpha} e^{-\frac{\ell_1^2}{t}}\left(\rho\ell_2 + \rho^2\right)\theta_\Omega \hausdBoundary.
\end{align*}
Thus, we have shown
\begin{align*}
    &\int_{\Omega_{good}}{\d{x}}^{\alpha}\abs{1-(4\pi t)^{n/2}p_\Omega(x,x;t) -e^{-\frac{{\d{x}}^2}{t}}}\dd{x} \\
    &\hspace{4em}\leq  2n\int_{\Omega_{good}}{\d{x}}^{\alpha}e^{-\frac{4\ell_2^2}{nt}}\dd{x} + \int_{\Omega_{good}}{\d{x}}^{\alpha}\abs{e^{-\frac{({\d{x}}\pm2\rho)^2}{t}} -e^{-\frac{{\d{x}}^2}{t}}}\dd{x} \\
    &\hspace{4em}\leq   \ell_2 \ell_i^{\alpha} \bigg(2ne^{-\frac{4\ell_2^2}{nt}} +  4et^{-1} e^{-\frac{\ell_1^2}{t}}\Big(\rho\ell_2 + \rho^2\Big)\bigg)\theta_\Omega\hausdBoundary.
\end{align*} 
\end{proof}

\subsection{The contribution of the bad region}
\label{sec:badRegion}

In this subsection, we prove that the contribution of the bad region of the domain to the weighted heat trace is small. To do so, we rely on the construction of good sets in \cite[Section 7.1]{frank2026uniformboundsneumannheat}. The authors proved that the family $\{\mathcal{G}_{\varepsilon,r}\}_{\varepsilon,r}$ of good sets can be chosen so that, for fixed $\varepsilon > 0$, the measure of the bad points near the boundary is small. In particular, we use the following result, which we cite for convenience:
\begin{lemma}[Lemma 7.3 in \cite{frank2026uniformboundsneumannheat}]
\label{lemma:boundForMeasureOfBadSet}
    Let $\Omega\subset \R^n$ be an open and bounded Lipschitz set. There exists a collection of $(\varepsilon,r)$-good sets $\{G_{\varepsilon,r}\}_{\varepsilon,r}$ such that, for every $\varepsilon \in \Big(0, \frac14\Big]$, the associated sawtooth region $\mathcal{G}_{\varepsilon,r}$ satisfies
    \begin{align*}
        \lim_{r\to 0^+} \limsup_{s\to 0^+} \frac{\abs{\{x\in \Omega \mid \d{x}<s\}\setminus \mathcal{G}_{\varepsilon,r}}}{s\hausdBoundary} = 0.
    \end{align*}
\end{lemma}

Using that lemma and the boundedness of the integrands, we show in the following proposition that the contribution of $\Omega_{bad}$ to the error term \eqref{eq:weightedIntegralErrorTerm} is small.

\begin{proposition}
\label{prop:boundForBadPart}
    There exists a collection of $(\varepsilon,r)$-good sets $\{G_{\varepsilon,r}\}_{\varepsilon,r}$ and a family of nonnegative functions $M_\varepsilon(\ell, r)$, such that for $\varepsilon\in\Big(0,\frac14\Big]$ and all $\alpha \in \R$, we have
    \begin{align*}
        \int_{\Omega_{bad}}  {\d{x}}^{\alpha }\abs{(4\pi t)^{n/2}p_\Omega(x,x;t) - 1+e^{-\frac{{\d{x}}^2}{t}}}\dd{x}  &\leq 3M_{\varepsilon}(\ell_2,r) \ell_i^{\alpha}\ell_2\hausdBoundary, \,\forall t > 0, %
    \end{align*}
    and $\lim_{r\to 0} \limsup_{\ell\to 0} M_\varepsilon(\ell,r) = 0$, where $i=1$ if $\alpha <0$ and $i=2$ if $\alpha \geq 0$.
\end{proposition}
\begin{proof}
    As $\d{x} \in (\ell_1, \ell_2)$ for all $x\in\Omega_{bad}$, we pull a factor of $\ell_1^{\alpha}$, if $\alpha < 0$, or a factor of $\ell_2^\alpha$, if $\alpha \geq 0$, out of the integral. Then, by using Proposition \ref{prop:boundednessOfHeatKernel}, we bound the remaining integrand by $3$ to obtain
\begin{align*}
    &\int_{\Omega_{bad}}  {\d{x}}^{\alpha }\abs{(4\pi t)^{n/2}p_\Omega(x,x;t) - 1+e^{-\frac{{\d{x}}^2}{t}}}\dd{x}  \\
    &\hspace{8em}\leq \ell_i^{\alpha}  \int_{\Omega_{bad}}\abs{(4\pi t)^{n/2}p_\Omega(x,x;t) - 1+e^{-\frac{{\d{x}}^2}{t}}}\dd{x} \\
    &\hspace{8em}\leq 3 \ell_i^{\alpha} \abs{\Omega_{bad} } .
\end{align*}
Thus, using
\begin{align*}
    \abs{\Omega_{bad}} \leq \abs{\{x\in\Omega \mid \d{x} < \ell_2\}\setminus\mathcal{G}_{\varepsilon,r}} = M_\varepsilon(\ell_2,r) \ell_2 \hausdBoundary
\end{align*} with 
\begin{align*}
    M_\varepsilon(\ell,r) = \frac{\abs{\{x\in\Omega \mid \d{x} < \ell\}\setminus \mathcal{G}_{\varepsilon,r}}}{\ell \hausdBoundary},
\end{align*}
the claim follows by choosing $\{G_{\varepsilon,r}\}_{\varepsilon,r}$ according to Lemma \ref{lemma:boundForMeasureOfBadSet}.

\end{proof}

\section{Proof of the main theorems}
\label{sec:proofSection}

Having established the behavior of the main terms in Section \ref{sec:mainContribution} and bounds for the error term \eqref{eq:weightedIntegralErrorTerm} in different parts of the domain $\Omega$ in Section \ref{sec:boundsForErrorTerms}, we are ready to prove the main theorems of this paper, which are stated in Section \ref{sec:introduction}.
We first prove that the second integral in \eqref{eq:decompInMainContrAndError} is of lower order, as this is a key estimate for the main theorems.
\begin{proposition}
\label{prop:errorTermIsSmall}
    Let $\Omega\subset\R^n$ be an open and bounded set with Lipschitz boundary and $g\in L^\infty(\overline\Omega)$. Let $c_w$ be the weak Hardy constant of $\Omega$ and let $\alpha > \alpha^\ast(c_w)$. Then
    \begin{align}
        \int_\Omega g(x){\d{x}}^{\alpha}\bigg((4\pi t)^{n/2}p_\Omega(x,x;t) - 1+e^{-\frac{{\d{x}}^2}{t}}\bigg)\dd{x} = o\Big(t^{\frac{1+\alpha}{2}}\Big), \qas t \to 0. \label{eq:errorTermIsSmall}
    \end{align}
\end{proposition}
\begin{proof}
    
    By the boundedness of $g$, we have
    \begin{align*}
        &\abs{\int_\Omega g(x){\d{x}}^{\alpha}\bigg((4\pi t)^{n/2}p_\Omega(x,x;t) - 1+e^{-\frac{{\d{x}}^2}{t}}\bigg)\dd{x}} \\
        &\hspace{5em}\leq \norm{g}_\infty \int_\Omega {\d{x}}^{\alpha}\abs{(4\pi t)^{n/2}p_\Omega(x,x;t) - 1+e^{-\frac{{\d{x}}^2}{t}}}\dd{x}.
    \end{align*}
    Therefore, we can restrict ourselves to the case $g= 1$ in the following.

    Let $t > 0, \varepsilon\in\Big(0,\frac14\Big]$, $r > 0$ and $\{G_{\varepsilon,r}\}_{\varepsilon,r}$ be the collection of $(\varepsilon,r)$-good sets according to Proposition \ref{prop:boundForBadPart}. Let $\varepsilon_1, \varepsilon_2 > 0$ be such that
    \begin{align}
    \label{eq:assumptionsForParametersSuchThatErrorTermIsSmall}
    \begin{split}
        \varepsilon_1\varepsilon_2 <1, \\ %
        4\frac{\sqrt{t}}{\varepsilon_2} &< r, \\ %
         8\frac{\varepsilon}{\varepsilon_2} &< \varepsilon_1 , \qand \\ %
        32 \frac{\varepsilon}{\varepsilon_2^2} &< 1. %
    \end{split}
    \end{align} 
    Consider the decomposition of $\Omega$ according to Definition \ref{def:decompositionOfDomain} with the sets $\mathcal{G}_{\varepsilon,r}$ associated with $G_{\varepsilon,r}$ and the length scales $\ell_1 = \varepsilon_1\sqrt{t}$ and $\ell_2 = \frac{\sqrt{t}}{\varepsilon_2}$.
    
    We consider the case $\alpha \in (\alpha^\ast(c_w),0)$ first. With the preceding choice of parameters, the assumptions of the Propositions \ref{prop:boundForBulkPartNegAlpha}, \ref{prop:boundForNearPartKernelOnly}, \ref{prop:boundForNearPartExpTerm}, \ref{prop:boundForGoodPart} and \ref{prop:boundForBadPart} are satisfied and the corresponding bounds may be applied to the left-hand side of \eqref{eq:errorTermIsSmall}. Thus, there exist $c > c_w$, $\alpha^\prime < \min(\alpha,-2)$, $\Lambda \geq 0$, $C(\alpha) > 0$ and $C_{\Omega,-\alpha^\prime}>0$ such that, in the bulk,
    \begin{align*}
        \int_{\Omega_{bulk}} {\d{x}}^{\alpha}\abs{(4\pi t)^{n/2}p_\Omega(x,x;t) - 1 +e^{-\frac{{\d{x}}^2}{t}}}\dd{x} &\lesssim_n \ell_2^{\alpha}\sqrt{t} \theta_\Omega \hausdBoundary \\
        &\lesssim_n \varepsilon_2^{-\alpha}  t^{\frac{1+\alpha}{2}}\theta_\Omega \hausdBoundary , 
    \end{align*}
    in the near-boundary region,
    \begin{align*}
        &\int_{\Omega_{near}} {\d{x}}^{\alpha}\abs{(4\pi t)^{n/2}p_\Omega(x,x;t) - 1 +e^{-\frac{{\d{x}}^2}{t}}}\dd{x}  \\
        &\hspace{2em}\lesssim_{n,\alpha^\prime,c,\Lambda,\varrho} \ell_1^{\alpha - \alpha^\prime} \Big(t^{\frac{1+\alpha^\prime}{2}} \theta_\Omega\hausdBoundary + t^\frac{\alpha^\prime}{4}\theta_\Omega\hausdBoundary+C_{\Omega,-\alpha^\prime}\Lambda\Big) +  C(\alpha)t^{-1}\ell_1^{3+\alpha}\theta_\Omega\hausdBoundary \\
        &\hspace{2em} \lesssim_{n,\alpha,\alpha^\prime,c,\Lambda,\varrho}    (\varepsilon_1^{\alpha-\alpha^\prime} + \varepsilon_1^{3+\alpha}) t^{\frac{1+\alpha}{2}}\theta_\Omega\hausdBoundary + \varepsilon_1^{\alpha-\alpha^\prime}t^{\frac{2\alpha-\alpha^\prime}{4}}\theta_\Omega\hausdBoundary +   \varepsilon_1^{\alpha-\alpha^\prime} t^{\frac{\alpha-\alpha^\prime}{2}}C_{\Omega,-\alpha^\prime}\Lambda,
    \end{align*}
    in the good region,
    \begin{align*}
        &\int_{\Omega_{good}}{\d{x}}^{\alpha}\abs{(4\pi t)^{n/2}p_\Omega(x,x;t) - 1+e^{-\frac{{\d{x}}^2}{t}}}\dd{x} \\
        &\hspace{4em}\leq \ell_2 \ell_1^{\alpha} \bigg(2ne^{-\frac{4\ell_2^2}{nt}} +  16et^{-1} e^{-\frac{\ell_1^2}{t}} \ell_2^2\varepsilon\big(1 + 4\varepsilon\big)\bigg)\theta_\Omega\hausdBoundary \\
        &\hspace{4em}\leq \varepsilon_2^{-1} \varepsilon_1^{\alpha} \bigg(2ne^{-\frac{4}{n\varepsilon_2^2}} +  16e\cdot e^{-\varepsilon_1^2}\varepsilon_2^{-2}\varepsilon\big(1 + 4\varepsilon\big)\bigg) t^{\frac{1+\alpha}{2}}\theta_\Omega\hausdBoundary,
    \end{align*}
    and in the bad region,
    \begin{align*}
        \int_{\Omega_{bad}}  {\d{x}}^{\alpha }\abs{(4\pi t)^{n/2}p_\Omega(x,x;t) - 1+e^{-\frac{{\d{x}}^2}{t}}}\dd{x}  &\leq 3M_{\varepsilon}(\ell_2,r) \ell_1^{\alpha}\ell_2\hausdBoundary \\
        &\leq 3M_{\varepsilon}(\sqrt{t}\cdot\varepsilon_2^{-1},r)  \varepsilon_1^\alpha \varepsilon_2^{-1} t^{\frac{1+\alpha}{2}}\hausdBoundary.
    \end{align*}
    Note that $C_{\Omega,-\alpha^\prime} = \int_\Omega \d{x}^{2+\alpha^\prime}\dd{x}$ is finite by Lemma \ref{lem:mainTermAlphaLess1ExtraTerm}.
    Combining the above four bounds yields
    \begin{align*}
        \frac{1}{t^{\frac{1+\alpha}{2}}} &\int_\Omega {\d{x}}^{\alpha}\abs{(4\pi t)^{n/2}p_\Omega(x,x;t) - 1 +e^{-\frac{{\d{x}}^2}{t}}}\dd{x} \\ &\hspace{1em}\lesssim_{n,\alpha, \alpha^\prime, c, \Omega}  \varepsilon_2^{-\alpha} + \varepsilon_1^{\alpha-\alpha^\prime} +  \varepsilon_1^{3+\alpha} + M_{\varepsilon}(\sqrt{t}\cdot \varepsilon_2^{-1},r) \varepsilon_1^{\alpha}\varepsilon_2^{-1}  \\&\hspace{6em}+  \varepsilon_2^{-1} \varepsilon_1^{\alpha} \bigg(e^{-\frac{4}{n\varepsilon_2^2}} + e^{-\varepsilon_1^2}\varepsilon_2^{-2}\varepsilon\big(1 + 4\varepsilon\big)\bigg) + \varepsilon_1^{\alpha-\alpha^\prime}t^{-\frac{\alpha^\prime+2}{4}} + \Lambda \varepsilon_1^{\alpha-\alpha^\prime} t^{-\frac{\alpha^\prime+1}{2}} .
    \end{align*}
    The right-hand side goes to zero as we send $t$, $r$, $\varepsilon$, $\varepsilon_2$, and $\varepsilon_1$ to zero in that order, by using the properties of $M_\varepsilon(\sqrt{t}\cdot \varepsilon_2^{-1},r)$ according to Proposition \ref{prop:boundForBadPart} and that $\alpha^\prime < -2$. By respecting the order, we can ensure that the assumptions in \eqref{eq:assumptionsForParametersSuchThatErrorTermIsSmall} stay satisfied.

    Next, we consider the case $\alpha \geq 0$. With the choice of parameters in \eqref{eq:assumptionsForParametersSuchThatErrorTermIsSmall}, the assumptions of the Propositions \ref{prop:boundForBulkPartPositiveAlpha}, \ref{prop:nearBoundaryRegionIsSmallForPosAlpha}, \ref{prop:boundForGoodPart} and \ref{prop:boundForBadPart} are satisfied and the corresponding bounds may be applied to the left-hand side of \eqref{eq:errorTermIsSmall}. Thus, we have, in the bulk,
    \begin{align*}
        \int_{\Omega_{bulk}} \d{x}^{\alpha}\abs{(4\pi t)^{n/2}p_\Omega(x,x;t) - 1 +e^{-\frac{\d{x}^2}{t}}}\dd{x} &\lesssim_{n,\alpha}  t^\frac{\alpha}{2}(t^{\frac{1}{2}}+\ell_2) e^{-\frac{\ell_2^2}{2nt}} \theta_\Omega \hausdBoundary, \\
        &\lesssim_{n,\alpha}  (1+\varepsilon_2^{-1})e^{-\frac{1}{2n\varepsilon_2^2}} t^{\frac{1+\alpha}{2}} \theta_\Omega \hausdBoundary,
    \end{align*}
    in the near-boundary region,
    \begin{align*}
        \int_{\Omega_{near}} {\d{x}}^{\alpha}\abs{(4\pi t)^{n/2}p_\Omega(x,x;t) - 1 +e^{-\frac{{\d{x}}^2}{t}}}\dd{x}  &\leq 3\ell_1^{\alpha+1} \theta_\Omega \hausdBoundary \\
        &= 3\varepsilon_1^{\alpha + 1} t^{\frac{1+\alpha}{2}}\theta_\Omega \hausdBoundary, 
    \end{align*}
    in the good region,
    \begin{align*}
        &\int_{\Omega_{good}}{\d{x}}^{\alpha}\abs{(4\pi t)^{n/2}p_\Omega(x,x;t) - 1+e^{-\frac{{\d{x}}^2}{t}}}\dd{x} \\
        &\hspace{4em}\leq \ell_2^{\alpha+1} \bigg(2ne^{-\frac{4\ell_2^2}{nt}} +  16et^{-1} e^{-\frac{\ell_1^2}{t}} \ell_2^2\varepsilon\big(1 + 4\varepsilon\big)\bigg)\theta_\Omega\hausdBoundary \\
        &\hspace{4em}\leq  \varepsilon_2^{-\alpha-1} \bigg(2ne^{-\frac{4}{n\varepsilon_2^2}} +  16e\cdot e^{-\varepsilon_1^2}\varepsilon_2^{-2}\varepsilon\big(1 + 4\varepsilon\big)\bigg) t^{\frac{1+\alpha}{2}}\theta_\Omega\hausdBoundary,
    \end{align*}
    and in the bad region,
    \begin{align*}
        \int_{\Omega_{bad}}  {\d{x}}^{\alpha }\abs{(4\pi t)^{n/2}p_\Omega(x,x;t) - 1+e^{-\frac{{\d{x}}^2}{t}}}\dd{x}  &\leq 3M_{\varepsilon}(\ell_2,r) \ell_2^{\alpha+1}\hausdBoundary \\
        &\leq 3M_{\varepsilon}(\sqrt{t}\cdot\varepsilon_2^{-1},r) \varepsilon_2^{-\alpha-1}  t^{\frac{1+\alpha}{2}} \hausdBoundary.
    \end{align*}
    Hence, we obtain
    \begin{align*}
        &\frac{1}{t^{\frac{1+\alpha}{2}}}\int_\Omega {\d{x}}^{\alpha}\abs{(4\pi t)^{n/2}p_\Omega(x,x;t) - 1 +e^{-\frac{{\d{x}}^2}{t}}}\dd{x}  \\ 
        &\hspace{3em}\lesssim_{n,\alpha,\Omega}  (1+\varepsilon_2^{-1}+\varepsilon_2^{-\alpha-1}) e^{-\frac{1}{2n\varepsilon_2^2}} + \varepsilon_1^{\alpha+1}  \\
        &\hspace{3em}\hspace{5em}+\varepsilon_2^{-\alpha-3} \varepsilon\big(1 + 4\varepsilon\big)e^{-\varepsilon_1^2} + M_{\varepsilon}(\sqrt{t}\cdot \varepsilon_2^{-1},r) \varepsilon_2^{-\alpha-1}.
    \end{align*}
    As before, sending $t, r, \varepsilon, \varepsilon_2$ and $\varepsilon_1$ to zero, in that order, the right-hand side goes to zero, which proves the claim.
\end{proof}

\begin{remark}
    The order in which the different parameters are sent to zero is important. Heuristically, this corresponds to the following:
    \begin{itemize}
        \item $t\to 0$: Localize the length scales such that all relevant phenomena are localized.
        \item $r \to 0$: Shrink the size of the sawtooth regions.
        \item $\varepsilon\to 0$: Shrink the angle of the cones enclosing the boundary around each point.
        \item $\varepsilon_2\to 0$: Shrink the bulk region.
        \item $\varepsilon_1\to 0$: Shrink the near-boundary region.
    \end{itemize}
\end{remark}

\begin{proof}[Proof of Theorem \ref{thm:mainThmWithG}]
    The existence of $Z_{g,\alpha}(t)$ for all $t>0$ follows from Corollary \ref{cor:boundForWeightedIntegralHeatContent} and Proposition \ref{prop:heatTraceBoundByHeatContentBound}.
    
    For the asymptotics, we employ the decomposition \eqref{eq:decompInMainContrAndError}, namely
    \begin{align*}
        (4\pi t)^{n/2}\int_\Omega g(x) {\d{x}}^{\alpha}p_\Omega(x,x;t)\dd{x} &= \int_\Omega g(x) {\d{x}}^{\alpha} \bigg(1-e^{-\frac{{\d{x}}^2}{t}}\bigg)\dd{x} \\
    &\hspace{2em}+ \int_\Omega g(x){\d{x}}^{\alpha}\bigg((4\pi t)^{n/2}p_\Omega(x,x;t) - 1+e^{-\frac{{\d{x}}^2}{t}}\bigg) \dd{x}. %
    \end{align*}
    The second term is $o\Big(t^{\frac{1+\alpha}{2}}\Big)$ by Proposition \ref{prop:errorTermIsSmall}.
    
    If $\alpha <-1$, the first term is, by applying Theorem \ref{thm:mainTermAlphaLarger1}, 
    \begin{align*}
        \int_\Omega  g(x) {\d{x}}^{\alpha} \bigg(1-e^{-\frac{{\d{x}}^2}{t}}\bigg)\dd{x} = -t^{\frac{1+\alpha}{2}} \frac{1}{2}\Gamma\bigg(\frac{1+\alpha}{2}\bigg) \int_{\partial\Omega} g(x) \dd\mathcal{H}^{n-1}(x) + o\Big(t^{\frac{1+\alpha}{2}}\Big), \qas t\to 0.
    \end{align*}
    
    If $\alpha > -1$, we obtain by Theorem \ref{thm:mainTermAlphaLess1}
    \begin{align}
        &\int_\Omega  g(x) {\d{x}}^{\alpha} \bigg(1-e^{-\frac{{\d{x}}^2}{t}}\bigg)\dd{x} = \int_\Omega  g(x) {\d{x}}^{\alpha}\dd{x} - \int_\Omega  g(x) {\d{x}}^{\alpha} e^{-\frac{{\d{x}}^2}{t}}\dd{x} \nonumber \\
        &\hspace{3em}= \int_\Omega  g(x) {\d{x}}^{\alpha}\dd{x} -t^{\frac{1+\alpha}{2}} \frac{1}{2}\Gamma\bigg(\frac{1+\alpha}{2}\bigg) \int_{\partial\Omega} g(x) \dd\mathcal{H}^{n-1}(x) + o\Big(t^{\frac{1+\alpha}{2}}\Big), \qas t\to 0. \label{eq:proofOfMainThmForAlphaLarger1Decomp}
    \end{align}
    Furthermore, by Lemma \ref{lem:mainTermAlphaLess1ExtraTerm}, we see that the first term in \eqref{eq:proofOfMainThmForAlphaLarger1Decomp} is finite.
    
    Lastly, if $\alpha = -1$, by Theorem \ref{thm:mainTermAlphaEquals1},
    \begin{align*}
        \int_\Omega g(x)\d{x}^{-1}\bigg(1-e^{-\frac{\d{x}^2}{t}}\bigg)\dd{x} =  -\frac{\log{t}}{2} \int_{\partial\Omega} g(x) \dd\mathcal{H}^{n-1}(x) + o\big(\abs{\log t}\big), \qas t\to 0.
    \end{align*}
    These are precisely the asymptotic terms claimed in the theorem.
\end{proof}

\begin{proof}[Proof of Corollary \ref{cor:introRelaxationOng}]
    The strategy is to decompose $g=g_1+g_2$ in the integral early on, that is, we write $Z_{g,\alpha} = Z_{g_1,\alpha} + Z_{g_2,\alpha}$ and apply Theorem \ref{thm:mainThmWithG} to $Z_{g_1,\alpha}$. To complete the proof, we need to estimate $Z_{g_2,\alpha}$.
    
    Let $\eta > 0$ such that $\mathrm{supp}(g_2) \subset \{x\in \Omega \mid \d{x} > \eta\}$. The idea now is that, as $t\to 0$, the support of $g_2$ is contained in the bulk and the integrals can be bounded by similar estimates as we have used before.
    
    If $\alpha \leq -1$, Proposition \ref{prop:boundednessOfHeatKernel} gives
    \begin{align*}
        \abs{Z_{g_2,\alpha}(t)} \leq  \int_{\{\d{x}>\eta\}} \abs{g_2(x)} \d{x}^\alpha \dd{x} \leq \norm{g_2}_1 \eta^\alpha, \text{ for all } t\in (0,\infty),
    \end{align*}
    which is constant in $t$, hence $o\Big(t^{\frac{1+\alpha}{2}}+\abs{\log t}\Big)$. Therefore, the claim follows in this case by noting that
    \begin{align*}
        \int_{\partial\Omega} g_1(x)\dd{\mathcal{H}^{n-1}(x)} = \int_{\partial\Omega} g(x)\dd{\mathcal{H}^{n-1}(x)}.
    \end{align*}

    If $\alpha > -1$, we split the integral into
    \begin{align*}
        Z_{g_2,\alpha}(t) = \int_\Omega g_2(x) \d{x}^\alpha \dd{x} + \int_\Omega g_2(x) \d{x}^\alpha \big((4\pi t)^\frac{n}{2} p_\Omega(x,x;t) - 1\big) \dd{x}.
    \end{align*}
    The first integral is finite as 
    \begin{align*}
        \abs{\int_\Omega g_2(x) \d{x}^\alpha \dd{x}} \leq \int_{\{\d{x}> \eta\}} \abs{g_2(x)} \d{x}^\alpha \dd{x} \\
        \leq
        \begin{cases}
            \norm{g_2}_1 \eta^\alpha, &\qfor \alpha < 0 , \\
            \norm{g_2}_1 \varrho^\alpha, &\qfor \alpha \geq 0 ,
        \end{cases}
    \end{align*}
    with $\varrho = \sup_{x\in\Omega} \d{x}$ as before. For the second integral we apply Proposition \ref{prop:boundBulkKernel} and Hölder's inequality to obtain
    \begin{align*}
        \abs{\int_\Omega g_2(x) \d{x}^\alpha \Big((4\pi t)^\frac{n}{2} p_\Omega(x,x;t) - 1\Big) \dd{x}} &\leq 2n\int_{\{\d{x}> \eta\}}\abs{g_2(x)}\d{x}^\alpha e^{-\frac{\d{x}^2}{nt}} \dd{x} \\
        &\leq 2n\norm{g_2}_1 \sup_{x\in\{\d{x}>\eta\}} \abs{\d{x}^\alpha e^{-\frac{\d{x}^2}{nt}}}.
    \end{align*}
    To analyze the supremum, denote $f(s) = s^\alpha e^{-\frac{s^2}{nt}}$. If $\alpha > 0$, $f$ attains its supremum over $[0,\infty)$ at $s^\ast = \sqrt{\frac{\alpha n t}{2}}$. By choosing $t$ small enough, we achieve $s^\ast < \eta$ and we can assume that the supremum over the relevant range is attained at $\eta$. If $\alpha \leq 0$, the supremum is always attained at $\eta$. Thus, we see
    \begin{align*}
        \abs{\int_\Omega g_2(x) \d{x}^\alpha \Big((4\pi t)^\frac{n}{2} p_\Omega(x,x;t) - 1\Big) \dd{x}} \leq 2n\norm{g_2}_1 \eta^\alpha e^{-\frac{\eta^2}{nt}},
    \end{align*}
    which is exponentially small as $t\to 0$, hence it is $o(t^k)$ for any $k> 0$. This shows the claim as
    \begin{align*}
        Z_{g,\alpha}(t) &= \int_{\Omega} g_1(x)\d{x}^{\alpha}\dd{x}- t^{\frac{1+\alpha}{2}} \frac{1}{2}\Gamma\bigg(\frac{1+\alpha}{2}\bigg) \int_{\partial\Omega}g_1\dd{\mathcal{H}^{n-1}} + \int_{\Omega} g_2(x)\d{x}^{\alpha}\dd{x} + o\Big(t^\frac{1+\alpha}{2}\Big) \\
        &= \int_{\Omega} g(x)\d{x}^{\alpha}\dd{x} - t^{\frac{1+\alpha}{2}} \frac{1}{2}\Gamma\bigg(\frac{1+\alpha}{2}\bigg) \int_{\partial\Omega}g\dd{\mathcal{H}^{n-1}} + o\Big(t^\frac{1+\alpha}{2}\Big),
    \end{align*}
    as $t\to 0$.
\end{proof}

\begin{proof}[Proof of Theorem \ref{thm:mainThmWithIndicatorFunc}]
    Again, the existence of $Z_{A,\alpha}(t)$ for all $t>0$ follows from Corollary \ref{cor:boundForWeightedIntegralHeatContent} and Proposition \ref{prop:heatTraceBoundByHeatContentBound}. For the asymptotics, we use the decomposition 
    \begin{align*}
        (4\pi t)^{\frac{n}{2}}\int_{\Omega\cap A}\d{x}^{\alpha}p_\Omega(x,x;t)\dd{x} &= \int_{\Omega\cap A} \d{x}^{\alpha} \bigg(1-e^{-\frac{\d{x}^2}{t}}\bigg)\dd{x} \\
    &\hspace{2em}+ \int_{\Omega\cap A} \d{x}^{\alpha}\bigg((4\pi t)^{n/2}p_\Omega(x,x;t) - 1+e^{-\frac{\d{x}^2}{t}}\bigg) \dd{x}.
    \end{align*}
    The second term is $o\Big(t^{\frac{1+\alpha}{2}}\Big)$ due to Proposition \ref{prop:errorTermIsSmall} with $g=\1_A$. For the first term, we apply, depending on the value of $\alpha$, the Theorems \ref{thm:mainTermAlphaLarger1}, \ref{thm:mainTermAlphaLess1} or \ref{thm:mainTermAlphaEquals1} with $g = \1_A$ and observe that
    \begin{align*}
        \int_{\partial\Omega} \1_A \dd\mathcal{H}^{n-1} = \hausdBoundarySet{\partial\Omega \cap A}.
    \end{align*}
    Furthermore, for $\alpha > -1$, the term
    \begin{align*}
        \int_{\Omega} \1_A \d{x}^{\alpha} \dd{x} = \int_{\Omega\cap A} \d{x}^{\alpha} \dd{x}
    \end{align*}
    is finite by Lemma \ref{lem:mainTermAlphaLess1ExtraTerm}.
\end{proof}

\begin{proof}[Proof of Corollary \ref{cor:weightedRieszMean}]
    Assume first that $g$ is nonnegative. Let
    \begin{align*}
        R_{g,\alpha}^\gamma(\lambda) &= \int_\Omega g(x) \d{x}^\alpha(-\Delta-\lambda)^\gamma_- (x,x)\dd{x} \\
        &= \sum_{k=1}^\infty (\lambda_k-\lambda)_-^\gamma C_k(\Omega), 
    \end{align*}
    where $C_k(\Omega) = \int_\Omega g(x) \d{x}^\alpha  u_k(x)^2\dd{x}$. Note that, with \eqref{eq:heatKernelRepInBasis},
    \begin{align*}
        (4\pi t)^{-\frac{n}{2}}Z_{g,\alpha}(t) = \int_\Omega g(x)\d{x}^\alpha p_\Omega(x,x;t)  \dd{x} = \sum_{k=1}^\infty e^{-t\lambda_k} \int_\Omega g(x)\d{x}^\alpha u_k(x)^2\dd{x},
    \end{align*}
    by Tonelli's theorem. Thus, $C_k(\Omega) < \infty$ for all $k\in \N$ since $Z_{g,\alpha}(t) < \infty$ for all $t>0$, which follows from Theorem \ref{thm:mainThmWithG}. Hence, $R_{g,\alpha}^\gamma(\lambda)$ is finite for all $\lambda > 0$. Furthermore, it is monotone increasing, as $g$ is non-negative. 
    First, we show that 
    \begin{align*}
        \int_0^\infty e^{-t\lambda} \dd R_{g,\alpha}^\gamma(\lambda) = t^{-\gamma} \Gamma(\gamma+1) \int_\Omega g(x)\d{x}^\alpha p_\Omega(x,x;t) \dd{x}.
    \end{align*}
    If $\gamma = 0$, we have
    \begin{align*}
        R_{g,\alpha}^0(\lambda) = \sum_{k=1}^\infty \1_{(\lambda_k,\infty)}(\lambda)C_k(\Omega),
    \end{align*}
    and therefore,
    \begin{align*}
        \int_0^\infty e^{-t \lambda} \dd R_{g,\alpha}^0(\lambda) &= \int_0^\infty e^{-t\lambda} \sum_{k=1}^\infty C_k(\Omega) \dd{\delta_{\lambda_k}(\lambda)} \\
        &= \sum_{k=1}^\infty e^{-t\lambda_k} C_k(\Omega) \\
        &= \int_\Omega g(x)\d{x}^\alpha p_\Omega(x,x;t)  \dd{x}.
    \end{align*}
    If $\gamma > 0$, then
    \begin{align*}
        \dd R_{g,\alpha}^\gamma(\lambda) 
        = \gamma\sum_{k=1}^\infty C_k(\Omega) (\lambda-\lambda_k)^{\gamma-1} \1_{\{\lambda > \lambda_k\}} \dd\lambda,
    \end{align*}
    for almost every $\lambda>0$. Note that for $0<\gamma<1$, the density has integrable singularities at the eigenvalues.
    The claim now follows by the straightforward calculation
    \begin{align*}
        \int_0^\infty e^{-t\lambda}\dd R_{g,\alpha}^\gamma(\lambda) &= \gamma\sum_{k=1}^\infty C_k(\Omega) \int_{\lambda_k}^\infty e^{-t\lambda}(\lambda-\lambda_k)^{\gamma-1}\dd\lambda \\
        &= \gamma\sum_{k=1}^\infty  C_k(\Omega)\int_0^\infty e^{-t(u+\lambda_k)}u^{\gamma-1}\dd{u} \\
        &= \gamma\sum_{k=1}^\infty  C_k(\Omega)e^{-t\lambda_k} \int_0^\infty e^{-tu}u^{\gamma-1}\dd{u} \\
        &= \sum_{k=1}^\infty  C_k(\Omega)e^{-t\lambda_k} t^{-\gamma}\Gamma(\gamma+1) \\
        &= t^{-\gamma}\Gamma(\gamma+1) \int_\Omega g(x)\d{x}^\alpha p_\Omega(x,x;t) \dd{x}.
    \end{align*}
    If $\alpha < -1$, we obtain by Theorem \ref{thm:mainThmWithG},
    \begin{align*}
        \int_0^\infty e^{-t\lambda} \dd R_{g,\alpha}^\gamma(\lambda) = -t^{\frac{1+\alpha-n}{2}-\gamma}\frac{\Gamma(\gamma+1) \Gamma(\frac{1+\alpha}{2})}{2(4\pi)^{\frac{n}{2}}} \int_{\partial\Omega}g\dd\mathcal{H}^{n-1} + o\Big(t^{\frac{1+\alpha-n}{2}-\gamma}\Big), \qas t\to 0.
    \end{align*}
    Hence, by Karamata's Tauberian Theorem, see \cite[Chapter IV, Theorem 8.1]{MR2073637},
    \begin{align*}
        R_{g,\alpha}^\gamma(\lambda) = -\lambda^{\gamma -\frac{1+\alpha-n}{2}}\frac{\Gamma(\gamma+1) \Gamma(\frac{1+\alpha}{2})}{2(4\pi)^{\frac{n}{2}}\Gamma(\gamma-\frac{1+\alpha-n}{2}+1)} \int_{\partial\Omega}g\dd\mathcal{H}^{n-1}  + o\Big(\lambda^{\gamma -\frac{1+\alpha-n}{2}}\Big), \qas \lambda \to \infty.
    \end{align*}
    If $\alpha > -1$, we obtain again by Theorem \ref{thm:mainThmWithG},
    \begin{align*}
        \int_0^\infty e^{-t\lambda} \dd R_{g,\alpha}^\gamma(\lambda) = t^{-\frac{n}{2}-\gamma}\frac{\Gamma(\gamma+1)}{(4\pi)^{\frac{n}{2}}} \int_{\Omega}g(x)\d{x}^\alpha\dd{x} + o\Big(t^{-\frac{n}{2}-\gamma}\Big), \qas t\to 0,
    \end{align*}
    and therefore, by Karamata's Tauberian Theorem,
    \begin{align*}
        R_{g,\alpha}^\gamma(\lambda) = \lambda^{\gamma +\frac{n}{2}}\frac{\Gamma(\gamma+1)}{(4\pi)^{\frac{n}{2}}\Gamma(\gamma+\frac{n}{2}+1)} \int_{\Omega}g(x)\d{x}^\alpha\dd{x}  + o\Big(\lambda^{\gamma+ \frac{n}{2}}\Big), \qas \lambda \to \infty.
    \end{align*}
    Finally, if $\alpha = -1$, we have by Theorem \ref{thm:mainThmWithG},
    \begin{align*}
        \int_0^\infty e^{-t\lambda}\dd{R_{g,-1}^\gamma(\lambda)} = -t^{-\frac{n}{2}-\gamma}\log(t)\cdot \frac{\Gamma(\gamma+1)}{2(4\pi)^\frac{n}{2}} \int_{\partial\Omega}g\dd\mathcal{H}^{n-1} + o\Big(\abs{t^{-\frac{n}{2}-\gamma} \log t}\Big), \qas t\to 0.
    \end{align*}
    We use again Karamata's Tauberian Theorem, \cite[Chapter IV, Theorem 8.1]{MR2073637}, now with slowly varying function $L(s) = \log s$, so that $L(1/t)=-\log{t}$. Hence,
    \begin{align*}
        R_{g,-1}^\gamma(\lambda) = \lambda^{\gamma+\frac{n}{2}}\log\lambda \cdot \frac{\Gamma(\gamma+1)}{2(4\pi)^\frac{n}{2} \Gamma(\gamma+\frac{n}{2}+1)} \int_{\partial\Omega} g\dd\mathcal{H}^{n-1} + o\Big(\lambda^{\gamma+\frac{n}{2}}\log\lambda\Big), \qas \lambda \to \infty.
    \end{align*}

    For general $g$ satisfying the assumptions of the statement, we can write $ g= g_+ - g_-$ with nonnegative functions $g_+ = \max(g,0)$ and $g_- = \max(-g,0)$, and both of them also satisfy the assumptions of Theorem \ref{thm:mainThmWithG}. Thus, we can apply the preceding calculation separately and use the linearity of all terms with respect to $g$ to deduce the corollary.
\end{proof}

\begin{proof}[Proof of Corollary \ref{cor:ConvergenceOfMeasures}]
    The corollary immediately follows from Theorem \ref{thm:mainThmWithG} and the definition of weak convergence of measures.
\end{proof}

\printbibliography

\end{document}